\documentclass{article}
\usepackage[english]{babel}
\usepackage[utf8]{inputenc}
\usepackage{johd}
\usepackage{mathptmx}
\usepackage{amsmath}
\usepackage{mathtools}
\usepackage{verbatim}
\usepackage{appendix}

\numberwithin{equation}{section}

\usepackage{amsmath,amsthm}       
\usepackage{mathrsfs}      
\usepackage{amssymb}       
\usepackage{amsfonts}      

\usepackage{cancel}
\usepackage{indentfirst}

\usepackage{graphicx}
\usepackage{subfigure}
\usepackage{wrapfig}  
\usepackage{multicol}
\usepackage{tikz}     

\usepackage[framemethod=TikZ]{mdframed}
\usepackage{color}
\usepackage{authblk}
\usepackage{textcomp}

\newtheorem{thm}{Theorem}[section]
\newtheorem{lemma}[thm]{Lemma}
\newtheorem{definition}[thm]{Definition}

\newtheorem{remark}[thm]{Remark}
\newtheorem{prop}[thm]{Proposition}

\newtheorem{corollary}[thm]{Corollary}

\author[1]{Zihao Gu}
\author[1]{Yiqing Lin}
\author[,1]{Kun Xu\footnote{Corresponding author. Email address: \url{1949101x_k@sjtu.edu.cn} (K. Xu).}}
\affil[1]{\small{School of Mathematical Sciences, Shanghai Jiao Tong University, 200240 Shanghai, China.}}

\title{Reflected BSDE driven by a marked point process with a convex/concave  generator}

\date{October 25, 2023}

\begin{document}

\maketitle

\begin{abstract}
In this paper, a class of reflected backward stochastic differential equations (RBSDE) driven  by a marked point process (MPP) with a convex/concave generator is studied. Based on fixed point argument, $\theta$-method and truncation technique, the well-posedness of this kind of RBSDE with unbounded terminal condition and obstacle is investigated. Besides, we present an application  on the pricing of American options via utility maximization, which is solved by constructing an RBSDE with a convex generator. 
\end{abstract}


\tableofcontents

\section{Introduction}
Backward stochastic differential equations (BSDEs) were first introduced by Bismut in 1973 \cite{bismut1973conjugate} as equation for the adjoint process in the stochastic version of Pontryagin maximum principle. Pardoux and Peng \cite{pardoux1990adapted} have generalized the existence and uniqueness result in the case when the driver is Lipschitz continuous and the terminal value $\xi$ is square integrable.
Due to a wide range of applications, the research of BSDEs with a quadratic generator (quadratic BSDEs) has attracted many people's attention. In particular, quadratic BSDEs for a bounded terminal value $\xi$ was studied by Kobylanski \cite{kobylanski2000backward} via an approximation procedure of the driver. Thereafter, the result was generated by Briand and Hu \cite{briand2006bsde, briand2008quadratic} for unbounded terminal value $\xi$ of some suitable exponential moments. In contrast, Tevzadze \cite{tevzadze2008solvability} proposed a fundamentally different approach by means of a fixed point argument. Fairly large number of applications of quadratic BSDEs can be found in literature, for instance, on PDEs \cite{delbaen2015uniqueness}, risk sensitive control problems \cite{ hu2011some}, indifference pricing in  incomplete market \cite{hu2005utility,  morlais2009quadratic} and etc.

In order to solve the obstacle problem in partial differential equations,  reflected backward stochastic differential equations (RBSDEs) was firstly introduced in El Karoui et al. \cite{Karoui1997ReflectedSO} in the Brownian framework,  with generator $f$, terminal condition $\xi$ and obstacle process $L$:
$$
L_t \leq Y_t=\xi+\int_t^T f\left(s, Y_s, Z_s\right) d s+K_T-K_t-\int_t^T Z_s d B_s, \quad t \in[0, T],
$$
where the solution $(Y, Z, K)$ satisfies the so-called flat-off condition (or, Skorokhod condition):
$$
\int_0^T\left(Y_t-L_t\right) d K_t=0,
$$
where $K$ is an increasing process. El Karoui et al. proved the solvability of RBSDE with Lipschitz $f$ and square integrable terminal $\xi$. Matoussi \cite{matoussi1997reflected} investigated RBSDE with square integrable terminal conditions when generators are continuous and linearly increasing with respect to variables $Y$ and $Z$. After that, Kobylanski et al. \cite{kobylanski2002reflected} investigated RBSDEs with bounded terminal conditions and bounded obstacles when the generator $f$ has superlinear growth in $Y$ and quadratic growth in $Z$, then the result was generalized by Lepeltier and Xu \cite{lepeltier2007reflected} who constructed the existence of a solution with unbounded terminal values, but still with a bounded obstacle. Bayrakstar and Yao \cite{bayraktar2012quadratic} studied the well-posedness of quadratic RBSDE under unbounded terminal and unbounded obstacles, with the help of $\theta$-method. For more 
studies related to RBSDEs, we refer the readers to  Essaky and Hassani \cite{essaky2011general}, Ren and Xia \cite{ren2006generalized}, Jia and Xu \cite{jia2008construction} and so on.

Moreover, motivated by the probabilistic interpretation of viscosity solution of semilinear integral-partial differential equations, BSDE with jumps (BSDEJ) with Lipschitz generator was studied by Barles, Buckdahn and Pardoux \cite{barles1997backward}. Meanwhile,  Li and Tang \cite{Tang_1994}  obtained the well-posedness for Lipschitz  BSDEJs via a fixed point approach similar to that used in  Pardoux and Peng \cite{pardoux1990adapted}, see also Papapantoleon, Possama\"i, and Saplao \cite{Papapantoleon_2018} for a more general framework. In particular, a class of BSDEs driven by a random measure associated with a marked point process as follows is investigated by many researchers. 
\begin{equation}
Y_t= \xi + \int_t^T f\left(t, Y_s, U_s\right)dA_s-\int_t^T \int_E U_s(e) q(d sd e).
\end{equation}
Here $q$ is a compensated integer random measure corresponding to some marked point process $(T_n,\zeta_n)_{n\ge 0}$, and $A$ is  the dual predictable projection of the event counting process related to the marked point process, which is a continuous and increasing process. The well-posedness of BSDEs driven by general marked point processes were investigated in Confortola \& Fuhrman \cite{Confortola2013} for the weighted-$L^2$ solution, Becherer \cite{Becherer_2006} and Confortola \& Fuhrman \cite{Confortola_2014} for the $L^2$ case,  Confortola,  Fuhrman \& Jacod \cite{Confortola2016} for the $L^1$ case and Confortola \cite{Confortola_2018} for the $L^p$ case. A more general BSDE with both Brownian motion diffusion term and a very general marked point process, which is non-explosive and has totally inaccessible jumps was studied in Foresta \cite{foresta2021optimal}.

In addition, in order to solve the utility maximization problem with jumps, BSDEJ driven by quadratic coefficient was studied by Morlais \cite{morlais2010new}, see also Kazi-Tani, Possama\"i, and Zhou \cite{kazi2015quadratic} for a fixed point approach. Besides, with the help of stability of quadratic semimartingales, Barrieu and El Karoui \cite{barrieu2013monotone} proposed a quadratic structure condition and showed the existence of a solution under the unbounded terminal condition in a continuous setup. This structure condition was generalized to a so-called quadratic-exponential growth condition 
in accordance with the jump setting in  Ngoupeyou \cite{ngoupeyou2010optimisation}, Jeanblanc, Matoussi \& Ngoupeyou \cite{jeanblanc2012robust}, El Karoui, Matoussi \& Ngoupeyou \cite{karoui2016quadratic}, and Fujii and Takahashi \cite{fujii2018quadratic}. However, those results for unbounded terminals only provided existence without uniqueness. Recently, with the help of $\theta$-method, Kaaka\"i, Matoussi and Tamtalini \cite{kaakai2022utility} obtained the well-posedness of a special class of quadratic exponential BSDEJs with unbounded terminal conditions aroused in a robust utility maximization problem, under several special structural conditions.

In this paper, we consider the following type of RBSDEs driven by a marked point process, which generalized the results of Brownian driven BSDEs studied in \cite{bayraktar2012quadratic} and Poisson driven BSDEJs attributed to e.g.  Hamad\'ene and Ouknine \cite{hamadene2003reflected,  hamadene2016reflected}. 
\begin{equation}
\left\{\begin{array}{l}
Y_t=\xi+\int_t^T f\left(s,Y_s, U_s\right) d A_s +\int_t^T d K_s-\int_t^T \int_E U_s(e) q(d s d e), \quad 0 \leq t \leq T, \\
Y_t \geq L_t, \quad 0 \leq t \leq T, \\
\int_0^T\left(Y_{s^{-}}-L_{s^{-}}\right) d K_s=0.
\end{array}\right.
\end{equation}
Some related studies on  doubly reflected RBSDEJs can be found in Cr\'epey and Matoussi \cite{crepey2008reflected}. Moreover, RBSDEJs driven by Lévy process considered in for instance,  Ren and El Otmani \cite{ren2010generalized}, Ren and Hu \cite{ren2007reflected} and El Otmani \cite{el2009reflected} are also enlightening. 
Compared with the jump setting in e.g. Matoussi and Salhi \cite{matoussi2020generalized}, the  process $A$ is not necessarily absolutely continuous with respect to the Lebesgue measure. This type of RBSDEs has been investigated in Foresta \cite{foresta2021optimal}. The author established the well-posedness with Lipschitz drivers with the help of a fixed point argument. Inspired by this work and motivated by the pricing of American options via utility maximization in a jump market, in this paper, a quadratic-exponential growth condition on the generator $f$ is assumed in accordance with the marked point process. In particular, compared with the growth condition in \cite{bayraktar2012quadratic}, we allow more flexibility for the operator $f(t,0,0)$.  Besides, with the help of the results for classical 
Brownian driven BSDEs in \cite{bayraktar2012quadratic}, together with our conclusions, we are able to build the well-posedness for reflected quadratic BSDEs driven by both a marked point process and a Brownian motion.

The outline of the paper is as follows. In section \ref{section pre}, we  introduce some basic notations, assumptions and existing results on reflected BSDEs driven by marked point process. The reminder is organized as four parts: section \ref{section 3} is devoted to give the comparison theorem for quadratic exponential RBSDEs with a convex / concave generator; section \ref{section priori} provides a priori estimates  in quadratic exponential RBSDEs; in section \ref{section existence for bdd} and \ref{section existence for unbdd}, we establish the well-posedness of the quadratic reflected BSDEs driven by a marked point process with bounded /  unbounded terminal condition and obstacle, respectively; in Section \ref{section appl}, we present an application of quadratic RBSDEs to the pricing of American options via utility maximization. 

\section{Preliminaries}
\label{section pre}
\subsection{Basic properties on marked point process and notations}
We recall some basic properties of marked point processes. For more details, see \cite{foresta2021optimal, Bremaud1981, last1995marked, cohen2012existence}. 

 Assume that $(\Omega, \mathscr{F}, \mathbb{P})$ is a complete probability space and $E$ is  a mark space equipped with the Borel $\sigma$-algebra $\mathscr{B}(E)$. Given a sequence of random variables $(T_n,\zeta_n)$ taking values in $[0,\infty]\times E$, set $T_0=0$ and $\mathbb P-a.s.$ 
\begin{itemize}
\item $T_n\le T_{n+1},\ \forall n\ge 0;$
\item $T_n<\infty$ implies $T_n<T_{n+1} \ \forall n\ge 0.$
\end{itemize}
Then the sequence $(T_n,\zeta_n)_{n\ge 0}$ is called a marked point process (MPP). Moreover, we assume the marked point process is non-explosive, i.e., $T_n\to\infty,\ \mathbb P-a.s.$.

 For each MPP, define an associated random discrete measure $p$ on $\left((0,+\infty) \times E, \mathscr{B}((0,+\infty) \times E)\right)$:
\begin{equation}
\label{eq p}
    p(\omega, D)=\sum_{n \geq 1} \mathbf{1}_{\left(T_n(\omega), \zeta_n(\omega)\right) \in D} ,\ \forall \omega\in\Omega.
\end{equation} 
For each $\tilde C \in \mathscr{B}(E)$, define a counting process $N_t(\cdot,\tilde C)=p(\cdot,(0, t] \times \tilde C)$ and denote $N_t(\cdot)=N_t(\cdot,E)$. Obviously, both are right continuous increasing process starting from zero.  Define for $t \geq 0$
$$
\mathscr{G}_t^0=\sigma\left(N_s(\tilde C): s \in[0, t], \tilde C \in \mathscr{B}(E)\right)
$$
and $\mathscr{G}_t=\sigma\left(\mathscr{G}_t^0, \mathscr{N}\right)$, where $\mathscr{N}$ is the family of $\mathbb{P}$-null sets of $\mathscr{F}$. 
Denote by $\mathbb{G}=\left(\mathscr{G}_t\right)_{t \geq 0}$ the  completed filtration generated $p$, which is right continuous and satisfies the usual hypotheses.\footnote{ Given a standard Brownian motion $W\in \mathbb R^d$, independent with the MPP, in order to coping with the BSDEs with a Brownian diffusion term as in \cite{foresta2021optimal}, it is natural to enlarge the filtration to $\tilde{\mathbb G}=(\tilde{\mathcal G}_t)$, the completed filtration generated by the MPP and $W$, which satisfies the usual conditions as well.}

Each marked point process has a unique compensator $v$, a predictable random measure such that
$$
\mathbb{E}\left[\int_0^{+\infty} \int_E C_t(e) p(d t d e)\right]=\mathbb{E}\left[\int_0^{+\infty} \int_E C_t(e) v(d t d e)\right]
$$
for all $C$ which is non-negative and $\mathscr{P}^{\mathbb{G}} \otimes \mathscr{B}(E)$-measurable, where $\mathscr{P}^{\mathbb{G}}$ is the $\sigma$-algebra generated by $\mathbb{G}$-predictable processes. Moreover, in this paper we always assume that there exists a function $\phi$ on $\Omega \times[0,+\infty) \times \mathscr{B}(E)$ such that $\nu(\omega, d t d e)=\phi_t(\omega, d e) d A_t(\omega)$, where $A$ is the dual predictable projection of $N$. In other words, $A$ is the unique right continuous increasing predictable process $A$ with $A_0=0$ such that for any non-negative predictable process $D$, it holds that
$$
\mathbb E\left[\int_{0}^\infty D_t dN_t\right]=E\left[\int_{0}^\infty D_t dA_t\right].
$$
In addition, as in \cite{Confortola2013}, we always assume that the kernel $\phi$ satisfies the following properties, which can be proved under mild assumptions on $E$, 
\begin{itemize}
\item for each $(\omega,t)\in \Omega\times [0,\infty)$, $\phi_t(\omega,\cdot)$ is a probability measure on $(E,\mathcal{B}(E))$;
\item for each $\tilde C\in \mathcal B(E)$, $\phi_t(\tilde C)$ is predictable. 
\end{itemize}

Fix a terminal time $T>0$, we can define the integral
$$
\int_0^T \int_E C_t(e) q(d t d e)=\int_0^T \int_E C_t(e) p(d t d e)-\int_0^T \int_E C_t(e) \phi_t(d e) d A_t,
$$
under the condition
$$
\mathbb{E}\left[\int_0^T \int_E\left|C_t(e)\right| \phi_t(d e) d A_t\right]<\infty .
$$
Indeed, the process $\int_0^{\cdot} \int_E C_t(e) q(d t d e)$ is a martingale. Note that  $\int_a^b$ denotes an integral on $(a, b]$ if $b<\infty$, or on $(a, b)$ if $b=\infty$.

The following spaces are frequently called in the sequel.
\begin{itemize}
\item $\mathbb{L}^0$ denotes the space of all real-valued, $\mathcal{G}_T$-measurable random variables.
\item $\mathbb{L}^p :=\left\{\xi \in \mathbb{L}^0 :\|\xi\|_p :=\left\{E\left[|\xi|^p\right]\right\}^{\frac{1}{p}}<\infty\right\}$, for all $p \in[1, \infty)$.
\item $\mathbb{L}^{\infty}:=\left\{\xi \in \mathbb{L}^0:\|\xi\|_{\infty} := {esssup}_{\omega \in \Omega}|\xi(\omega)|<\infty\right\}$.
\item  $\mathbb{C}^0$ denotes the set of all real-valued continuous  process adapted to $\mathbb{G}$ on $[0, T]$. 
\item Let $\mathbb{K}$ be the subset of $\mathbb{C}^0$ that consists of all real-valued increasing and continuous adapted process starting from 0, and $\mathbb K^p$ is a subset of $\mathbb{K}$ such that for each $X\in \mathbb K^p$, $X_T\in \mathbb{L}^p$.
\item  ${S}^0$ denotes the set of real-valued, adapted and c\`adl\`ag processes $\left\{Y_t\right\}_{t \in[0, T]}$. 
\item For any $\left\{\ell_t\right\}_{t \in[0, T]} \in  S^0$, define $\ell_*^{ \pm} := \sup _{t \in[0, T]}\left(\ell_t\right)^{ \pm}$. Then
$$
\ell_* := \sup _{t \in[0, T]}\left|\ell_t\right|=\sup _{t \in[0, T]}\left(\left(\ell_t\right)^{-} \vee\left(\ell_t\right)^{+}\right)=\sup _{t \in[0, T]}\left(\ell_t\right)^{-} \vee \sup _{t \in[0, T]}\left(\ell_t\right)^{+}=\ell_*^{-} \vee \ell_*^{+} .
$$
\item For any real $p \geq 1,\ {S}^p$ denotes the set of real-valued, adapted and c\`adl\`ag processes $\left\{Y_t\right\}_{t \in[0, T]}$ such that
$$
\|Y\|_{{S}^p}:=\mathbb{E}\left[\sup _{0 \leq t \leq T}\left|Y_t\right|^p\right]^{1 / p}<+\infty.
$$
Then $\left({S}^p,\|\cdot\|_{ {S}^p}\right)$ is a Banach space.
\item ${S}^{\infty}$ is the space of $\mathbb{R}$-valued càdlàg and $\mathbb{G}$-progressively measurable processes $Y$ such that
$$
\|Y\|_{{S}^\infty}:=\sup _{0 \leq t \leq T}\left\|Y_t\right\|_{\infty}<+\infty.
$$
\item For any $p \geq 1$, we denote by $\mathcal{E}^p$ the collection of all stochastic processes $Y$ such that $e^{|Y|} \in$ $S^p$.   We write $Y \in \mathcal{E}$ if $Y \in \mathcal{E}^p$ for any $p \geq 1$.
\item  $L^{2}(A)$ is the space of all $\mathbb{G}$-progressively measurable  processes $Y$ such that
$$
\|Y\|_{L^{2}(A)}^2=\mathbb{E}\left[\int_0^T \left|X_s\right|^2 d A_s\right]<\infty .
$$
\item $\mathbb{H}^p$ is the space of real-valued and $\mathbb{G}$-progressively measurable processes $Z$ such that
$$
\|Z\|_{\mathbb{H}^p}^p:=\mathbb{E}\left[\left(\int_0^T\left|Z_t\right|^2 d t\right)^{\frac{p}{2}}\right]<+\infty .
$$
\item $L^0\left(\mathscr{B}(E)\right)$ denotes the space of $\mathcal B(E)$-measurable functions.  For $u \in L^0\left(\mathscr{B}(E)\right)$, define
$$
L^2(E,\mathcal{B}(E),\phi_t(\omega,dy)):=\left\{\left\|u\right\|_t:=\left(\int_E\left|u(e)\right|^2  \phi_t(d e)\right)^{1 / 2}<\infty\right\}.
$$
\item $H^{2,p}_{\nu}$ is the space of predictable processes $U$ such that 
$$
\|U\|_{H_\nu^{2,p}}:=\left(\mathbb{E}\left[\int_{[0, T]} \int_E\left|U_s(e)\right|^2 \phi_s(de) d A_s\right ]^{\frac{p}{2}}\right)^{\frac{1}{p}}<\infty.
$$
\item $H^{2,loc}_{\nu}$ is the space of predictable processes $U$ such that 
$$
\|U\|_{H_\nu^{2,loc}}:=\int_{[0, T]} \int_E\left|U_s(e)\right|^2 \phi_s(de) d A_s<\infty,\ \mathbb P-a.s.
$$
As in \cite{Confortola2013}, we say that $U, U^{\prime} \in H^{2,p}_{\nu}$ (respectively, $U, U^{\prime} \in H^{2,loc}_{\nu}$ ) are equivalent if they coincide almost everywhere with respect to the measure $\phi_t(\omega, d y) d A_t(\omega) \mathbb{P}(d \omega)$  and this happens if and only if $\left\|U-U^{\prime}\right\|_{H_{\nu}^{2,p}}=0$ (equivalently, $\left\|U-U^{\prime}\right\|_{H_{\nu}^{2,loc}}=0,\ \mathbb P-a.s.$). With a little abuse of notation, we  still denote $H^{2,p}_{\nu}$ (respectively, $H^{2,loc}_{\nu}$) the corresponding set of equivalence classes, endowed with the norm $\|\cdot\|_{H_\nu^{2,p}}$ (respectively, $\|\cdot\|_{H_\nu^{2,loc}} $). In addition, $H_{\nu}^{2,p}$ is a Banach space. 
\item $\mathbb{J}^{\infty}$ is the space of functions such that
$$
\|\psi\|_{\mathbb{J}^\infty}:=\left\| \left\|\psi\right\|_{{L}^{\infty}(\nu(\omega))}\right\|_{\infty}<\infty.
$$
\item $\mathcal S_{0,T}$ denotes the collection of $\mathbb{G}$-stopping times $\tau$ such that $0\le\tau\le T,\ \mathbb P-a.s.$. For any $\tau\in\mathcal S_{0,T}$, $\mathcal S_{\tau,T}$ denotes the collection of $\mathbb{G}$-stopping times $\tilde\tau$ such that $\tau\le\tilde\tau\le T,\ \mathbb P-a.s.$
\end{itemize}

For $u \in L^0\left(\mathscr{B}(E)\right)$ and $\lambda >0$, we introduce the positive predictable process denoted by
$$
j_\lambda (\cdot,u)=\int_E (\mathrm{e}^{\lambda  u(e)}-1-\lambda  u(e))\phi_{\cdot}(de).
$$

\subsection{Reflected BSDEs with a convex/concave generator driven by a marked point process}
As in \cite{foresta2021optimal}, a reflected BSDE driven by a marked point process is formulated as follows, $\mathbb P$-a.s.
\begin{equation}
\label{reflected BSDE}
\left\{\begin{array}{l}
Y_t=\xi+\int_t^T f\left(s,Y_s, U_s\right) d A_s +\int_t^T d K_s-\int_t^T \int_E U_s(e) q(d s d e), \quad 0 \leq t \leq T, \\
Y_t \geq L_t, \quad 0 \leq t \leq T, \\
\int_0^T\left(Y_{s^{-}}-L_{s^{-}}\right) d K_s=0.\\
(Y,U,K)\in \mathcal E\times H_{\nu}^{2,p}\times \mathbb K^p,\ \text{for each $p\ge 1$}.
\end{array}\right.
\end{equation}
The RBSDE (\ref{reflected BSDE}) with coefficient $(\xi ,f,L)$ is denoted by RBSDE $(\xi ,f,L)$.

We are now ready to state the general assumptions that will be adopted throughout the paper.
\hspace*{\fill}\\
\hspace*{\fill}\\
\noindent(\textbf{H1}) The process A is continuous, with $\|A_T\|_\infty<\infty$.
\hspace*{\fill}\\
\hspace*{\fill}\\

Assumption (H1) is on the dual predictable projection $A$ of the counting process $N$ corresponding to the measure $p$. We would like to emphasize that for $dA_t$, we do not require absolute continuity with respect to the Lebesgue measure $dt$. That is to say the compensator $v$ does not admit a decomposition of the form $v_t(\omega,dt,dx)=\zeta(\omega,t,x)\lambda_t(dx)dt$ as in \cite{karoui2016quadratic}.

\hspace*{\fill}\\
\hspace*{\fill}\\
\noindent(\textbf{H2}) The obstacle process $L$ is continuous with $L_T\le\xi$.
\hspace*{\fill}\\
\hspace*{\fill}\\

\noindent(\textbf{H3})
For every $\omega \in \Omega, \ t \in[0, T],\ r \in \mathbb{R}$, the mapping
$
f(\omega, t, r, \cdot):L^0(\mathcal{B}(E)) \rightarrow \mathbb{R}
$
satisfies:
for every $U \in {H_\nu^{2,2}}$,
$$
(\omega, t, r) \mapsto f\left(\omega, t, r, U_t(\omega, \cdot)\right)
$$
is Prog $\otimes \mathscr{B}(\mathbb{R})$-measurable.
\hspace*{\fill}\\

The following integrability and growth conditions are naturally assumed in literature. The so-called quadratic-exponential growth condition is aroused from a utility maximization problem. See for instance \cite{Morlais_2009}. We will present an application that fulfills this assumption in section \ref{section appl}. 
\hspace*{\fill}\\

\noindent(\textbf{H4})

\textbf{(a) (Continuity condition)} 
For every $\omega \in \Omega,\ t \in[0, T],\ y \in \mathbb{R}$, $u\in L^2(E,\mathcal{B}(E),\phi_t(\omega,dy))$, $(y, u) \longrightarrow f(t, y, u)$ is continuous.
\hspace*{\fill}\\
\hspace*{\fill}\\

\textbf{(b) (Lipschitz condition in $y$)} 
There exist $\tilde{\beta}\geq 0$, such that for every $\omega \in \Omega,\ t \in[0, T],\ y, y^{\prime} \in \mathbb{R}$, $u\in L^2(E,\mathcal{B}(E),\phi_t(\omega,dy))$, we have
$$
\begin{aligned}
& \left|f(\omega, t, y, u(\cdot))-f\left(\omega, t, y^{\prime}, u(\cdot)\right)\right| \leq \tilde{\beta}\left|y-y^{\prime}\right|.
\end{aligned}
$$
\hspace*{\fill}\\
\hspace*{\fill}\\

\textbf{(c) (Quadratic-exponential growth condition)}
For all $t\in[0, T]$, $(y,u) \in \mathbb{R} \times  L^2(E,\mathcal{B}(E),\phi_t(\omega,dy)):\ \mathbb{P}$-a.s,
$$
\underline{q}(t, y, u)=-\frac{1}{\lambda} j_{\lambda}(t,- u)-\alpha_t-\beta|y| \leq f(t, y, u) \leq \frac{1}{\lambda} j_{\lambda}(t, u)+\alpha_t+\beta|y|=\bar{q}(t, y, u) .
$$
where $\{\alpha_t\}_{0 \leq t \leq T}$ is  a progressively measurable nonnegative stochastic process.

\hspace*{\fill}\\
\hspace*{\fill}\\

\textbf{(d) Integrability condition} We assume necessarily,
{\color{blue}
$$
\forall p>0, \quad \mathbb{E}\left[\exp \left\{p \lambda e^{\beta A_T}(|\xi|\vee L^+_*)+p\lambda  \int_0^T e^{\beta A_s}\alpha_s d A_s\right\}+\int_0^T\alpha_s^2dA_s\right]<\infty.
$$
}
\hspace*{\fill}\\

\textbf{(e) (Convexity/Concavity condition)}
For each $(t, y) \in[0, T] \times \mathbb{R}$, $u\in L^2(E,\mathcal{B}(E),\phi_t(\omega,dy)),\ u \rightarrow f(t, y, u)$ is convex or concave.
\hspace*{\fill}\\

Under appropriate  assumptions, we aim to find a unique solution $(Y, U, K) \in \mathcal E\times H_{\nu}^{2,p}\times \mathbb K^p$, for each $p\ge1$, for RBSDE $\left(\xi, f, L\right)$ in this paper.

\begin{remark}
We get rid of the Brownian term as in e.g.\cite{Confortola2013} for simplicity. The discussion on Brownian motion driven quadratic-exponential RBSDE paralleled with this paper can be found in \cite{bayraktar2012quadratic}. 
More precisely, consider the following RBSDE,
\begin{equation}
\label{BSDE with B}
\left\{\begin{array}{l}
Y_t=\xi+\int_t^T f\left(s,Y_s, U_s\right) d A_s +\int_t^T g\left(s,Y_s, Z_s\right) d s+\int_t^T d K_s-\int_t^T Z_s d B_s-\int_t^T \int_E U_s(e) q(d s d e), \quad 0 \leq t \leq T, \quad \mathbb{P} \text {-a.s. } \\
Y_t \geq L_t, \quad 0 \leq t \leq T,\quad \mathbb{P} \text {-a.s. } \\
\int_0^T\left(Y_{s^{-}}-L_{s^{-}}\right) d K_s=0, \quad \mathbb{P} \text {-a.s. }
\end{array}\right.
\end{equation}

Assume that $g(t,y,z)$ satisfies additionally:\\
(i) for all $t \in[0, T]$, for all $y \in \mathbb{R}, z \longmapsto g(t, y, z)$ is convex or concave;\\
(ii) for all $(t, z) \in[0, T] \times \mathbb{R}$,
$$
\forall\left(y, y^{\prime}\right) \in \mathbb{R}^2, \quad\left|g(t, y, z)-g\left(t, y^{\prime}, z\right)\right| \leq \tilde{\beta}\left|y-y^{\prime}\right| ;
$$
(iii) $g$ has the following quadratic growth: there exists a positive constant $\gamma$ such that
$$
\forall(t, y, z) \in[0, T] \times \mathbb{R} \times \mathbb{R}^d, \quad|g(t, y, z)| \leq \alpha_t+\beta|y|+\frac{\gamma}{2}|z|^2.
$$

Combining our discussion with that in \cite{bayraktar2012quadratic}, the well-posedness of RBSDEs including Brownian term (\ref{BSDE with B}) as in \cite{foresta2021optimal} holds well without additional technical issue.  
\end{remark}

\begin{remark}
\label{remark_RBSDEJ}
Our techniques are easily be adopted to RBSDEs with jumps as in \cite{hamadene2003reflected}, in which the compensator admits a decomposition $\nu(dtde)=\phi_t(de)dt$.
More generally, RBSDEs of the following form, with a general generator $f(t,y,z,u)$, where $C_t$ is a predictable, continuous non-decreasing process starting from zero, are also included in our scope with a parallel discussion.
\begin{equation}
\label{BSDEJ with B}
\left\{\begin{array}{l}
Y_t=\xi+\int_t^T f\left(s,Y_s, Z_s, U_s\right) d C_s +\int_t^T d K_s-\int_t^T Z_s d B_s-\int_t^T \int_E U_s(e) q(d s d e), \quad 0 \leq t \leq T, \quad \mathbb{P} \text {-a.s. } \\
Y_t \geq L_t, \quad 0 \leq t \leq T,\quad \mathbb{P} \text {-a.s. } \\
\int_0^T\left(Y_{s^{-}}-L_{s^{-}}\right) d K_s=0, \quad \mathbb{P} \text {-a.s. }
\end{array}\right.
\end{equation}

We assume the structural and growth conditions on $f$ read as follows, where the differential form  is understood as integration on any measurable subset of $[0,T]$.
\hspace*{\fill}\\

\noindent(\textbf{H3*}) For every $\omega \in \Omega,\ t \in[0, T],\ r\in \mathbb{R},\ z\in\mathbb R^d$, the mapping
$
f(\omega, t, r, z,\cdot):L^2(E,\mathcal{B}(E),\phi_t(\omega,dy))\rightarrow \mathbb{R}
$
satisfies:
for every $U \in {H_\nu^{2,2}}$,
$$
(\omega, t, r,z) \mapsto f\left(\omega, t, r,z, U_t(\omega, \cdot)\right)
$$
is Prog $\otimes \mathscr{B}(\mathbb{R})\otimes \mathscr{B}(\mathbb{R}^d)$-measurable.

\hspace*{\fill}\\
\hspace*{\fill}\\
\noindent(\textbf{H4*})

\textbf{(a) (Continuity condition)} 
For every $\omega \in \Omega, t \in[0, T], y \in \mathbb{R}$, $z\in\mathbb R^d$ $u\in L^2(E,\mathcal{B}(E),\phi_t(\omega,dy))$, $(y,z, u) \longrightarrow f(t, y,z, u)$ is continuous.
\hspace*{\fill}\\
\hspace*{\fill}\\

\textbf{(b) (Lipschitz condition in $y$)}  
There exists $\beta\geq 0$, such that for every $\omega \in \Omega,\ t \in[0, T],\ y, y^{\prime} \in \mathbb{R}$, $z\in\mathbb R^d$, \ $u\in L^2(E,\mathcal{B}(E),\phi_t(\omega,dy))$, we have
$$
\begin{aligned}
& \left|f(\omega, t, y,z, u(\cdot))-f\left(\omega, t, y^{\prime},z, u(\cdot)\right)\right| dC_t\leq \beta\left|y-y^{\prime}\right|(dA_t+dt).
\end{aligned}
$$

\hspace*{\fill}\\

\textbf{(c) (Quadratic-exponential growth condition)}
For all $t \in[0, T], \ (y,z, u) \in \mathbb{R} \times\mathbb R^d\times L^2(E,\mathcal{B}(E),\phi_t(\omega,dy)):\ \mathbb{P}$-a.s, there exists $\lambda>0$ such that
\begin{equation}
\label{fc condition}
\begin{aligned}
-\left(\alpha_t+\beta|y|+\frac{\gamma}{2}|z|^2 \right)dt-\left(-\alpha_t-\beta|y|-\frac{\gamma}{2}|z|^2 -\frac{1}{\lambda} j_{\lambda}(t, -u)\right)dA_t\\
\le f(t, y, z,u)dC_t
\leq \left(\alpha_t+\beta|y|+\frac{1}{\lambda} j_{\lambda}(t, u)\right)dA_t+\left(\alpha_t+\beta|y|+\frac{\gamma}{2}|z|^2 \right)dt,
\end{aligned}
\end{equation}
where $\{\alpha_t\}_{0 \leq t \leq T}$ is  a progressively measurable nonnegative stochastic process.

\hspace*{\fill}\\
\hspace*{\fill}\\

\textbf{(d) (Integrability condition)}
We assume necessarily,
$$
\forall p>0, \quad \mathbb{E}\left[\exp \left(p\left(\left|\xi\right|+\int_0^T \alpha_s (d A_s+ds)\right)\right)+\int_0^T\alpha_s^2(dA_s+ds)\right]<+\infty.
$$
\hspace*{\fill}\\

\textbf{(e) (Convexity/Concavity condition)}  
 For all $t \in[0, T]$ and $y \in \mathbb{R}$, $(z,u)\in\mathbb R^d\times L^2(E,\mathcal{B}(E),\phi_t(\omega,dy))\longmapsto f(t, y, z,u)$ is 
jointly convex or concave.
\hspace*{\fill}\\

Other assumptions on the process $A$, the barrier $L$ and the generator $f$ are inherited from ours. Based on the well-posedness of (\ref{BSDEJ with B}), we provide a financial application on American option pricing in a jump market in section \ref{section appl}.
\end{remark}

\subsection{Lipschitz RBSDEs  driven by a marked point process with bounded terminal and bounded obstacle }

In order to obtain the well-posedness of the quadratic exponential RBSDE (\ref{reflected BSDE}), we first restate the well-posedness result on Lipschitz RBSDEs driven by a marked point process with bounded terminal and bounded obstacle. The well-podedness of Lipschitz RBSDEs plays an important role in the sequel, since inspired by Kobylanski \cite{kobylanski2000backward}, we will take advantage of an approximation procedure to approximate the quadratic generator by a sequence of Lipschitz generators. The well-posedness result for Lipschitz RBSDEs  is inherited from Foresta \cite[Theorem 4.1]{foresta2021optimal} and the a priori estimate Lemma 3.2 therein.  For the convenience of later reference, it is enough to restate the result in \cite{foresta2021optimal} under stronger assumptions that the terminal and obstacle are bounded. The additional assumptions are listed as follows.

\noindent(\textbf{H4'})

(\textbf{a})
There exists a constant $M>0$ such that $L_*+|\xi|\le M, \ \mathbb P-a.s.$
\hspace*{\fill}\\
\hspace*{\fill}\\

\textbf{(b)} 
There exist $L_f \geq 0,\ L_U \geq 0$ such that for every $\omega \in \Omega,\ t \in[0, T],\ y, y^{\prime} \in \mathbb{R}$, $u, u^{\prime} \in L^2(E,\mathcal{B}(E),\phi_t(\omega,dy))$ we have
$$
\left|f(\omega, t, y, u(\cdot))-f\left(\omega, t, y^{\prime}, u^{\prime}(\cdot)\right)\right| \leq L_f\left|y-y^{\prime}\right| +L_U\left(\int_E\left|u(e)-u^{\prime}(e)\right|^2 \phi_t(\omega, d e)\right)^{1 / 2}.
$$
\hspace*{\fill}\\
\hspace*{\fill}\\

\textbf{(c)} 
We have
$$
\mathbb{E}\left[\int_0^T |f(s, 0,0)|^2 d A_s\right]<\infty.
$$

The well-posedness result on Lipschitz RBSDE (\ref{reflected BSDE}) reads as follows.

\begin{thm}
\label{thm_lip}
Let assumptions (H1), (H2), (H3) and (H4') hold, then,

(i) there exists a solution $(Y,U,K)\in L^2(A)\times H_{\nu}^{2,2} \times \mathbb K^2$ to (\ref{reflected BSDE}). Moreover, with the help of \cite[Lemma 3.2]{foresta2021optimal}, $Y\in S^2$.

(ii) If in addition, there exists a positive constant $M_0$ such that , for each $t\in[0,T]$ and $(y,u)\in\mathbb R\times  L^2(E,\mathcal{B}(E),\phi_t(\omega,dy))$
$$
|f(t,y,u)|\le M_0.
$$
Then, there exists a unique solution $(Y,U,K)\in S^\infty\times J^\infty\times \mathbb K^2$.
\end{thm}
\begin{proof}
Assertion (i) is directly inherited from  \cite[Theorem 4.1 ]{foresta2021optimal} and the a priori estimate Lemma 3.2 therein. For assertion (ii), note that by \cite[Proposition 3.1]{foresta2021optimal},
\begin{equation}
Y_t=\underset{\tau \in \mathscr{S}_{t,T}}{\operatorname{ess} \sup } \mathbb{E}\left[\int_t^\tau f(s,Y_s,U_s) d A_s+h_\tau {1}_{\{\tau<T\}}+\xi {1}_{\{\tau \geq T\}} \mid \mathscr{G}_t\right] .
\end{equation}
Then, $\|Y\|_{S^\infty}\le M_0\|A_T\|_{\infty}+M<\infty$.
Therefore, in the spirit of \cite[ Corollary 1 ]{Morlais_2009}, we are able to find $U$ in $J^\infty$ by a truncation technique.
\end{proof}

Thereafter, We will break the proof of well-posedness of the quadratic exponential RBSDE (\ref{reflected BSDE}) with a convex/concave generator  down into the following steps:
\begin{itemize}
\item Prove a comparison theorem for quadratic exponential RBSDE with a convex/concave generator by applying $\theta$-method as in \cite{briand2008quadratic} like Bayraktar and Yao \cite{bayraktar2012quadratic}. (Section \ref{section 3})
\item Make use of this comparison theorem to prove uniform a priori estimates in quadratic exponential RBSDEs. (Section \ref{section priori})
\item With the help of those a priori estimates to prove the well-posedness with bounded terminal \& bounded obstacle (Section \ref{section existence for bdd}) and   unbounded terminal \& unbounded obstacle (Section \ref{section existence for unbdd}).
\end{itemize}

\section{A comparison theorem for convex / concave quadratic exponential RBSDEs}
\label{section 3}

We first prove a comparison theorem for quadratic exponential RBSDE with a convex generator by applying $\theta$-method as in \cite{briand2008quadratic}. The concave generator scenario can be handled with a parallel discussion. The idea borrows from Theorem 5.1 in \cite{bayraktar2012quadratic}. The uniqueness is a straightforward corollary of the comparison theorem.
\begin{thm}
\label{comparison thm}
Let $(\xi, f, L),(\hat{\xi}, \hat{f}, \widehat{L})$ be two parameter sets and let $(Y, U, K)($ resp. $(\widehat{Y}, \widehat{U}, \widehat{K}))$ be a solution of $R B S D E(\xi, f, L)$ (resp. $R B S D E(\hat{\xi}, \hat{f}, \widehat{L}))$ such that $\mathbb P$-a.s., $\xi \leq \hat{\xi}$ and that $L_t \leq \widehat{L}_t$ for any $t \in[0, T]$.
For  process $\alpha$ and constants $\beta, \tilde{\beta} \geq 0, \lambda >0$, suppose (H1)-(H2) and (H4)(d) hold, $Y,\ \hat Y\in\mathcal E$, $U, \ \hat U\in H_{\nu}^{2,p}$ and $K,\ \hat K\in \mathbb K^p$, for each $p\ge1$.  If in addition either of the following two holds:
\begin{enumerate}
    \item[(i)] $f$ satisfies (H3), {\color{red} (H4)(a-b)}, $f$ is convex in $u$, $\Delta f(t) := f\left(t, \widehat{Y}_t, \widehat{U}_t\right)-\hat{f}\left(t, \widehat{Y}_t, \widehat{U}_t\right) \leq 0$, dt $\otimes d \mathbb P$-a.e., and \\
    $$f(t,y,u)\le \alpha_t+\beta|y|+\frac{1}{\lambda}j_{\lambda}(t,u);$$
    \item[(ii)] $\hat{f}$ satisfies (H3), {\color{red} (H4)(a-b)}, $\hat{f}$ is convex in $u$, $\Delta f(t) := f\left(t, Y_t, U_t\right)-\hat{f}\left(t, Y_t, U_t\right) \leq 0,\ d t \otimes d \mathbb P$-a.e., and\\ 
    $$\hat f(t,y,u)\le \alpha_t+\beta|y|+\frac{1}{\lambda}j_{\lambda}(t,u);$$
\end{enumerate}
then it holds $\mathbb P$-a.s. that $Y_t \leq \widehat{Y}_t$ for any $t \in[0, T]$
\end{thm} 
\begin{proof}
Fix $\theta \in(0,1)$. We set $\tilde Y := Y-\theta \widehat{Y},\tilde U := U-\theta \widehat{U}$ and define an $\mathbb{G}$-progressively measurable process
$$
\begin{aligned}
a_t := & \mathbf{1}_{\left\{Y_t \geq 0\right\}}\left(\mathbf{1}_{\left\{Y_t \neq \widehat{Y}_t\right\}} \frac{\mathfrak{F}\left(t, Y_t, \widehat{U}_t\right)-\mathfrak{F}\left(t, \widehat{Y}_t, \widehat{U}_t\right)}{Y_t-\widehat{Y}_t}-\tilde{\beta} \mathbf{1}_{\left\{Y_t=\widehat{Y}_t\right\}}\right)-\tilde{\beta} \mathbf{1}_{\left\{Y_t<0 \leq \widehat{Y}_t\right\}} \\
& +\mathbf{1}_{\left\{Y_t \vee \widehat{Y}_t<0\right\}}\left(\mathbf{1}_{\left\{\tilde Y_t \neq 0\right\}} \frac{\mathfrak{F}\left(t, Y_t, U_t\right)-\mathfrak{F}\left(t, \theta \widehat{Y}_t, U_t\right)}{\tilde Y_t}-\tilde{\beta} \mathbf{1}_{\left\{\tilde Y_t=0\right\}}\right), \quad t \in[0, T],
\end{aligned}
$$
where $\mathfrak{F}$ stands for $f$ if (i) holds, and for $\hat{f}$ otherwise. It follows that $\tilde A_t := \int_0^t a_s d A_s, t \in[0, T]$ is an $\mathbb{G}$-adapted process. By \textcolor{red}{(H4)(b)}, it holds $d t \otimes d \mathbb P$-a.s. that $\left|a_t\right| \leq \tilde{\beta}$. Thus $\tilde A_* := \sup _{t \in[0, T]}\left|\tilde A_t\right| \leq \int_0^T\left|a_s\right| d A_s \leq \tilde{\beta} \|A_T\|_\infty, P$-a.s. In the light of (H3)(e) and the convexity of $\mathfrak{F}$ in $u$, it holds $d t \otimes d \mathbb P$-a.s. that
\begin{equation}
\label{cvx estimate}
\mathfrak{F}\left(t, y, U_t\right) \leq \theta \mathfrak{F}\left(t, y, \widehat{U}_t\right)+(1-\theta) \mathfrak{F}\left(t, y, \frac{\tilde U_t}{1-\theta}\right) \leq \theta \mathfrak{F}\left(t, y, \widehat{U}_t\right)+(1-\theta)(\alpha_t+\beta|y|)+\frac{1-\theta}{\lambda}j_{\lambda}(\frac{\tilde U_t}{1-\theta}), \quad \forall y \in \mathbb{R} .
\end{equation}
Let $\zeta_\theta := \frac{\lambda  e^{\tilde{\beta} \|A_T\|_\infty}}{1-\theta}$. Applying Itô's formula to the process $\Gamma_t := \exp \left\{\zeta_\theta e^{\tilde A_t} \tilde Y_t\right\}$, $t \in[0, T]$ yields that
$$
\begin{aligned}
& \Gamma_t=\Gamma_T+\int_t^T \Gamma_{s}\left[\zeta_\theta e^{\tilde A_s}F_s-\int_E(e^{\zeta_\theta e^{\tilde A_s}\tilde U_s}-\zeta_\theta e^{\tilde A_s}\tilde U_s-1)\phi_s(de)\right]dA_s\\
&\quad\quad-\int_t^T\int_E\Gamma_{s^-}(e^{\zeta_\theta e^{\tilde A_s}\tilde U_s}-1)q(dsde)+\zeta_\theta \int_{t}^{T} \Gamma_{s} e^{\tilde A_s}\left(d K_s-\theta d \widehat{K}_s\right)\\
&:=\Gamma_T+\int_t^T G_s d A_s-\int_t^T \int_E Q_s q(dsde)+\zeta_\theta \int_{t}^{T} \Gamma_{s} e^{\tilde A_s}\left(d K_s-\theta d \widehat{K}_s\right)
\end{aligned}
$$
where 
$$
F_t=f\left(t, Y_t, U_t\right)-\theta \hat{f}\left(t, \widehat{Y}_t, \widehat{U}_t\right)-a_t \tilde Y_t;
$$
$$
G_t= \Gamma_t \hat{G}_t=\Gamma_t\left( \zeta_\theta  e^{\tilde A_t}\left(f\left(t, Y_t, U_t\right)-\theta \hat{f}\left(t, \widehat{Y}_t, \widehat{U}_t\right)-a_t \tilde Y_t\right )-j_1(\zeta_\theta e^{\tilde A}\tilde U_t)\right);
$$
$$
Q_t=\Gamma_{t^-}(e^{\zeta_\theta e^{\tilde A_t}\tilde U_t}-1).
$$
Clearly, it holds $d t \otimes d \mathbb P$-a.e. that
\begin{equation}
\label{G estimate1}
G_t \leq \zeta_\theta \Gamma_t e^{\tilde A_t}\left(\mathfrak{F}\left(t, Y_t, U_t\right)-\theta \mathfrak{F}\left(t, \widehat{Y}_t, \widehat{U}_t\right)-a_t \tilde Y_t\right)-\Gamma_tj_1(\zeta_\theta e^{\tilde A}\tilde U_t).
\end{equation}
whether (i) or (ii) holds. Moreover, we claim that
\begin{equation}
\label{G estimate2}
G_t \leq  \Gamma_t \tilde G_t, \quad d t \otimes d \mathbb P \text {-a.e. },
\end{equation}
where $\tilde G_t = \lambda  e^{2 \tilde{\beta} \|A_T\|_\infty} \left(\alpha_t+(\beta+\tilde{\beta})\left(Y_t^{+}+\widehat{Y}_t^{-}\right)\right)$.
This conclusion will be proved later in Lemma \ref{G-estimation}.

Now, we define a process
$$
D_t :=\exp\left\{ \int_0^t\tilde G_sdA_s\right\} = \exp \left\{\lambda  e^{2 \tilde{\beta} \|A_T\|_\infty} \int_0^t\left(\alpha_t+(\beta+\tilde{\beta})\left(Y_s^{+}+\widehat{Y}_s^{-}\right)\right) d A_s\right\}, \quad t \in[0, T] .
$$
Given $n \in \mathbb{N}$, we define a sequence of $\mathbb{G}$-stopping time
$\{\tilde \tau_n\}$ as the localizing sequence of the local martingale $\int_0^{\cdot}\int_E D_s Q_s q(dsde)$ and $\tau_n=\tilde\tau_n\wedge T\vee t$, for a fixed $t \in[0, T]$.
Clearly, $\lim _{n \rightarrow \infty} \uparrow \tau_n=T, P$-a.s. Note that:
$$
\begin{aligned}
d D_t \Gamma_t &= \Gamma_t d D_t + D_t d \Gamma_t\\
&= \Gamma_t  D_t \tilde G_t d A_t + D_t\left(-G_t d A_t +\int_E Q_t q(dtde)- \zeta_\theta \Gamma_t e^{\tilde A_{s}}\left(d K_s-\theta d \widehat{K}_s \right)\right)\\
&= D_t \left(\Gamma_t \tilde G_t - G_t \right)d A_t + D_t\left(\int_E Q_t q(dtde)- \zeta_\theta \Gamma_t e^{\tilde A_{s}}\left(d K_s-\theta d \widehat{K}_s \right)\right).
\end{aligned}
$$
Integration by parts and (\ref{G estimate2}) imply that $\mathbb P$-a.s.
\begin{equation}
\label{5.9}
\Gamma_{t} \leq D_{t} \Gamma_{ t} \leq D_{\tau_n} \Gamma_{\tau_n}+\zeta_\theta \int_{ t}^{\tau_n} D_s \Gamma_{s^{-}} e^{\tilde A_s} d K_s-\int_{ t}^{\tau_n}\int_E D_sQ_sq(dsde) .
\end{equation}
Since it holds $\mathbb P$-a.s. that $L_s \leq \widehat{L}_s \leq \widehat{Y}_s$ for any $s \in[0, T]$, the flat-off condition of $(Y, U, K)$ implies that $\mathbb P$-a.s.
\begin{equation}
\label{5.10}
\int_0^T D_s \Gamma_{s^-} d K_s=\int_0^T \mathbf{1}_{\left\{Y_{s^{-}}=L_{s^{-}}\right\}} D_s \Gamma_{s^-} d K_s \leq \int_0^T \mathbf{1}_{\left\{Y_{s^-} \leq \hat{Y}_{s^-}\right\}} D_s \Gamma_{s^-} d K_s \leq \int_0^T D_s e^{\lambda  e^{2 \tilde{\beta} \|A_T\|_\infty} Y_{s^-}^{+}} d K_s \leq \eta K_T
\end{equation}
with $\eta := \exp \left\{(\beta+\tilde{\beta})  \|A_T\|_\infty \lambda e^{2 \tilde{\beta} \|A_T\|_\infty} \widehat{Y}_*^{-}+(1+(\beta+\tilde{\beta}) \|A_T\|_\infty) \lambda  e^{2 \tilde{\beta} \|A_T\|_\infty} Y_*^{+} +\lambda e^{2 \tilde{\beta} \|A_T\|_\infty}\int_0^T \alpha_t d A_t\right\}$. Hölder's inequality and the integrability conditions of $(Y,\hat Y,K)$ imply that
\begin{equation}
\label{5.11}
\mathbb E\left[\eta\left(1+K_T\right)\right] \leq\|\eta\|_{\mathbb{L}^{2}\left(\mathcal{G}_T\right)}\left(1+\left\|K_T\right\|_{\mathbb{L}^2\left(\mathcal{G}_T\right)}\right)<\infty, 
\end{equation}
and
\begin{equation}
\label{5.12}
\mathbb E\left[D_T \Gamma_*\right] \leq \mathbb E\left[\exp \left\{\left((\beta+\tilde{\beta}) \lambda  \|A_T\|_\infty e^{2 \tilde{\beta} \|A_T\|_\infty}+\zeta_\theta e^{\tilde{\beta} \|A_T\|_\infty}\right)\left(Y_*^{+}+\widehat{Y}_*^{-}\right)+\lambda e^{2 \tilde{\beta} \|A_T\|_\infty}\int_0^T \alpha_t d A_t\right\}\right]<\infty .
\end{equation}
Then since
 $\int_0^{\cdot\wedge \tau_n}\int_E D_s Q_s q(dsde)$ is a martingale,
taking $\mathbb E\left[\cdot \mid \mathcal{G}_{ t}\right]$ in (\ref{5.9}), we can deduce from (\ref{5.10})-(\ref{5.12}) that $\mathbb P$-a.s.
$$
\Gamma_{t} \leq\mathbb  E\left[D_{\tau_n} \Gamma_{\tau_n} \mid \mathcal{G}_{ t}\right]+e^{\tilde\beta\|A_T\|_\infty} \zeta_\theta E\left[\eta K_T \mid \mathcal{G}_{ t}\right].
$$
As $n \rightarrow \infty$, dominated convergence theorem, (\ref{5.12}) and (\ref{5.11}) together imply that $\mathbb P$-a.s.
$$
\begin{aligned}
\Gamma_t &\leq \mathbb E\left[D_T \Gamma_T \mid \mathcal{G}_t\right]+e^{\tilde\beta\|A_T\|_\infty} \zeta_\theta \mathbb E\left[\eta K_T \mid \mathcal{G}_t\right] \\
&\leq \mathbb E\left[D_T e^{\lambda  e^{2 \tilde{\beta} \|A_T\|_\infty} \xi^{+}} \mid \mathcal{G}_t\right]+e^{\tilde\beta \|A_T\|_\infty} \zeta_\theta \mathbb E\left[\eta K_T \mid \mathcal{G}_t\right]\\
&\leq e^{\tilde\beta\|A_T\|_\infty}\left(1 \vee \zeta_\theta\right) \mathbb E\left[\eta\left(1+K_T\right) \mid \mathcal{G}_t\right],
\end{aligned}
$$
which leads to
$$
Y_t-\theta \widehat{Y}_t \leq \frac{1-\theta}{\lambda } \ln \left(1 \vee \frac{\lambda  e^{\tilde{\beta} \|A_T\|_\infty}}{1-\theta}\right) e^{-\tilde{\beta} \|A_T\|_\infty-\tilde A_t}+\frac{1-\theta}{\lambda }\left(e^{\tilde\beta\|A_T\|_\infty}+\ln \left(\mathbb E\left[\eta\left(1+K_T\right) \mid \mathcal{G}_t\right]\right)\right) e^{-\tilde{\beta} \|A_T\|_\infty-\tilde A_t}, \quad \mathbb P \text {-a.s. }
$$
Letting $\theta \rightarrow 1$ yields that $Y_t-\widehat{Y}_t \leq 0, \mathbb P$-a.s. Then the right continuity of processes $Y$ and $\widehat{Y}$ proves the theorem.
\end{proof}
As stated above, the proof of Theorem \ref{comparison thm} is based on the following estimate on $G$.
\begin{lemma}
\label{G-estimation}
With the same notation and assumptions in Theorem \ref{comparison thm} and its proof, we define
$$
G_t=\Gamma_t\hat G_t,
$$
where
$$
\hat G_t=\zeta_\theta  e^{\tilde A_t}\left(f\left(t, Y_t, U_t\right)-\theta \hat{f}\left(t, \widehat{Y}_t, \widehat{U}_t\right)-a_t \tilde Y_t\right )- j_1(\zeta_\theta e^{\tilde A}\tilde U_t).
$$
Then the following estimation on process $\hat G$ holds.
\begin{equation}
\hat G_t \leq \lambda  e^{2 \tilde{\beta} \|A_T\|_\infty} \left(\alpha_t+(\beta+\tilde{\beta})\left(Y_t^{+}+\widehat{Y}_t^{-}\right)\right):=\tilde G_t, \quad d t \otimes d \mathbb P \text {-a.e. }
\end{equation}
\end{lemma}
The proof of Lemma \ref{G-estimation} is postponed to Appendix \ref{appexdix A} for completeness.

\begin{remark}
\label{bsde compare}
In fact, the comparison theorem still holds well if we assume $L=-\infty$. Since in this case, the RBSDE boils down to a BSDE. As long as it is uniquely solvable with a solution $(Y,U)\in \mathcal E\times H_\nu^{2,2}$, see {\color{red}\cite[Theorem 3.11]{2310.14728}}, we are able to repeat the proof above with $K\equiv 0$. 
\end{remark}

The following uniqueness result is straightforward, in view of applying It\^o's formula to $|Y-\hat Y|^2$.
\begin{corollary}[Uniqueness]
\label{uniqueness}
 Assume that (H1)-(H4) are fulfilled, then the RBSDE $\left(\xi, f, L\right)$ admits  at most one  solution $(Y, U, K) \in \mathcal E\times H_{\nu}^{2,p}\times \mathbb K^p$, for each $p\ge 1$.
\end{corollary}

\section{A priori estimates in quadratic exponential RBSDEs}
\label{section priori}
In this section, we provide some useful a priori estimates for quadratic exponential RBSDEs. To prove the a priori estimates, we need the following additional assumptions.
\hspace*{\fill}\\

{\color{red}
\noindent(\textbf{H1'}) The process $A$ is continuous with $\|A_s-A_t\|_\infty<|\rho(s)-\rho(t)|$, for any $s,t\in[0,T]$, where $\rho(\cdot)$ is a deterministic continuous increasing function with $\rho(0)=0$.
}
\hspace*{\fill}\\

It is obvious that (H1') implies (H1), i.e. $\|A_T\|_\infty<\infty$.

\hspace*{\fill}\\

{\color{red}
\noindent\textbf{(H5) (Uniform linear bound condition )}  
There exists a positive constant $C_0$ such that for each $t \in[0, T]$, $u\in L^2(E,\mathcal{B}(E),\phi_t(\omega,dy))$, if 
 $f$ is convex (resp. concave) in $u$, then $f(t,0,u)-f(t,0,0)\ge -C_0\|u\|_t$ (resp. $f(t,0,u)-f(t,0,0)\le C_0\|u\|_t$ ).
}
\hspace*{\fill}\\

In the sequel, if no special announcement, we always  assume $f$ is convex in $u$ without loss of generality.

\subsection{ A priori estimate on $Y$ with bounded terminal and bounded obstacle}
\label{subsec4.1}
For each $\mathcal{G}$-stopping time $\tau$ taking values in $[0, T]$ and for every {\color{red}\textbf{bounded}} $\mathcal{G}_\tau$-measurable function $\eta$, 
we first define the following $f$-evaluation:
$$
\mathcal{E}_{t, \tau}^f[\eta]:=y_t^\tau, \quad \forall t \in[0, T],
$$
where $(y^\tau,u^{\tau})$ is the solution to the following BSDE,
\begin{equation}
\label{tautruncBSDE}
y_t^{\tau}=\eta+\int_t^T 1_{s\le\tau}f\left(s, y_s^{\tau}, u_{s}^{\tau}\right) d A_s-\int_t^T\int_E u_{s}^{\tau} q(dsde). 
\end{equation}

In view of {\color{red}\cite[Theorem 3.11]{2310.14728}}, under assumptions (H1') and (H2)-(H5), the above BSDE (\ref{tautruncBSDE}) is uniquely solvable, whose solution $(y^{\tau},u^{\tau})\in \mathcal E\times H_{\nu}^{2,p},$ for each $p\ge 1$, we have, in view of uniqueness, $(y^{\tau},u^{\tau})=(y^{\tau}_{\tau\wedge\cdot},u^{\tau}1_{[0,\tau]}(\cdot))$. Hence, (\ref{tautruncBSDE}) can be rewrite as 
\begin{equation}
\label{tauBSDE}
y_t^\tau=\eta+\int_t^\tau f\left(s, y_s^\tau, u_s^\tau\right) d A_s-\int_t^\tau\int_E u_s^\tau q(dsde). 
\end{equation}
With the help of the well-posedness of (\ref{tautruncBSDE}), the following nonlinear Snell envelope representation for the solution of the reflected BSDE plays an important role in the sequel. The idea is inspired by Peng \cite{Peng_2004} and  Dumitrescu, Quenez and Sulem \cite{Dumitrescu_2014}. See also Djehiche, Dumitrescu and Zeng \cite{Djehiche2021} and Hu, Moreau and Wang \cite{Hu2022b} and the references therein. 
\begin{lemma}
\label{represetation Y}
Under assumptions (H1'),(H2)-(H3), (H4)(b-e), (H4')(a) and (H5), suppose that $(Y, U, K) \in \mathcal{E} \times {H}_{\nu}^{2,p} \times \mathbb{K}^p$, for each $p\ge1$, is a solution to the reflected BSDE (\ref{reflected BSDE}). Then,
$$
Y_t=\underset{\tau \in \mathcal{S}_{t,T}}{\operatorname{ess} \sup } \ \mathcal{E}_{t, \tau}^f\left[\xi \mathbf{1}_{\{\tau=T\}}+L_\tau \mathbf{1}_{\{\tau<T\}}\right] .
$$
\end{lemma}
\begin{proof}
  For any $\tau \in \mathcal{S}_{t,T}$, consider the following RBSDE,
\begin{equation}
\label{reflected tauBSDE}
\left\{\begin{array}{l}
Y_r^{\tau}=Y_{\tau}+\int_r^T 1_{s\le \tau}f_s\left(Y^{\tau}_s, U^{\tau}_s\right) d A_s +\int_r^T d K^{\tau}_s-\int_r^T \int_E U^{\tau}_s(e) q(d s d e), \quad 0 \leq r \leq T, \\
Y_r^{\tau} \geq L_{r\wedge\tau}, \quad 0 \leq r \leq T,\\
\int_0^T\left(Y^{\tau}_{s^{-}}-L_{s\wedge\tau}\right) d K_s=0.
\end{array}\right.
\end{equation}
Then, it is obvious that $(Y^{\tau},U^{\tau},K^{\tau})=(Y_{\cdot\wedge\tau}, U1_{[0,\tau]}(\cdot),K_{\cdot\wedge\tau})$ is a solution to (\ref{reflected tauBSDE}).
Note that $Y_\tau \geq \xi \mathbf{1}_{\{\tau=T\}}+L_\tau \mathbf{1}_{\{\tau<T\}}$. It follows from the comparison Theorem \ref{comparison thm} and Remark \ref{bsde compare} that
\begin{equation}
\label{lowerbdd}
Y^{\tau}_t=Y_t \geq \mathcal{E}_{t, \tau}^f\left[\xi \mathbf{1}_{\{\tau=T\}}+L_\tau \mathbf{1}_{\{\tau<T\}}\right], \quad \forall \tau \in \mathcal{S}_{t,T}
\end{equation}
On the other hand, we define the stopping time $\tau^*=\inf \left\{r \in[t, T]: Y_r=L_r \right\} \wedge T$. Since $Y_r \geq L_r$, $K$ and $L$ are continuous, we find that $\int_t^T\left(Y_{r^-}-L_r\right) d K_r=\int_t^T\left(Y_{r}-L_r\right) d K_r=0$, then we conclude that $K_{\tau^*}=K_t$, which indicates that
$$
Y_t=Y^{\tau^*}_t=Y_{\tau^*}+\int_t^T1_{s\le\tau^*} f\left(s, Y^{\tau^*}_s, U^{\tau^*}_s\right) d A_s-\int_s^{T}\int_E U^{\tau^*}_s(e) q(dsde).
$$
Note that $Y_{\tau^*}=\xi \mathbf{1}_{\left\{\tau^*=T\right\}}+L_{\tau^*} \mathbf{1}_{\left\{\tau^*<T\right\}}$ by the definition of $\tau^*$. It follows that
\begin{equation}
\label{taustar}
Y_t=\mathcal{E}_{t, \tau^*}^f\left[\xi \mathbf{1}_{\left\{\tau^*=T\right\}}+L_{\tau^*} \mathbf{1}_{\left\{\tau^*<T\right\}}\right].
\end{equation}
The proof is then finished by combining (\ref{lowerbdd}) and (\ref{taustar}). 
\end{proof}

\begin{prop}
\label{submartingle property}
Let $(\xi, f, L)$ be a parameter set such that (H1'), (H2)-(H3), (H4)(b-e), (H4')(a) and (H5) hold. If $(Y, U, K)\in \mathcal E\times H^{2,p}\times \mathbb K^p$, for each $p\ge1$, is a solution of the quadratic exponential $R B S D E(\xi, f, L)$,  then it holds $\mathbb P$-a.s. that for each $t \in[0, T]$,
\begin{equation}
\label{bound_Y}
\exp \left\{p \lambda Y_t\right\} \leq \mathbb{E}_t\left[\exp \left\{p \lambda e^{\beta A_T}(\xi^+\vee L_*^+)+p\lambda  \int_t^T e^{\beta A_s}\alpha_s d A_s\right\}\right].
\end{equation}

\end{prop}
\begin{proof} 
In (\ref{tauBSDE}), let $\eta=\xi \mathbf{1}_{\left\{\tau=T\right\}}+L_{\tau} \mathbf{1}_{\left\{\tau<T\right\}}$, then in view of {\color{red}\cite[Theorem 3.11 ]{2310.14728}}, BSDE(\ref{tauBSDE}) is uniquely solvable with solution $(y^{\tau}, u^{\tau})\in \mathcal E\times H_{\nu}^{2,p}$, for each $p\ge 1$.
With the help of {\color{red}\cite[Lemma 3.2]{2310.14728}}, by means of optional sampling theorem and $\tau\le T$, we deduce by the integrability condition on $y^\tau$ and (H4)(d) that
$$
\begin{aligned}
\exp\left\{\lambda y^\tau_t\right\}&\le\exp\left\{e^{\beta A_t}\lambda (y^\tau_t)^+\right\}\le \mathbb E_t\left[ \exp\left\{e^{\beta A_\tau}\lambda (\xi \mathbf{1}_{\{\tau=T\}}+L_\tau \mathbf{1}_{\{\tau<T\}})^++\int_t^\tau\lambda e^{\beta A_t}\alpha_s dA_s\right\}\right]\\
&\le \mathbb E_t\left[ \exp\left\{e^{\beta A_T}\lambda (\xi^+\vee L_*^+)+\int_t^T\lambda e^{\beta A_t}\alpha_s dA_s\right\}\right].
\end{aligned}
$$
Thus, by means of Lemma \ref{represetation Y}, taking essential supremum  of $\tau\in\mathscr{S}_{t,T}$, it turns out that
$$
\exp\left\{\lambda Y_t\right\}\le \mathbb E_t\left[ \exp\left\{e^{\beta A_T}\lambda (\xi^+\vee L_*^+)+\int_t^T\lambda e^{\beta A_t}\alpha_s dA_s\right\}\right].
$$
Hence, for each $p\ge 1$, by Jensen's inequality, it holds that
$$
\exp\left\{p\lambda Y_t\right\}\le \mathbb E_t\left[ \exp\left\{pe^{\beta A_T}\lambda (\xi^+\vee L_*^+)+\int_t^Tp\lambda e^{\beta A_t}\alpha_s dA_s\right\}\right].
$$
\end{proof}

With the help of comparison theorem \ref{comparison thm}, we have the following a priori estimate for $Y$.
\begin{corollary}
\label{sub mart |y|}
Under the assumptions of Proposition \ref{submartingle property}, it holds, for each $p\ge 1$,
\begin{equation}
\label{|y| bound}
\exp\left\{p\lambda |Y_t|\right\}\le \mathbb E_t\left[ \exp\left\{pe^{\beta A_T}\lambda (|\xi|\vee L_*^+)+\int_t^Tp\lambda e^{\beta A_t}\alpha_s dA_s\right\}\right].
\end{equation}
 As a result, with the help of Doob's maximum inequality, $e^{p\lambda Y_*}$, $e^{p\lambda Y^+_*}$ and $e^{p\lambda Y^-_*}$ are uniformly bounded in $S^p$ for each $p>0$. Denote the bound of $e^{pY_*}$ with parameter $\alpha,\,\beta$ in the growth condition (H4)(c) by $\Xi(p,\ \alpha,\ \beta)$.
\end{corollary}
\begin{proof}
Consider BSDE
\begin{equation}
\label{BSDE}
\hat Y_t=\xi+\int_t^T f\left(s,\hat Y_s, \hat U_s\right) d A_s -\int_t^T \int_E \hat U_s(e) q(d s d e), \quad 0 \leq t \leq T.
\end{equation}
Under the assumptions of Proposition \ref{submartingle property}, (\ref{BSDE}) is uniquely solvable with solution $(\hat Y, \hat U)\in \mathcal E\times H_{\nu}^{2,p}$ for each $p\ge 1$.  In view of the comparison theorem \ref{comparison thm} and Remark \ref{bsde compare}, $Y_t\ge \hat Y_t$, $\mathbb P-a.s.$, for each $t\in[0,T]$. Then, $ Y_{t}^-\le \hat Y_t^-$. Thus, with the help of {\color{red}\cite[Lemma 3.2]{2310.14728}}, 
\begin{equation}
\label{Y- bound}
\exp\left\{p\lambda Y_t^-\right\}\le \exp\left\{p\lambda \hat Y_t^-\right\}\le \exp\left\{p\lambda |\hat Y_t|\right\}\le \mathbb E_t\left[ \exp\left\{pe^{\beta A_T}\lambda |\xi|+\int_t^Tp\lambda e^{\beta A_t}\alpha_s dA_s\right\}\right].
\end{equation}
Then, (\ref{|y| bound}) holds by (\ref{bound_Y}) and (\ref{Y- bound}).
\end{proof}

\begin{remark}
\cite[Proposition 2.1]{bayraktar2012quadratic} provides a similar a priori estimate of the component $Y$. However, they assume $\alpha_t$ to be a constant $\alpha$ in the quadratic exponential growth condition and the estimate holds for $Y$ satisfying $\|Y^+\|_{S^\infty}<\infty$. We notice that their argument is based on the well-posedness and comparison theorem of reflected backward ODEs constructed, for example, in Lepeltier and Xu \cite{lepeltier2007reflected}. In this paper, we generalize this result for quadratic exponential BSDEs with jumps and get rid of these additional assumptions via a nonlinear Snell envelope representation.

\end{remark}

\subsection{A priori estimate on $U$ and $K$ }
{\color{blue}
\begin{prop}
\label{priori estimate on U and K}
Let $(\xi,f,L)$ be a parameter such that (H1)-(H4) hold. If $(Y,U,K)$ is a solution of the quadratic RBSDE$(\xi,F,L)$ such that $Y\in\mathcal E$, then for each $p\ge 1$,
\begin{equation}
\label{a pri eq on U,K}
\mathbb E\left[\left(\int_0^T\int_E|U_t(e)|^2\phi_t(de)dA_t\right)^{\frac{p}{2}}+K_T^p\right]\le C_p\mathbb E\left[e^{36p\lambda(1+\beta\|A_T\|_\infty)Y_*}\right]<\infty,
\end{equation}
where $C_p$ is a constant depending on $p$ and the constants in (H1)-(H4).
\end{prop}
}

{\color{blue}
\begin{proof}
 In the sequel, we assume without loss of generality that all local martingales in the derivation are true martingales, otherwise one can take advantage of a standard localization and monotone convergence argument.Define $\underline G_t=-Y_t+\int_0^t\alpha_sdA_s+\int_0^t\beta|Y_s|dA_s$. We first claim that $e^{\lambda\underline G_{\cdot}}$ is a positive submartingale. Indeed, applying It\^o's formula to $e^{\lambda\underline G_t}$, we obtain,
\begin{equation}
\label{ebarG}
\begin{aligned}
de^{\lambda\underline G_t}&=e^{\lambda\underline G_{t^-}}\left(\lambda d \underline G_t+\int_E\left(e^{\lambda U_t(e)}-\lambda U_t(e)-1\right)p(dtde)\right)\\
&=e^{\lambda\underline G_{t^-}}\left(\left[\lambda\beta |Y_{t}| dA_t-\lambda d Y_t+\lambda\alpha_tdA_t\right]+\int_E\left(e^{-\lambda U_t(e)}+\lambda U_t(e)-1\right)p(dtde)\right)\\
&=e^{\lambda\underline G_{t^-}}\left(\left[\lambda\beta |Y_{t}| dA_t+\lambda \left(f(t,Y_t,U_t)dA_t-\int_EU_t(e)q(dtde)+dK_t\right)+\lambda\alpha_tdA_t\right]\right.\\
&\quad\quad\quad\quad\left.+\int_E\left(e^{-\lambda U_t(e)}+\lambda U_t(e)-1\right)p(dtde)\right)\\
&=e^{\lambda\underline G_{t^-}}\left(\left[\lambda\beta |Y_{t}|+\lambda f(t,Y_t,U_t)+\lambda\alpha_t+j_1(-\lambda U_t)\right]dA_t+\lambda dK_t+\int_E\left(e^{-\lambda U_t(e)}-1\right)q(dtde)\right)\\
&\ge  e^{\lambda\underline G_{t^-}}\int_E\left(e^{-\lambda U_t(e)}-1\right)q(dtde) ,
\end{aligned}
\end{equation}
where we make use of the growth condition of $f$ in the last inequality.

Then,
\begin{equation}
\label{compact bar G}
d e^{\lambda\underline G_t}=e^{\lambda\underline G_{t^-}}(d\underline A_t+\lambda dK_t+d\underline M_t).
\end{equation}
where,
$$
\underline M_t=\int_0^t\int_E\left(e^{-\lambda U_s(e)}-1\right)q(dsde),
$$
and $\underline A$ is a non-decreasing process with $\underline A_0=0$, with the form
$$
\underline A_t=\int_0^t\left(\lambda \left(f(t,Y_s,U_s)+\alpha_s+\beta |Y_s|\right)+j_1(-\lambda U_s(e))\right)dA_s.
$$

Now we are going to estimate the quadratic variation of $\underline{M}$ :
$$
d[\underline{M}]_t=\int_E\left(e^{-\lambda U_t(e)}-1\right)^2 \phi_t\left(de\right) d A_t+\int_E\left(e^{ -\lambda U_t(e)}-1\right)^2 q(dtde).
$$
Obviously, by (\ref{compact bar G}),
$$
d[e^{\lambda\underline{G}}]_t=e^{2 \lambda\underline{G}_{t-}} d[\underline{M}]_t.
$$
We also find the predictable quadratic variations by direct calculation, 
$$
d\langle \underline M\rangle_t=\int_E\left(e^{- \lambda U_t(e)}-1\right)^2 \phi_t\left(de\right) d A_t,
$$
and
$$
d\langle e^{\lambda\underline G}\rangle_t= e^{2 \lambda\underline{G}_{t-}} d\langle\underline{M}\rangle_t.
$$

Then, for any stopping time $\sigma \leq T$, it holds that
\begin{equation}
\label{M-qv}
\langle\underline{M}\rangle_{T}-\langle\underline{M}\rangle_\sigma =\int_{\sigma}^T \frac{d\langle e^{\lambda\underline{G}}\rangle_t}{e^{2 \lambda\underline{G}_t-}}
\leq \sup _{0 \leq t \leq T}\left(e^{-2 \lambda\underline{G}_t}\right)\left(\langle e^{\lambda\underline{G}}\rangle_{T}-\langle e^{\lambda\underline{G}}\rangle_\sigma\right).
\end{equation}
Next, we find a priori estimate of $\langle e^{\lambda\underline{G}}\rangle_T-\langle e^{\lambda\underline{G}}\rangle_\sigma$ via It\^o's formula,
\begin{equation}
\label{ito}
d e^{2\lambda\underline G_t}=2 e^{2\lambda \underline{G}_t-}\left(d \underline{M}_t+d \underline{A}_t+\lambda dK_t\right)+d[e^{\lambda\underline{G}}]_t\ge 2 e^{2\lambda \underline{G}_t-}d \underline{M}_t+e^{2\lambda \underline{G}_t-}\int_E\left(e^{-\lambda U_t(e)}-1\right)^2 q(dtde)+d\langle e^{\lambda\underline{G}}\rangle_t.
\end{equation}
Taking conditional expectation on both sides of (\ref{ito}), we obtain,
$$
\mathbb{E}\left[\langle e^{\lambda\underline{G}}\rangle_T-\langle e^{\lambda\underline{G}}\rangle_\sigma \mid \mathcal G_\sigma\right]\leqslant \mathbb{E}\left[e^{2\lambda \underline{G}_T}-e^{2 \lambda\underline{G}_\sigma} \mid \mathcal G_\sigma\right]
\leq \mathbb{E}\left[e^{2 \lambda\underline{G}_T} 1_{\sigma<T} \mid \mathcal G_\sigma\right].
$$
Then, making use of Garsia-Neveu Lemma, see for example \cite[Lemma 4.3]{barrieu2013monotone}, it turns out that for each $p\ge 1$,
$$
\mathbb{E}\left[\left(\langle e^{\lambda\underline{G}}\rangle_T\right)^p\right] \leqslant p^p \mathbb{E}\left[e^{2 p\lambda \underline{G}_T}\right].
$$
Then, by (\ref{M-qv}),
\begin{equation}
\label{bar M}
\begin{aligned}
\mathbb{E}\left[\left(\langle \underline{M}\rangle_T\right)^p\right] & \leqslant \mathbb{E}\left[\sup _{0 \leq t \leqslant T}\left(e^{-2 p\lambda \underline{G}_t}\right)\left(\langle e^{\lambda\underline G}\rangle_T\right)^p\right] \\
& \leqslant\left(\mathbb{E}[\sup _{0 \leq t \leq T} e^{-4 p\lambda\underline{G}_t}]\right)^{\frac{1}{2}}\left(\mathbb{E}\left[\left(\langle e^{\lambda\underline{G}}\rangle_T\right)^{2 p}\right]\right)^{\frac{1}{2}} \\
& \leqslant\left(\mathbb{E}\left[\sup _{0 \leq t \leq T} e^{-4 p\lambda \underline{G} _t}\right]\right)^{\frac{1}{2}} \cdot(2 p)^p\left(\mathbb{E}\left[e^{4 p \lambda\underline{G}_T}\right]\right)^{\frac{1}{2}} \\
& \leqslant(2 p)^p \mathbb{E}\left[\sup _{0 \leq t \leq T} e^{4 p\lambda\widetilde{G_t}}\right]\\
&\le C_p\mathbb E\left[e^{8p\lambda(1+\beta\|A_T\|_\infty)Y_*}\right],
\end{aligned}
\end{equation}
where
$
\widetilde{G_t}=|Y_t|+\int_0^t\alpha_sdA_s+\int_0^t\beta|Y_s|dA_s,
$
and $C_p$ is a positive constant depending on $p$. That is to say, for each $p\ge 1$,
\begin{equation}
\label{e^-U}
\mathbb E\left[\left(\int_0^T\int_E\left(e^{-\lambda U_t(e)}-1\right)^2\phi_t(de)dA_t\right)^p\right]\le C_p\mathbb E\left[e^{8p\lambda(1+\beta\|A_T\|_\infty)Y_*}\right]<\infty.
\end{equation}

Next, applying It\^{o}'s formula to $e^{-\lambda Y_t}$, we deduce that
$$
\begin{aligned}
e^{-\lambda Y_{T}}-e^{-\lambda Y_0} &= \int_0^{T}\lambda e^{-\lambda Y_{s^-}}f(s,Y_s,U_s)dA_s+\int_0^{T}\int_Ee^{-\lambda Y_{s^-}}(e^{-\lambda U_s(e)}+\lambda U_s(e)-1)\phi_s(de)dA_s\\
&\quad +\int_0^{T}e^{-\lambda Y_{t^-}}\int_E(e^{-\lambda U(e)}-1)q(dtde)+\int_0^{T}\lambda e^{-Y_{t^-}}dK_t\le e^{\lambda Y^-_*}.
\end{aligned}
$$
With the help of the growth condition of $f$, it turns out that:
$$
\begin{aligned}
&\int_0^{T}\lambda e^{-\lambda Y_{t^-}}(-\alpha_t-\beta|Y_t|-j_\lambda (t,-U_t))dA_t+\int_0^{T}\int_Ee^{-\lambda Y_{t^-}}(e^{-\lambda U_t(e)}+\lambda U_t(e)-1)\phi_t(de)dA_t\\
&+\int_0^{T}e^{-\lambda Y_{t^-}}\int_E(e^{-\lambda U_t(e)}-1)q(dtde)+\int_0^{T}\lambda e^{-\lambda Y_{t^-}}dK_t\le e^{\lambda Y^-_*}.
\end{aligned}
$$
Equivalently,
\begin{equation}
\label{estimate_K_U}
\lambda e^{-\lambda Y^+_*}K_{T}\le\int_0^{T}\lambda e^{-\lambda Y_{t^-}}dK_t
\le\int_0^{T}\lambda e^{-\lambda Y_{t^-}}(\alpha_s+\beta|Y_s|)dA_s+ e^{\lambda Y^-_*}-\int_0^{T}e^{-\lambda Y_{t^-}}\int_E(e^{-\lambda U_t(e)}-1)q(dtde).
\end{equation}
Note that $\int_0^{T}\lambda e^{-\lambda Y_{t^-}}(\alpha_s+\beta|Y_s|)dA_s+e^{\lambda Y^-_*}\le \lambda e^{\lambda Y^-_*}\int_0^T\alpha_s dA_s+\beta e^{2\lambda Y_*}\|A_T\|_\infty+e^{\lambda Y^-_*}\le \tilde C e^{2\lambda Y_*}$, for some positive random variable $\tilde C\in L^p$ for each $p\ge 1$.
With the help of Burkholder-Davis-Gundy inequality and making use of (\ref{bar M}) for $p=2$, we obtain, for each $p\ge1$,
\begin{equation}
\label{K2estimate}
\begin{aligned}
\mathbb E\left[\lambda^2 e^{-2p\lambda Y^+_*}K_{T}^2\right]\le\mathbb E\left[\lambda^2 e^{-2\lambda Y^+_*}K_{T}^2\right]&\le 2\mathbb E\left[\tilde C^2e^{4\lambda Y_*}\right]+2\mathbb E\left[\left|\int_0^{T}e^{-\lambda Y_{t^-}}\int_E(e^{-\lambda U_t(e)}-1)q(dtde)\right|^2\right]\\
&\le  2\mathbb E\left[\tilde C^2e^{4\lambda Y_*}\right]+2C\mathbb E\left[\int_0^{T}e^{-2\lambda Y_{t^-}}\int_E(e^{-\lambda U_t(e)}-1)^2\phi_t(de)dA_t\right]\\
&\le 2\mathbb E\left[\tilde C^2e^{4\lambda Y_*}\right]+2C\mathbb E\left[e^{2\lambda Y_*}\int_0^{T}\int_E(e^{-\lambda U_t(e)}-1)^2\phi_t(de)dA_t\right]\\
&\le 2\mathbb E\left[\tilde C^2e^{4\lambda Y_*}\right]+C\mathbb E\left[e^{4\lambda Y_*}\right]+C\mathbb E\left[\left(\int_0^{T}\int_E(e^{-\lambda U_t(e)}-1)^2\phi_t(de)dA_t\right)^2\right]\\
&\le 2\mathbb E\left[\tilde C^2e^{4\lambda Y_*}\right]+C\mathbb E\left[e^{4\lambda Y_*}\right]+C\mathbb E\left[e^{16\lambda(1+\beta\|A_T\|_\infty)Y_*}\right]\\
&\le C\mathbb E\left[e^{16\lambda(1+\beta\|A_T\|_\infty)Y_*}\right],
\end{aligned}
\end{equation}
where $C$ is a positive constant depending only on the constants given in the assumptions and differing from line to line.

\color{red}
Similarly, define $\bar G_t=Y_t+\int_0^t\alpha_sdA_s+\int_0^t\beta|Y_s|dA_s$ and for each $p\ge 1$, applying It\^o's formula to $e^{p\lambda\bar G_t}$, we obtain,
\begin{equation}
\label{ebarG2}
\begin{aligned}
de^{p\lambda\bar G_t}&=e^{p\lambda\bar G_{t^-}}\left(p\lambda d \bar G_t+\int_E\left(e^{p\lambda U_t(e)}-p\lambda U_t(e)-1\right)p(dtde)\right)\\
&=e^{p\lambda\bar G_{t^-}}\left(\left[p\lambda\beta |Y_{t}| dA_t+p\lambda d Y_t+p\lambda\alpha_tdA_t\right]+\int_E\left(e^{p\lambda U_t(e)}-p\lambda U_t(e)-1\right)p(dtde)\right)\\
&=e^{p\lambda\bar G_{t^-}}\left(\left[p\lambda\beta |Y_{t}| dA_t+p\lambda \left(-f(t,Y_t,U_t)dA_t+\int_EU_t(e)q(dtde)- dK_t\right)+\lambda\alpha_tdA_t\right]\right.\\
&\quad\quad\quad\left.+\int_E\left(e^{p\lambda U_t(e)}-p\lambda U_t(e)-1\right)p(dtde)\right)\\
&=e^{p\lambda\bar G_{t^-}}\left(\left[p\lambda\beta |Y_{t}|-p\lambda f(t,Y_t,U_t)+p\lambda\alpha_t+j_1(p\lambda U_t)\right]dA_t+\int_E\left(e^{p\lambda U_t(e)}-1\right)q(dtde)-p\lambda dK_t\right)\\
&\ge e^{p\lambda\bar G_{t^-}}\left(\left[-pj_{\lambda}(U_t)+j_1(p\lambda U_t)\right]dA_t+\int_E\left(e^{p\lambda U_t(e)}-1\right)q(dtde)-p\lambda dK_t\right)\\
&\ge  e^{p\lambda\bar G_{t^-}}\int_E\left(e^{p\lambda U_t(e)}-1\right)q(dtde)-p\lambda e^{p\lambda\bar G_{t^-}}dK_t ,
\end{aligned}
\end{equation}
where we make use of the growth condition of $f$ in the first inequality and the last inequality follows from the fact that $j_\lambda(ku)\ge kj_\lambda(u)$, for each $k\ge 1$.

Then,
\begin{equation}
\label{compact bar G2}
d e^{p\lambda\bar G_t}=e^{p\lambda\bar G_{t^-}}(d\bar A_t+d\bar M_t-p\lambda dK_t),
\end{equation}
where
$$
\bar M_t=\int_0^t\int_E\left(e^{p\lambda U_s(e)}-1\right)q(dsde),
$$
and $\bar A$ is a non-decreasing process with $\bar A_0=0$, with the form
$$
\bar A_t=\int_0^t\left(p\lambda \left(-f(t,Y_s,U_s)+\alpha_s+\beta |Y_s|\right)+j_1(p\lambda U_s(e))\right)dA_s.
$$

Now we are going to estimate the quadratic variation of $\bar{M}$ :
$$
d[\bar{M}]_t=\int_E\left(e^{p\lambda U_t(e)}-1\right)^2 \phi_t\left(de\right) d A_t+\int_E\left(e^{p \lambda U_t(e)}-1\right)^2 q(dtde).
$$
Obviously, by (\ref{compact bar G2}),
$$
d[e^{p\lambda\bar{G}}]_t=e^{2 p\lambda\bar{G}_{t-}} d[\bar{M}]_t.
$$
We also find the predictable quadratic variations by direct calculation, 
$$
d\langle \bar M\rangle_t=\int_E\left(e^{ p\lambda U_t(e)}-1\right)^2 \phi_t\left(de\right) d A_t,
$$
and
$$
d\langle e^{p\lambda\bar G}\rangle_t= e^{2 p\lambda\bar{G}_{t-}} d\langle\bar{M}\rangle_t.
$$

Then, for any stopping time $\sigma \leq T$, it holds that
\begin{equation}
\label{M-qv2}
\langle\bar{M}\rangle_{T}-\langle\bar{M}\rangle_\sigma =\int_{\sigma}^T \frac{d\langle e^{p\lambda\bar{G}}\rangle_t}{e^{2p \lambda\bar{G}_t-}}
\leq \sup _{0 \leq t \leq T}\left(e^{-2 p\lambda\bar{G}_t}\right)\left(\langle e^{p\lambda\bar{G}}\rangle_{T}-\langle e^{p\lambda\bar{G}}\rangle_\sigma\right).
\end{equation}
Next, we find a priori estimate of $\langle e^{p\lambda\bar{G}}\rangle_T-\langle e^{p\lambda\bar{G}}\rangle_\sigma$ via It\^o's formula to $e^{2p\lambda\bar G}$,
\begin{equation}
\label{ito2}
d e^{2p\lambda\bar G_t}+2p\lambda e^{2p\lambda \bar{G}_t-}dK_t=2 e^{2p\lambda \bar{G}_t-}\left(d \bar{M}_t+d \bar{A}_t\right)+d[e^{p\lambda\bar{G}}]_t\ge 2 e^{2p\lambda \bar{G}_t-}d \bar{M}_t+e^{2p\lambda \bar{G}_t-}\int_E\left(e^{p\lambda U_t(e)}-1\right)^2 q(dtde)+d\langle e^{p\lambda\bar{G}}\rangle_t.
\end{equation}
Taking conditional expectation on both sides of (\ref{ito2}), we obtain,
$$
\begin{aligned}
\mathbb{E}\left[\langle e^{p\lambda\bar{G}}\rangle_T-\langle e^{p\lambda\bar{G}}\rangle_\sigma \mid \mathcal G_\sigma\right]&\leqslant \mathbb{E}\left[e^{2p\lambda \bar{G}_T}+\int_0^T 2p\lambda e^{2p\lambda \bar{G}_t-}dK_t-e^{2 p\lambda\bar{G}_\sigma}-\int_0^{\sigma} 2p\lambda e^{2p\lambda \bar{G}_t-}dK_t \mid \mathcal G_\sigma\right]\\
&\leq \mathbb{E}\left[\left(e^{2p \lambda\bar{G}_T}+\int_0^T 2p\lambda e^{2p\lambda \bar{G}_t-}dK_t\right) 1_{\sigma<T} \mid \mathcal G_\sigma\right].
\end{aligned}
$$
Then, making use of Garsia-Neveu Lemma again, it turns out that
$$
\mathbb{E}\left[\left(\langle e^{p\lambda\bar{G}}\rangle_T\right)^{3/2}\right] \leqslant (\frac{3}{2})^{3/2} \mathbb{E}\left[\left(e^{2p \lambda\bar{G}_T}+\int_0^T 2p\lambda e^{2p\lambda \bar{G}_t-}dK_t\right)^{3/2}\right].
$$
Then, by (\ref{M-qv2}) and (\ref{K2estimate}),
\begin{equation}
\begin{aligned}
\mathbb{E}\left[\langle \bar{M}\rangle_T\right] & \leqslant \mathbb{E}\left[\sup _{0\leqslant t \leqslant T}\left(e^{-2p \lambda \bar{G}_t}\right)\langle e^{p\lambda\bar G}\rangle_T\right] \\
& \leqslant\left(\mathbb{E}[\sup _{0\leq t \leq T} e^{-6 p\lambda\bar{G}_t}]\right)^{\frac{1}{3}}\left(\mathbb{E}\left[\left(\langle e^{p\lambda\bar{G}}\rangle_T\right)^{3/2}\right]\right)^{\frac{2}{3}} \\
& \leqslant \frac{3}{2}\left(\mathbb{E}\left[\sup _{0\leq t \leq T} e^{-6p\lambda \bar{G} _t}\right]\right)^{\frac{1}{3}}\left(\mathbb{E}\left[\left(e^{2p \lambda\bar{G}_T}+\int_0^T 2p\lambda e^{2p\lambda \bar{G}_t-}dK_t\right)^{3/2}\right]\right)^{\frac{2}{3}}\\
& \leqslant C\left(\mathbb{E}\left[\sup _{0\leq t \leq T} e^{-6p\lambda \bar{G} _t}\right]\right)^{\frac{1}{3}}\left(\mathbb{E}\left[e^{3p \lambda\bar{G}_T}+\left(\int_0^T 2p\lambda e^{2p\lambda \bar{G}_t-}dK_t\right)^{3/2}\right]\right)^{\frac{2}{3}}\\
& \leqslant C_p\left(\mathbb{E}\left[\sup _{0 \leq t \leq T} e^{-6p\lambda \bar{G} _t}\right]\right)^{\frac{1}{3}}\left(\mathbb{E}\left[e^{3p \lambda\bar{G}_T}+e^{3p\lambda \bar{G}_*}K_T^{3/2}\right]\right)^{\frac{2}{3}}\\
 & = C_p\left(\mathbb{E}\left[\sup _{0 \leq t \leq T} e^{-6p\lambda \bar{G} _t}\right]\right)^{\frac{1}{3}}\left(\mathbb{E}\left[e^{3 p\lambda\bar{G}_T}+e^{3p\lambda \bar{G}_*+\frac{3}{2}p\lambda Y^+_*}e^{-\frac{3}{2}p\lambda Y^+_*}K_T^{3/2}\right]\right)^{\frac{2}{3}}\\
  & \leqslant C_p\left(\mathbb{E}\left[\sup _{0 \leq t \leq T} e^{-6p\lambda \bar{G} _t}\right]\right)^{\frac{1}{3}}\left(\mathbb{E}\left[e^{3p \lambda\bar{G}_T}\right]+\left(\mathbb E\left[e^{12p\lambda \bar{G}_*+6p\lambda Y^+_*}\right]\right)^{1/4}\left(\mathbb E\left[e^{-2p\lambda Y^+_*}K_T^2\right]\right)^{3/4}\right)^{\frac{2}{3}}\\
  & \leqslant C_p\left(\mathbb{E}\left[\sup _{0 \leq t \leq T} e^{-6p\lambda \bar{G} _t}\right]\right)^{\frac{1}{3}}\left(\mathbb{E}\left[e^{3p \lambda\bar{G}_T}\right]+\left(\mathbb E\left[e^{12p\lambda \bar{G}_*+6p\lambda Y^+_*}\right]\right)^{1/4}\left(\mathbb E\left[e^{16\lambda(1+\beta\|A_T\|_\infty)Y_*}\right]\right)^{3/4}\right)^{\frac{2}{3}}\\
&\le C_p\mathbb E\left[e^{36p\lambda(1+\beta\|A_T\|_\infty)Y_*}\right]<\infty,
\end{aligned}
\end{equation}
where $C$ is a  positive constant depending only on the constants given in the assumptions and $C_p$ is a  positive constant also depending on $p$ however  varies from line to line. 
That is to say,
\begin{equation}
\label{e^U0}
\mathbb E\left[\int_0^T\int_E\left(e^{p\lambda U_t(e)}-1\right)^2\phi_t(de)dA_t\right]\le C_p\mathbb E\left[e^{36p\lambda(1+\beta\|A_T\|_\infty)Y_*}\right]<\infty.
\end{equation}
\textcolor{blue}{We want to  clarify  that in fact, the constant $\frac{3}{2}$ in the derivation above does not matter too much and can be replaced by any $1<q<2$.}

Making use of the inequality $(x-1)^{2p}\le (x^p-1)^2$ for $x>0$ and $p\ge 1$, we observe that for $p\ge 1$,
\begin{equation}
\mathbb E\left[\int_0^T\int_E\left(e^{\lambda U_t(e)}-1\right)^{2p}\phi_t(de)dA_t\right]\le\mathbb E\left[\int_0^T\int_E\left(e^{p\lambda U_t(e)}-1\right)^2\phi_t(de)dA_t\right]\le C_p\mathbb E\left[e^{36p\lambda(1+\beta\|A_T\|_\infty)Y_*}\right]<\infty.
\end{equation}
Thus, with the help of H\"older's inequality, we obtain, for each $p\ge 1$,
\begin{equation}
\label{e^U}
\mathbb E\left[\left(\int_0^T\int_E\left(e^{\lambda U_t(e)}-1\right)^{2}\phi_t(de)dA_t\right)^p\right]\le C_p\mathbb E\left[\int_0^T\int_E\left(e^{\lambda U_t(e)}-1\right)^{2p}\phi_t(de)dA_t\right]\le C_p\mathbb E\left[e^{36p\lambda(1+\beta\|A_T\|_\infty)Y_*}\right]<\infty.
\end{equation}
Then, combining (\ref{e^-U}) and (\ref{e^U}), we conclude that for each $p\ge 1$,
$$
\mathbb E\left[\left(\int_0^T\int_E|U_t(e)|^2\phi_t(de)dA_t\right)^{\frac{p}{2}}\right]\le\left(E\left[\left(\int_0^T\int_E|U_t(e)|^2\phi_t(de)dA_t\right)^p\right]\right)^{1/2}\le C_p\mathbb E\left[e^{36p\lambda(1+\beta\|A_T\|_\infty)Y_*}\right]<\infty.
$$

We finish the proof by showing that for each $p\ge 1$, $\mathbb E[K_T^p]<\infty$. 
Note that 
$K_T=Y_0-\xi-\int_0^T f\left(s,Y_s, U_s\right) d A_s +\int_0^T \int_E U_s(e) q(d s, d e)$, then,
$$
\begin{aligned}
\mathbb{E}\left[K_T^p\right] \leq & C_p \mathbb{E}\left[Y_*^p+\left(\int_0^T \alpha_t d A_t+\int_0^T \beta\left|Y_t\right| d A_t+\frac{1}{\lambda} \int_0^T \left(j_\lambda\left(U_t\right)+j_\lambda\left(-U_t\right)\right) d A_t\right)^p+\left|\int_0^T \int_E U_t(e)q(dtde)\right|^p\right] \\
\leq & C_p \mathbb{E}\left[1 +Y_*^p+\left(\int_0^T \int_E\left(e^{\lambda U_t(e)}-1+e^{-\lambda U_t(e)}-1\right) \phi_t(d e) d A_t\right)^p +\left|\int_0^T \int_E U_t(e)q(dtde) \right|^p\right]\\
\leq & C_p \mathbb{E}\left[1 +Y_*^p+\left(\int_0^T \int_E\left(e^{\lambda U_t(e)}-1\right)^2\phi_t(d e) d A_t\right)^{p/2}+\left(\int_0^T\int_E\left(e^{-\lambda U_t(e)}-1\right)^2 \phi_t(d e) d A_t\right)^{p/2}\right.\\
&\quad\quad\quad\left.+\left|\int_0^T \int_E U_t(e)q(dtde)\right|^p\right]\\
\le & C_p\mathbb E\left[e^{36p\lambda(1+\beta\|A_T\|_\infty)Y_*}\right]+C_p\mathbb E\left[\left|\int_0^T \int_E U_t(e)q(dtde)\right|^p\right],
\end{aligned}
$$
where we make use of (\ref{e^-U}), (\ref{e^U}) and H\"older's inequality in the last inequality. To deal with the second term, we take advantage of a generalized Burkholder-Davis-Gundy inequality (\cite[Theorem 2.1]{Hern_ndez_Hern_ndez_2022}) and obtain,
$$
\begin{aligned}
\mathbb E\left[\left|\int_0^T \int_E U_t(e)q(dtde)\right|^p\right]&\le C_p\mathbb E\left[\left(\left[\int_0^\cdot\int_E U_t(e)q(dtde)\right]_T\right)^{p/2}\right]\\
&\le C_p\mathbb E\left[\max\left\{\left(\int_0^T\int_E|U_t(e)|^2\phi_t(de)dA_t\right)^{p/2},\int_0^T\int_E|U_t(e)|^p\phi_t(de)dA_t\right\}\right]\\
&\le C_p\mathbb E\left[\max\left\{\left(\int_0^T\int_E(e^{\lambda|U_t(e)|}-1)\phi_t(de)dA_t\right)^{p/2},\int_0^T\int_E(e^{\lambda|U_t(e)|}-1)\phi_t(de)dA_t\right\}\right]\\
&\le C_p\mathbb E\left[\left(\int_0^T\int_E(e^{\lambda|U_t(e)|}-1)^2\phi_t(de)dA_t\right)^{p/4}+\left(\int_0^T\int_E(e^{\lambda|U_t(e)|}-1)^2\phi_t(de)dA_t\right)^{1/2}\right]\\
&\le  C_p\mathbb E\left[e^{36p\lambda(1+\beta\|A_T\|_\infty)Y_*}\right],
\end{aligned}
$$
where we make use of (\ref{e^-U}), (\ref{e^U}) and H\"older's inequality again in the last inequality. Thus, we conclude that for each $p\ge1$, 
$$
\mathbb{E}\left[K_T^p\right] \leq C_p\mathbb E\left[e^{36p\lambda(1+\beta\|A_T\|_\infty)Y_*}\right]<\infty,
$$
and end the proof.

\end{proof}
}

\begin{remark}
Based on Garsia-Neveu Lemma, \cite[Proposition 4.5]{karoui2016quadratic} presents a priori estimation for quadratic variation of quadratic exponential semimartingales and then constructs existence of quadratic BSDEs with jumps. However, our proof for RBSDEs is quite different and more complicated as a result of the additional component $K$. 
\end{remark}

\section{ Existence of convex / concave quadratic exponential RBSDEs with bounded terminal and obstacle}
\label{section existence for bdd}

We are at the position to show the existence of quadratic RBSDEs. Before proceeding with the proof, we present the following lemma which provides essential properties of the auxiliary drivers. Remerber that without special announcement, we always assume that $f$ is convex  with respect to $u$.  The concave generator scenario can be handled with a parallel discussion similar as in \cite{briand2008quadratic}. For $t\in[0,T]$,
on  $(y, u) \in \mathbb R \times L^2(E,\mathcal{B}(E),\phi_t(\omega,dy))$, define a set of auxiliary generators $\left( f^{n}\right)_n$ as follows.
$$
\begin{aligned}
{f}^{n}(t, y, u)=\inf _{ r\in L^{2}(E,\mathcal{B}(E),\phi_t(\omega,dy))} \left\{f(t, y, r)+n \|u-r\|_t\right\}.
\end{aligned}
$$
The properties of the auxiliary drivers read as follows. Refer to {\color{red}\cite[Lemma 3.8]{2310.14728}} for the proof of Lemma \ref{lemma monotonicity}.

\begin{lemma}
\label{lemma monotonicity}
Under the assumptions (H1'), (H2)-(H5),

(i) The sequence $\{{f}^{n}\}_n$  are globally Lipschitz with respect to (y, u) in $\mathbb R \times L^2(E,\mathcal{B}(E),\phi_t(\omega,dy))$. 

(ii) 
The sequence $\{f^{n}\}_n$ is convex with respect to $u$ if $f$ is convex with respect to $u$, for $u\in L^2(E,\mathcal{B}(E),\phi_t(\omega,dy))$.

(iii)
For $t\in[0,T]$,  the sequence $\{{f}^{n}\}_n$ converges to ${f}$ on $(y, u) \in \mathbb R \times L^2(E,\mathcal{B}(E),\phi_t(\omega,dy))$. 

(iv) For $n>C_0$,
\begin{equation}
\label{fn bound}
-3\alpha_t-3\beta|y|-\frac{1}{\lambda}j_\lambda(t,-u)\leq f^{n}(t,y,u)\leq f(t,y,u) \leq \alpha_t+\beta|y|+\frac{1}{\lambda}j_\lambda(t,u)\leq 3\alpha_t+3\beta|y|+\frac{1}{\lambda}j_\lambda(t,u).
\end{equation}
\end{lemma}

\begin{remark}
For $f$ concave in $u$, the auxiliary generators should be defined as   
$$
\begin{aligned}
\tilde {f}^{n}(t, y, u)=\sup _{ r\in L^{2}(E,\mathcal{B}(E),\phi_t(\omega,dy))} \left\{f(t, y, r)-n \|u-r\|_t\right\}.
\end{aligned}
$$
Similar properties hold by a parallel argument.
\end{remark}

Thanks to the properties of the auxiliary generators, we are ready to construct the solution of RBSDE (\ref{reflected BSDE}) with bounded terminal and bounded obstacle. Notice that in view of (\ref{fn bound}), the parameter $(\alpha,\ \beta)$ in the  a priori estimates (\ref{|y| bound}) and (\ref{priori estimate on U and K}) are replaced by $(3\alpha,\ 3\beta)$ when estimating solutions of RBSDE($\xi,\ f^n,\ L$).

\begin{thm}
\label{thm Quadratic bounded}
Assume that assumptions (H1'), (H2)-(H3), (H4)(b-e), (H4')(a) and (H5) are fulfilled.  Then the RBSDE (\ref{reflected BSDE}) admits  a unique  solution $(Y, U, K) \in \mathcal E\times H^{2,p}\times \mathbb K^p$, for all $p\ge 1$.
\end{thm}
The uniqueness of Theorem \ref{thm Quadratic bounded} is inherited from Corollary \ref{uniqueness}, we turn to show the existence part of Theorem \ref{thm Quadratic bounded}.

\begin{proof}[Proof of the existence part of Theorem \ref{thm Quadratic bounded}]
With the help of Lemma \ref{lemma monotonicity} and Theorem \ref{thm_lip}, for $n>C_0$, there exists a unique solution $(Y^n,U^n, K^n)\in S^2\times H_{\nu}^{2,2}\times \mathbb K^2$ for RBSDE($\xi,\ f^n,\ L$).  To construct uniform a priori estimates for RBSDE($\xi,\ f^n,\ L$),  for $n>C_0$, we first  consider BSDE ($\xi,\ f^{n,k},\ L$), where $f^{n,k}=(f^n\wedge -k)\vee k$, for $k\in\mathbb N$. It can be easily checked that for each $k\ge 1$, $t\in[0,T]$ and $(y,\ y',\ u,\ u') \in \mathbb R \times \mathbb R\times  L^2(E,\mathcal{B}(E),\phi_t(\omega,dy))\times L^2(E,\mathcal{B}(E),\phi_t(\omega,dy))$,
$$
|f^{n,k}(t,y,u)-f^{n,k}(t,y',u')|\le|f^n(t,y,u)-f^n(t,y',u')|\le \tilde \beta|y-y'|+n\|u-u'\|_t.
$$
Thus, making use of assertion (ii) of  Theorem \ref{thm_lip}, for $n>C_0$, there exists a unique solution $(Y^{n,k},U^{n,k}, K^{n,k})\in S^\infty \times \mathbb J^\infty\times \mathbb K^2$ for RBSDE($\xi,\ f^{n,k},\ L$). 
With the help of the following facts,
$$
f^{n,k}(t,y,u)\le( f^{n,k})^+(t,y,u)\le (f^{n})^+(t,y,u)\le 3\alpha_t+3\beta|y|+\frac{1}{\lambda}j_{\lambda}(t,u),   
$$
and
$$
f^{n,k}(t,y,u)\ge-( f^{n,k})^-(t,y,u)\ge -(f^{n})^-(t,y,u)\ge -3\alpha_t-3\beta|y|-\frac{1}{\lambda}j_{\lambda}(t,-u),
$$
where $(f^{n,k})^+=f^{n,k}1_{\{f^{n,k}\ge 0\}}$, and $(f^{n,k})^-=-f^{n,k}1_{\{f^{n,k}< 0\}}$,
we deduce that
\begin{equation}
\label{fnk bound}
-3\alpha_t-3\beta|y|-\frac{1}{\lambda}j_\lambda(t,-u)\leq f^{n,k}(t,y,u)\leq 3\alpha_t+3\beta|y|+\frac{1}{\lambda}j_\lambda(t,u).
\end{equation}
Thus, in view of the a priori estimate (\ref{|y| bound}) and $Y^{n,k}\in S^\infty$, we conclude that for each $n>C_0$ and $k\ge 1$,
\begin{equation}
\label{Ykn}
\exp \left\{p \lambda\left|Y^{n,k}_t\right|\right\} \leq \mathbb{E}_t\left[\exp \left\{p \lambda e^{3\beta A_T}(|\xi|\vee L_*^+)+p\lambda  \int_t^T e^{3\beta A_s}3\alpha_s d A_s\right\}\right].
\end{equation}
The remaining proof is broken down into five steps. In the first step, we find the convergence of $\{Y^{n,k}\}$ to $Y^n$ and show that $Y^n\in\mathcal E$, which contributes to a uniform estimate for $\{Y^n\}$. Next we construct a candidate solution $Y^0$ by the comparison theorem and a candidate solution $U^0$ is also found in step 2.
In the third step, we derive an a priori estimate of $|Y^n-Y^m|$. The  regularity of $Y^0$ is proved via this estimate with the help of convergence in $S^1$ of the sequence $Y^n$ in the forth step. In the last step, we find a condidate solution $K^0$ and verify that $(Y^0, U^0, K^0)$ is truly a solution in appropriate spaces.

\textbf{Step 1: The convergence of the sequence $\{Y^{n,k}\}_k$ to $Y^n$.}

In this step, we are going to show that $\lim_{k\to\infty}\mathbb E\left[|Y^n_t-Y^{n,k}_t|^2\right]=0$. By It\^o's formula,
\begin{equation}
\label{ynk-yn}
\begin{aligned}
d(Y^{n,k}_t-Y^n_t)^2 &=2(Y_t^{n,k}-Y_t^{n})d(Y_t^{n,k}-Y_t^n)+\int_E|U_t^{n,k}-U_t^n|^2p(dtde)\\
&=2(Y_t^{n,k}-Y_t^n)\left(-f^{n,k}(t,Y_t^{n,k},U_t^{n,k})+f^{n}(t,Y_t^n,U_t^n)\right)dA_t+2(Y_t^{n,k}-Y_t^n)\int_E(U_t^{n,k}-U_t^n)q(dtde)\\
&\quad+\int_E\left|U_t^{n,k}-U_t^n\right|^2q(dtde)+\int_E\left|U_t^{n,k}-U_t^n\right|^2\phi_t(de)dA_t-2(Y_t^{n,k}-Y_t^n)(dK_t^{n,k}-dK_t^n).
\end{aligned}
\end{equation}
We claim first that for each $t\in[0,T]$,
\begin{equation}
\label{dk=0}
\mathbb E\left[\int_{t}^T2(Y_s^{n,k}-Y_s^n)(dK_s^{n,k}-dK_s^n)\right]\le 0.
\end{equation}
Indeed,
$$
\begin{aligned}
\mathbb E\left[\int_{t}^T(Y_s^{n,k}-Y_s^n)(dK_s^{n,k}-dK_s^n)\right]&=\mathbb E\left[\int_{t}^T(Y_s^{n,k}-L_s+L_s-Y_s^n)(dK_s^{n,k}-dK_s^n)\right]\\
&=\mathbb E\left[\int_{t}^T(Y_s^{n,k}-L_s)(dK_s^{n,k}-dK_s^n)\right]+\mathbb E\left[\int_{t}^T(Y_s^{n}-L_s)(dK_s^{n}-dK_s^{n,k})\right]\\
&\le \mathbb E\left[\int_{t}^T(Y_s^{n,k}-L_s)dK_s^{n,k}\right]+\mathbb E\left[\int_{t}^T(Y_s^{n}-L_s)dK_s^{n}\right]\\
&\le \mathbb E\left[\int_{0}^T(Y_s^{n,k}-L_s)dK_s^{n,k}\right]+\mathbb E\left[\int_{0}^T(Y_s^{n}-L_s)dK_s^{n}\right]=0,
\end{aligned}
$$
where we make use of the flat-off condition in the last inequality and recall that $K^n$ and $K^{n,k}$ are continuous. Then, integrating from $t$ to $T$ and taking expectation on both sides of (\ref{ynk-yn}), we obtain,
$$
\begin{aligned}
-\mathbb E[|Y^{n,k}_t-Y^n_t|^2]
&\ge 2\mathbb E\left[\int_t^T(Y_s^{n,k}-Y_s^n)\left(-f^{n,k}(s,Y_s^{n,k},U_s^{n,k})+f^{n}(s,Y_s^n,U_s^n)\right)dA_s\right]+\mathbb E\left[\int_{t}^T\int_E\left|U_s^{n,k}-U_s^n\right|^2\phi_s(de)dA_s\right],
\end{aligned}
$$
where we make use of the fact that $\int_0^{\cdot}2(Y_s^{n,k}-Y_s^n)\int_E(U_s^{n,k}-U_s^n)q(dsde)$ and $\int_{0}^{\cdot}\int_E\left|U_s^{n,k}-U_s^n\right|^2q(dsde)$ are martingales due to the integrability conditions $Y^{n,k},\ Y^{n}\in L^2(A)$ and $U^{n,k},\ U^n\in H_{\nu}^{2,2}$. Then, rearranging the terms and in view of the Lipschitz conditions of $f^{n,k}$, we obtain,
$$
\begin{aligned}
&\mathbb E[|Y^{n,k}_t-Y^n_t|^2]+\mathbb E\left[\int_t^T\int_E\left|U_s^{n,k}-U_s^n\right|^2\phi_s(de)dA_s\right]\\
&\le 2\mathbb E\left[\int_t^T(Y_s^{n,k}-Y_s^n)\left(f^{n,k}(s,Y_s^{n,k},U_s^{n,k})-f^{n}(s,Y_s^n,U_s^n)\right)dA_s\right]\\
&\le 2\mathbb E\left[\int_t^T|Y_s^{n,k}-Y_s^n|\left(\left|f^{n,k}(s,Y_s^{n,k},U_s^{n,k})-f^{n,k}(s,Y_s^n,U_s^n)\right|+\left|f^{n,k}(s,Y_s^{n},U_s^{n})-f^{n}(s,Y_s^n,U_s^n)\right|\right)dA_s\right]\\
&\le 2\mathbb E\left[\int_t^T\left(\tilde\beta|Y_s^{n,k}-Y_s^n|^2+n|Y_s^{n,k}-Y_s^n|\|U_s^{n,k}-U_s^n\|_s+|Y_s^{n,k}-Y_s^n|\left|f^{n,k}(s,Y_s^{n},U_s^{n})-f^{n}(s,Y_s^n,U_s^n)\right|\right)dA_s\right]\\
&\le 2\mathbb E\left[\int_t^T\left(\tilde\beta|Y_s^{n,k}-Y_s^n|^2+\frac{n^2}{2}|Y_s^{n,k}-Y_s^n|^2+\frac{1}{2}\|U_s^{n,k}-U_s^n\|_s^2+\frac{1}{2}|Y_s^{n,k}-Y_s^n|^2+\frac{1}{2}\left|f^{n,k}(s,Y_s^{n},U_s^{n})-f^{n}(s,Y_s^n,U_s^n)\right|^2\right)dA_s\right],
\end{aligned}
$$
where we make use of the inequality of arithmetic and geometric means in the last line. Equivalently,
$$
\begin{aligned}
\mathbb E[|Y^{n,k}_t-Y^n_t|^2]&\le (1+n^2+2\tilde\beta)\mathbb E\left[\int_t^T|Y_s^{n,k}-Y_s^n|^2dA_s\right]+\mathbb E\left[\int_t^T\left|f^{n,k}(s,Y_s^{n},U_s^{n})-f^{n}(s,Y_s^n,U_s^n)\right|^2dA_s\right]\\
&{\color{red}\le(1+n^2+2\tilde\beta)\mathbb E\left[\int_t^T|Y_s^{n,k}-Y_s^n|^2\rho(ds)\right]+\mathbb E\left[\int_t^T\left|f^{n,k}(s,Y_s^{n},U_s^{n})-f^{n}(s,Y_s^n,U_s^n)\right|^2dA_s\right]}\\
&{\color{red}=(1+n^2+2\tilde\beta)\int_t^T\mathbb E\left[|Y_s^{n,k}-Y_s^n|^2\right]\rho(ds)+\mathbb E\left[\int_t^T\left|f^{n,k}(s,Y_s^{n},U_s^{n})-f^{n}(s,Y_s^n,U_s^n)\right|^2dA_s\right]},
\end{aligned}
$$
where we adopt assumption (H1') and  the fact that $Y^{n,k},\ Y^n\in S^2$, and  Fubini's Theorem 
 is used in the last equality. Then by the backward Grownwall's inequality (see e.g. \cite[Proposition A.2]{Djehiche2021} ), it turns out that
$$
\mathbb E[|Y^{n,k}_t-Y^n_t|^2]\le \mathbb E\left[\int_t^T\left|f^{n,k}(s,Y_s^{n},U_s^{n})-f^{n}(s,Y_s^n,U_s^n)\right|^2dA_s\right]\exp\{(1+n^2+2\tilde\beta)\rho(T)\}.
$$
Recall that $f^{n,k}\to f^n$ as $k\to\infty$ and by dominated convergence theorem, it follows that
$$
\lim_{k\to\infty}\mathbb E[|Y^{n,k}_t-Y^n_t|^2]=0.
$$
Then, up to a subsequence, we have $\mathbb P-a.s.$, $Y_t^{n,k}\to Y_t^n$ as $k\to\infty$. Therefore, with the help of (\ref{Ykn}), 
\begin{equation}
\label{sub mart |yn|}
\exp \left\{p \lambda\left|Y^{n}_t\right|\right\} \leq \mathbb{E}_t\left[\exp \left\{p \lambda e^{3\beta A_T}(|\xi|\vee L_*^+)+p\lambda  \int_t^T e^{3\beta A_s}3\alpha_s d A_s\right\}\right].
\end{equation}
Thus $Y^n\in\mathcal E$, for $n>C_0$. Then, with the help of Proposition \ref{priori estimate on U and K}, we conclude
\begin{equation}
\label{un kn a priori estimate}
\sup_{n>C_0} \mathbb E\left[\left(\int_0^T\int_E|U^n_t(e)|^2\phi_t(de)dA_t\right)^{\frac{p}{2}}+(K_T^n)^p\right]\le C_p\mathbb E\left[e^{36p\lambda(1+3\beta\|A_T\|_\infty)Y_*}\right]<\infty.
\end{equation}

\textbf{Step 2: Construction of candidate solution $(Y^0, U^0)$. }

Recall that for $n>C_0$, $(Y^n, U^n, K^n)\in \mathcal E\times H_{\nu}^{2,2}\times \mathbb K^2$ is  the unique solution to RBSDE($\xi,\ f^n,\ L$) and $f^n\uparrow f$. Then, with the help of the comparison theorem \ref{comparison thm}, $Y^n_t$ is increasing with respect to $n$ for each $t\in[0,T]$. Define $Y_t^n\uparrow Y_t^0$ as $n\to\infty$. 
Then the process $Y^0$ is adapted and admits a progressive version. Therefore, we assume that the process $Y^0$ is progressively measurable.
Moreover, in view of Fatou's Lemma,
\begin{equation}
\label{Y0}
\mathbb E\left[\int_0^T|Y^0_s|^2dA_s\right]\le\mathbb E[\sup_{n>C_0}|Y^n_*|^2]\|A_T\|_\infty<\infty.
\end{equation}

With $Y^0$ in hand, we are going to find a candidate solution $U^0$.
For any $m,\ n>C_0$, applying Itô's formula to the process $\left|Y^n-Y^m\right|^2$, we can deduce that
$$
\begin{aligned}
&\int_0^T\int_E\left|U_s^n(e)-U_s^m(e)\right|^2 \phi_s(de)d A_s\\
&= -\left|Y_0^n-Y_0^m\right|^2+2 \int_0^T\left(Y_{s^-}^n-Y_{s^-}^m\right)\left(f^n\left(s, Y_s^n, U_s^n\right)-f^m\left(s, Y_s^m, U_s^m\right)\right) d A_s \\
&\quad -2 \int_0^T\int_E\left(\left(Y_{s^-}^n-Y_{s^-}^{m}\right)\left(U_s^n-U_s^m\right)+|U_s^n-U_s^m|^2\right) q(dsde) +2\int_0^T(Y_t^{n}-Y_t^m)(dK_t^{n}-dK_t^m)\\
&\leq  2\left(\int_0^T\left|Y_t^n-Y_t^m\right|^2dA_t\right)^{1/2}\left(\int_0^T\left(f^n\left(s, Y_s^n, U_s^n\right)-f^m\left(s, Y_s^m, U_s^m\right)\right)^2dA_t \right)^{1/2}\\
& \quad -\left|Y_0^n-Y_0^m\right|^2-2 \int_0^T\int_E\left(\left(Y_{s^-}^n-Y_{s^-}^{m}\right)\left(U_s^n-U_s^m\right)+|U_s^n-U_s^m|^2\right) q(dsde)+2\int_0^T(Y_t^{n}-Y_t^m)(dK_t^{n}-dK_t^m)\\
&\leq  2(\int_0^T\left|Y_t^n-Y_t^m\right|^2dA_t)^{1/2}\left( \int_0^T2\left(3\alpha_s+3\beta|Y_s^m|+\frac{1}{\lambda} \left(j_\lambda(U_s^m)+j_{\lambda}(-U_s^m)\right)\right)^2d A_s\right.\\
&\left.\quad+\int_0^T2\left(3\alpha_s+3\beta|Y_s^n|+\frac{1}{\lambda} \left(j_\lambda(U_s^n)+j_{\lambda}(-U_s^n)\right)\right)^2dA_s\right)^{1/2}\\
& \quad -\left|Y_0^n-Y_0^m\right|^2-2 \int_0^T\int_E\left(\left(Y_{s^-}^n-Y_{s^-}^{m}\right)\left(U_s^n-U_s^m\right)+|U^n-U^m|^2\right) q(dsde)+2\int_0^T(Y_t^{n}-Y_t^m)(dK_t^{n}-dK_t^m) , \quad \mathbb P \text {-a.s. }
\end{aligned}
$$
The second last term is a martingale due to the integrability condition of $Y^n,\ Y^m$ and $U^n,\ U^m$. Then taking expectation, similar to (\ref{dk=0}) and making use of Hölder's inequality, we deduce from (\ref{e^-U}) and (\ref{e^U}), with parameter $(\alpha,\ \beta)$ replaced by $(3\alpha,\ 3\beta)$,  that there exists a constant $c>0$ such that
$$
\begin{aligned}
\mathbb E\left[\left(\int_0^T\int_E\left|U_s^n-U_s^m\right|^2\phi_s(e) d A_s\right)\right]&\le c\left(\mathbb E[\int_0^T\left|Y_t^n-Y_t^m\right|^2dA_t]\right)^{1/2}-\mathbb E[\left|Y_0^n-Y_0^m\right|^2]\\
&\le c\left(\mathbb E[\int_0^T\left|Y_t^n-Y_t^0\right|^2dA_t]+\mathbb E[\int_0^T\left|Y_t^m-Y_t^0\right|^2dA_t]\right)^{1/2}+2\mathbb E[\left|Y_0^n-Y_0^0\right|^2]+2\mathbb E[\left|Y_0^m-Y_0^0\right|^2].
\end{aligned}
$$
 Hence, it follows from (\ref{Y0}) and monotone convergence theorem that
$$
\lim _{N\rightarrow \infty} \sup_{m,n\ge N}\mathbb E\left[\left(\int_0^T\int_E\left|U_s^n(e)-U_s^m(e)\right|^2\phi_s(de) d A_s\right)\right]=0.
$$
Thus, $\{U^n\}$ is a Cauchy sequence in $H_\nu^{2,2}$, which implies that there exists a $U^0\in H_\nu^{2,2}$ such that 
\begin{equation}
\label{estimate-U}
\lim _{ n \rightarrow \infty} \mathbb E\left[\left(\int_0^T\int_E\left|U_s^n(e)-U_s^0(e)\right|^2\phi_s(de) d A_s\right)\right]=0.
\end{equation}
Thus, up to a subsequence,
$$
\lim _{ n \rightarrow \infty} \left(\int_0^T\int_E\left|U_s^n(e)-U_s^0(e)\right|^2\phi_s(de) d A_s\right)=0, \ \mathbb P-a.s.
$$
Indeed, with the help of (\ref{un kn a priori estimate}) and Fatou's Lemma, $U^0\in H_{\nu}^{2,p}$ for each $p\ge 1$. More precisely, for some positive constant $C_p$ depending on $p$,
\begin{equation}
\label{u0}
\mathbb E\left[\left(\int_0^T\int_E|U_t^0(e)|^2\phi_t(de)dA_t\right)^p\right]\le C_p\sup_{n>C_0}\mathbb E\left[\left(\int_0^T\int_E|U_t^n(e)|^2\phi_t(de)dA_t\right)^p\right]<\infty.
\end{equation}

Moreover, on the set $\{(t,\omega)\in[0,T]\times \Omega:dA_t(\omega)\not=0\}$, it holds that
$$
\lim _{n \rightarrow \infty} \left(\int_E\left|U_t^n(e)-U_t^0(e)\right|^2\phi_t(de) \right)=0.
$$

\textbf{Step 3: A priori estimate of $\left|Y^n-Y^m\right|$.}

For $m, n >C_0$ and fixed $\theta \in(0,1)$. We first show that $\mathbb{P}$-a.s.
$$
\left|Y_t^n-Y_t^m\right| \leq(1-\theta)\left(\left|Y_t^m\right|+\left|Y_t^n\right|\right)+\frac{1-\theta}{\lambda} \ln \left(\sum_{i=1}^3 J_t^{m, n, i}\right), \quad t \in[0, T],
$$
where $\zeta_\theta:=\frac{\lambda e^{\tilde{\beta}\|A_T\|_{\infty}}}{1-\theta}$, and $J_t^{n, m, i}:=E\left[J_T^{n, m, i} \mid \mathscr{G}_t\right]$ for $i=1,2,3$ such that
$$
\begin{aligned}
& J_T^{n, m, 1}:=\left(D_T^m+D_T^n\right) \eta \text { with } D_t^m:=\exp \left\{\lambda e^{2 \tilde{\beta}\left\|A_T\right\|_{\infty}} \int_0^t\left(3\alpha_s+(3\beta+\tilde{\beta})\left|Y_s^m\right|\right) d A_s\right\}; \\
& D_t^n:=\exp \left\{\lambda e^{2 \tilde{\beta}\left\|A_T\right\|_{\infty}} \int_0^t\left(3\alpha_s+(3\beta+\tilde{\beta})\left|Y_s^n\right|\right) d A_s\right\}, t \in[0, T]; \\
& \eta:=\exp \left\{\zeta_\theta e^{\tilde{\beta}\left\|A_T\right\|_{\infty}}\left(1-\theta\right)|\xi|\right\}= \exp\{\lambda e^{2\tilde\beta \|A_T\|_\infty}|\xi|\} ; \\
& J_T^{n, m, 2}:=\zeta_\theta \exp \left\{\tilde{\beta}\left\|A_T\right\|_{\infty}\right\}\left(D_T^m+D_T^n\right) \exp \left\{\lambda e^{2 \tilde{\beta}\left\|A_T\right\|_{\infty}}\left(Y_*^m+Y_*^n\right)\right\}\left(K_T^m+K_T^n\right); \\
& \Upsilon_{n, m}:=\exp \left\{\zeta_\theta e^{\tilde{\beta}\left\|A_T\right\|_{\infty}}\left(Y_*^n+Y_*^m\right)\right\} ; \\
& J_T^{n, m, 3}:=\zeta_\theta e^{\tilde{\beta}\left\|A_T\right\|_{\infty}}\left(D_T^m+D_T^n\right) \Upsilon_{n, m} \int_0^T\left|\Delta_{n, m} f(s)\right| d A_s; \\
& \Delta_{n, m} f(t):=\left|f^n\left(t, Y_t^m, U_t^m\right)-f^m\left(t, Y_t^m, U_t^m\right)\right|+\left|f^m\left(t, Y_t^n, U_t^n\right)-f^n\left(t, Y_t^n, U_t^n\right)\right|,\ t \in[0, T].
\end{aligned}
$$
 We set $\tilde{Y}^{(n, m)}:=Y^n-\theta Y^m, \tilde{U}^{(n, m)}:=U^n-\theta U^m$ and define two processes
$$
a_t^{(n, m)}:=\mathbf{1}_{\left\{\tilde{Y}_t^{(n, m)} \neq 0\right\}} \frac{f^n\left(t, Y_t^n, U_t^n\right)-f^n\left(t, \theta Y_t^m, U_t^n\right)}{\tilde{Y}_t^{(n, m)}}-\tilde{\beta} \mathbf{1}_{\left\{\tilde{Y}_t^{(n, m)}=0\right\}}, \quad \tilde{A}_t^{(n, m)}:=\int_0^t a_s^{(n, m)} d A_s, \quad t \in[0, T] .
$$
For a given $t\in [0,T]$, consider a sequence of stopping times $\left(T^t_K\right)_{K \geq 0}$ defined by:
\begin{equation}
\label{def-TKt}
T_K^t=\inf \left\{s>t, \mathbb{E}\left[\exp \left(e^ {3\beta 
A_T}\lambda(\left|\xi\right|\vee L_*^+)+\int_0^T \lambda e^{3\beta A_s}3\alpha_s d A_s\right) \mid \mathcal{G}_s\right]>e^{\lambda K}\right\}.
\end{equation}
The sequence $\left(T^t_K\right)_{K \geq 0}$ converges to infinity when $K$ goes to infinity. Moreover,  from (\ref{|y| bound}), we find that the sequence $\{Y_{\cdot \wedge T_K^t}^{n}\}_{n\ge 1}$ is uniformly bounded on $(t,T]$, in precise, $\sup_n\|Y^{n}_{\cdot\wedge T_K^t}\|_{S^{\infty}(t,T]}\le K$. Let  $U_s^{n, K}=U_s^{n} 1_{s<T_K^t}$.  Then, in view of \cite[Corollary 1]{Morlais_2009}, $\mathbb P-a.s.$, $|U_s^{n,K}|\le 2K$ on $(t,T]$. 
Applying Itô's formula to the process $\Gamma_s^{n, m}:=\exp \left\{\zeta_\theta e^{\tilde{A}_s^{(n, m)}} \tilde{Y}_s^{(n, m)}\right\}$ on $s \in[t, T_K^t\wedge T]$ yields that
\begin{equation}
\label{gamma ito1}
\begin{aligned}
\Gamma_t^{n, m}= & \Gamma_{T_K^t\wedge T}^{n, m}+\int_t^{T_K^t\wedge T} \Gamma_{s^-}^{n, m}\left[\zeta_\theta e^{\tilde{A}_s^{n, m}} F_s^{n, m}-\int_E\left(e^{\zeta_\theta e^{\tilde{A}_s^{n, m}} \tilde{U}_s^{n, m}}-\zeta_\theta e^{\tilde{A}_s^{n, m}} \tilde{U}_s^{n, m}-1\right) \phi_s(d e)\right] d A_s \\
& -\int_t^{T_K^t\wedge T} \int_E \Gamma_{s^-}^{n,m}\left(e^{\zeta_\theta e^{\tilde{A}_s^{n, m}} \tilde{U}_s^{n, m}}-1\right) q(d s d e)+\zeta_\theta \int_t^{T_K^t\wedge T} \Gamma_{s^{-}} e^{\tilde{A}_s^{n, m}}\left(d K_s^{(n)}-\theta d K_s^{(m)}\right) \\
:= & \Gamma_{T_K^t\wedge T}^{n, m}+\int_t^{T_K^t\wedge T} G_s^{n, m} d A_s-\int_t^{T_K^t\wedge T} \int_E Q_s^{n, m} q(d s d e)+\zeta_\theta \int_t^{T_K^t\wedge T} \Gamma_{s^{-}}^{n, m} e^{\tilde{A}_s^{(n, m)}}\left(d K_s^{(n)}-\theta d K_s^{(m)}\right),
\end{aligned}
\end{equation}
where
$$
\begin{gathered}
F_t^{n, m}=f^n\left(t, Y_t^n, U_t^n\right)-\theta f^m\left(t, Y_t^m, U_t^m\right)-a_t^{(n, m)} \tilde{Y}_t^{(n, m)}, \\
G_t^{n, m}=\zeta_\theta \Gamma_t^{n, m} e^{\tilde{A}_t^{(n, m)}}\left(f^n\left(t, Y_t^n, U_t^n\right)-\theta f^m\left(t, Y_t^m, U_t^m\right)-a_t^{(n, m)} \tilde{Y}_t^{(n, m)}\right)-\Gamma_t^{1, n, m} j_1\left(\zeta_\theta e^{\tilde{A}_t^{(n, m)}} \tilde{U}_t^{n, m}\right) .
\end{gathered}
$$
Similarly to (\ref{cvx estimate}), (H4)(b) and the convexity of $f^n$ in $u$ show that $d t \otimes d \mathbb{P}$-a.e.
$$
f^n\left(t, Y_t^m, U_t^n\right) \leq \theta f^n\left(t, Y_t^m, U_t^m\right)+(1-\theta)\left(\alpha_t+\beta\left|Y_t^m\right|\right)+\frac{1-\theta}{\lambda} j_\lambda\left(\frac{\tilde{U}_t^{n, m}}{1-\theta}\right),
$$
which together with (\ref{gamma ito1}) implies that $d t \otimes d \mathbb{P}$-a.e.,
$$
\begin{aligned}
& G_t^{n, m}=\zeta_\theta \Gamma_t^{n, m} e^{\tilde{A}_t^{(n, m)}}\left(f^n\left(t, \theta Y_t^m, U_t^n\right)-\theta f^m\left(t, Y_t^m, U_t^m\right)\right)-\Gamma_t^{n, m} j_1\left(\zeta_\theta e^{\tilde{A}_t^{(n, m)}} \tilde{U}_t^{n, m}\right) \\
& \leq \zeta_\theta \Gamma_t^{n, m} e^{\tilde{A}_t^{(n, m)}}\left(\left|f^n\left(t, \theta Y_t^m, U_t^n\right)-f^n\left(t, Y_t^m, U_t^n\right)\right|+f^n\left(t, Y_t^m, U_t^n\right)-\theta f^m\left(t, Y_t^m, U_t^m\right)\right)-\Gamma_t^{n, m} j_1\left(\zeta_\theta e^{\tilde{A}_t^{(n, m)}} \tilde{U}_t^{n, m}\right) \\
& \leq \lambda e^{2 \tilde{\beta}\left\|A_T\right\|_{\infty}} \Gamma_t^{n, m}\left(\alpha_t+(\beta+\tilde{\beta})\left|Y_t^m\right|\right)+\zeta_\theta e^{\tilde{\beta}\left\|A_T\right\|_{\infty}} \Gamma_t^{n, m}\left|\Delta_{n, m} f(t)\right|
\end{aligned}.
$$
It follows from integration by parts that
\begin{equation}
\label{ito02}
\begin{aligned}
\Gamma_t^{n, m} \leq & D_t^m \Gamma_t^{n, m} \\
\leq & D_{T_K^t\wedge T}^m \Gamma_{T_K^t\wedge T}^{n, m}+\zeta_\theta \int_t^{T_K^t\wedge T} D_s^m \Gamma_{s^{-}}^{ n, m} e^{\tilde{A}_s^{n, m}} d K_s^n-\zeta_\theta \int_t^{T_K^t\wedge T} \int_E D_s^m \Gamma_{s^{-}}^{n, m}\left(e^{\zeta_\theta e^{\tilde{A}_s^{n, m}} \tilde{U}_s^{n, m}}-1\right) q(d s d e) \\
& \quad+\zeta_\theta e^{\tilde{\beta}\left\|A_T\right\|_{\infty}} \int_t^{T_K^t\wedge T} D_s^m \Gamma_s^{n, m}\left|\Delta_{n, m} f(s)\right| d A_s \\
& \leq D_T^m\eta^{n,m}_{T_K^t\wedge T}+\zeta_\theta e^{\tilde{\beta}\left\|A_T\right\|_{\infty}} D_T^m \int_0^T \Gamma_{s^{-}}^{n, m} d K_s^n-\zeta_\theta \int_t^T \int_E D_s^m \Gamma_{s^{-}}^{n, m}\left(e^{\zeta_\theta e^{\tilde{A}_s^{n, m}} \tilde{U}_s^{n, m}}-1\right) q(d s d e)+J_T^{n, m, 3}, \quad t \in[0, T] .
\end{aligned}
\end{equation}
where we define 
\begin{equation}
\label{etamn}
\eta^{n,m}_{T_K^t\wedge T}=\exp \left\{\zeta_\theta e^{\tilde\beta\|A_T\|}\left(|Y_{T_K^t\wedge T}^n-\theta Y_{T_K^t\wedge T}^m|\vee|Y_{T_K^t\wedge T}^m-\theta Y_{T_K^t\wedge T}^n|\right)\right\}\stackrel{K\to\infty}{\longrightarrow}\eta.
\end{equation}
for each $n,m>C_0$.
We deal with the right-hand side of (\ref{ito02}) term by term. Firstly, the flat-off condition of $\left(Y^n, U^n, K^n\right)$ implies that
$$
\int_0^T \Gamma_s^{n, m} d K_s^n  =\int_0^T \mathbf{1}_{\left\{Y_{s^{-}}^n=L_s\right\}} \Gamma_s^{n, m} d K_s^n\leq \int_0^T \mathbf{1}_{\left\{Y_{s^{-}}^n \leq Y_{s^{-}}^m\right\}} \exp \left\{\lambda e^{2 \tilde{\beta}\left\|A_T\right\|_{\infty}}\left|Y_s^m\right|\right\} d K_s^n \leq \exp \left\{\lambda e^{2 \tilde{\beta}\left\|A_T\right\|_{\infty}} Y_*^m\right\} K_T^n, \quad \mathbb{P} \text {-a.s. }
$$
We also note that
for each $p \in(1, \infty)$, (\ref{sub mart |yn|}) and (\ref{un kn a priori estimate}) imply that
{\color{red}
\begin{equation}
\label{uniform bounded Y U }
\sup _{n >C_0}\mathbb E\left[e^{p \lambda Y_*^{n}}+\left(\int_0^T\int_E\left|U_s^{n}\right|^2\phi_t(de) d A_s\right)^p+(K^n_T)^p\right]  \leq C_p \Xi\left(72 p \lambda(1+3\beta\|A_T\|_\infty),\ 3\alpha,\ 3\beta\right) ,
\end{equation}
where we make use of the notation $\Xi$ defined in Corollary \ref{sub mart |y|}. Thus, it follows that for $n,m >C_0$
\begin{equation}
\label{ess-eta-nm}
\mathbb E\left[\eta^p\right] \leq\mathbb E\left[e^{p \zeta_\theta e^{\tilde \beta \|A_T\|}(1-\theta)\left|\xi\right|}\right]\le\mathbb E\left[e^{p \lambda e^{2\tilde\beta \|A_T\|}\left|\xi\right|}\right]  \leq \Xi\left(p\lambda e^{2\tilde\beta \|A_T\|},\ 3\alpha,\ 3\beta\right);
\end{equation}
\begin{equation}
\label{ess02}
\begin{aligned}
\mathbb E\left[\Upsilon_{n,m}^p\right] & \leq \frac{1}{2}\mathbb E\left[e^{2 p \zeta_\theta e^{\tilde\beta \|A_T\|} Y_*^n}+e^{2 p \zeta_\theta e^{\tilde\beta \|A_T\|} Y_*^m}\right] \leq {\Xi}\left(\frac{2 p\lambda}{1-\theta} e^{2\tilde \beta \|A_T\|},\ 3\alpha,\ 3\beta\right).\\
\end{aligned}
\end{equation}                                                             
}
Meanwhile,

\begin{equation}
\label{delta f}
\begin{aligned}
& E\left[\left(\int_0^T\left|\Delta_{n, m} f(s)\right| d A_s\right)^p\right] \\
&\quad \leq E\left[\left(2 \int_0^T 3\alpha_t d A_t+2\left\|A_T\right\|_{\infty}\left(3\beta Y_*^m+3\beta Y_*^n\right)+\int_0^T\left(2 j_\lambda\left(U^m\right)+2 j_\lambda\left(U^n\right)+2 j_\lambda\left(-U^m\right)+2 j_\lambda\left(-U^n\right)\right) d A_s\right)^p\right]. \\
\end{aligned}
\end{equation}
Moreover,
$$
D_T^m \leq \exp \left\{\lambda e^{2 \tilde{\beta}\left\|A_T\right\|_{\infty}} \int_0^T 3\alpha_s d A_s\right\} \exp \left\{\lambda(3\beta+\tilde{\beta})\left\|A_T\right\|_{\infty} e^{2 \tilde{\beta}\left\|A_T\right\|_{\infty}} Y_*^m\right\},
$$
and then $D_T^m \in \mathbb{L}^p\left(\mathscr{G}_T\right)$ with the help of (\ref{uniform bounded Y U }). Thus, random variables $J_T^{n, m, i}, i=1,2,3$ are integrable by Young's inequality and (\ref{ess-eta-nm})-(\ref{delta f}). 
In addition, $\eta^{n,m}_{T_K^t\wedge T}\le \exp\{\frac{\lambda e^{2\tilde\beta\|A_T\|_\infty}}{1-\theta}(Y^n_*+Y^m_*)\}$,  in which the right hand side  is also integrable.

Note that $\int_t^{T_{K}^t\wedge T}\int_E D^m_s \Gamma_{s^-}^{n,m}(e^{\zeta_\theta e^{\tilde A^{n,m}_s}\tilde U^{n,m}_s}-1)q(dsde)$ is a true martingale due to the boundedness of the integrand. Taking $E\left[\cdot \mid \mathscr{G}_t\right]$ in (\ref{ito02}) yields that $\Gamma_t^{1, n, m} \leq \mathbb E_t[D_T^m\eta^{n,m}_{T^t_K\wedge T}]+\sum_{i=2}^3 J_t^{n, m, i}, \ \mathbb{P}$-a.s.
Letting $K\to\infty$, by dominated convergence and (\ref{etamn}), it holds
$$
\Gamma_t^{1, n, m} \leq \mathbb E_t[D_T^m\eta]+\sum_{i=2}^3 J_t^{n, m, i}\le \sum_{i=1}^3 J_t^{n, m, i}, \ \mathbb{P}-a.s.
$$
It then follows that
$$
Y_t^n-\theta Y_t^m \leq \frac{1-\theta}{\lambda} e^{-\tilde{\beta}\left\|A_T\right\|_{\infty}-\tilde{A}_t^{n, m}} \ln \left(\sum_{i=1}^3 J_t^{n, m, i}\right) \leq \frac{1-\theta}{\lambda} \ln \left(\sum_{i=1}^3 J_t^{n, m, i}\right), \quad \mathbb{P} \text {-a.s. }
$$
which implies
$$
Y_t^n-Y_t^m \leq(1-\theta)\left|Y_t^m\right|+\frac{1-\theta}{\lambda} \ln \left(\sum_{i=1}^3 J_t^{n, m, i}\right) \leq(1-\theta)\left(\left|Y_t^m\right|+\left|Y_t^n\right|\right)+\frac{1-\theta}{\lambda} \ln \left(\sum_{i=1}^3 J_t^{n, m, i}\right), \quad \mathbb{P} \text {-a.s. }
$$
Exchanging the role of $Y^m$ and $Y^n$, we deduce the other side of the inequality.
Hence,
\begin{equation}
\label{ess05}
|Y_t^m-Y_t^n| \leq(1-\theta)(\left|Y_t^m\right|+\left|Y_t^n\right|)+\frac{1-\theta}{\lambda} \ln \left(\sum_{i=1}^3 J_t^{m,n, i}\right), \quad \mathbb P \text {-a.s. }
\end{equation}

\textbf{Step 4: Construction of candidate solution $\tilde Y^0$ in appropriate space.}

For any $N>C_0$, (\ref{ess05}) and Jensen's inequality imply that
\begin{equation}
\label{y s1 cvg}
\begin{aligned}
& \sup _{m, n \geq N} \mathbb{E}\left[\sup _{t \in[0, T]}\left|Y_t^n-Y_t^m\right|\right] \leq \sup _{m, n \geq N} \mathbb E\left[(1-\theta)\left(Y_*^m+Y_*^n\right)\right]+\sup _{m, n \geq N} \mathbb{E}\left[\frac{1-\theta}{\lambda} \ln \left(\sum_{i=1}^3 J_*^{n, m, i}\right)\right] \\
& \leq \sup _{m, n \geq N} \mathbb E\left[(1-\theta)\left(Y_*^m+Y_*^n\right)\right]+\frac{1-\theta}{\lambda} \sup _{m, n \geq N} \ln \left(\sum_{i=1}^3 \mathbb{E}\left[J_*^{n, m, i}\right]\right) \\
& \leq \sup _{m, n>C_0} \frac{1-\theta}{\lambda} \mathbb{E}\left[e^{\lambda Y_*^m}+e^{\lambda Y_*^n}\right]+\frac{1-\theta}{\lambda} \ln \left(2\sup _{m, n>C_0} \left(\mathbb{E}\left[\left(D_T^m+D_T^n\right)^2\exp\{2\lambda e^{2\tilde\beta \|A_T\|_\infty}|\xi|\}\right]\right)^{1/2}\right. \\
&\quad \quad+2\sup _{m, n>C_0}\left( \mathbb{E}\left[\zeta_\theta^2 \exp \left\{2\tilde{\beta}\left\|A_T\right\|_{\infty}\right\}\left(D_T^m+D_T^n\right)^2 \exp \left\{2\lambda e^{2 \tilde{\beta}\left\|A_T\right\|_{\infty}}\left(Y_*^m+Y_*^n\right)\right\}\left(K_T^m+K_T^n\right)^2\right]\right)^{1/2}\\
&\quad\quad +\left.2\sup _{m, n \geq N} \left(\mathbb{E}\left[\zeta_\theta^2 e^{2\tilde{\beta}\left\|A_T\right\|_{\infty}} (D_T^n+D_T^m)^2 \Upsilon_{n,m}^2\left( \int_0^T\left|\Delta_{n, m} f(s)\right| d A_s\right)^2\right]\right)^{1/2}\right)\\
& \leq \sup _{m, n>C_0} \frac{1-\theta}{\lambda} \mathbb{E}\left[e^{\lambda Y_*^m}+e^{\lambda Y_*^n}\right]+\frac{1-\theta}{\lambda} \ln \left(2\sup _{m, n>C_0} \left(\mathbb{E}\left[\left(D_T^m+D_T^n\right)^2\exp\{2\lambda e^{2\tilde\beta \|A_T\|_\infty}|\xi|\}\right]\right)^{1/2}\right. \\
&\quad \quad+2\sup _{m, n>C_0}\left( \mathbb{E}\left[\zeta_\theta^2 \exp \left\{2\tilde{\beta}\left\|A_T\right\|_{\infty}\right\}\left(D_T^m+D_T^n\right)^2 \exp \left\{2\lambda e^{2 \tilde{\beta}\left\|A_T\right\|_{\infty}}\left(Y_*^m+Y_*^n\right)\right\}\left(K_T^m+K_T^n\right)^2\right]\right)^{1/2}\\
&\quad\quad +\left.2\sup _{m, n>C_0} \left(\mathbb{E}\left[\zeta_\theta^4 e^{4\tilde{\beta}\left\|A_T\right\|_{\infty}} (D_T^n+D_T^m)^4 \Upsilon_{n,m}^4\right]\right)^{1/4}\sup_{m,n\ge N}\left(\mathbb E\left[\left( \int_0^T\left|\Delta_{n, m} f(s)\right| d A_s\right)^4\right]\right)^{1/4}\right),\\
\end{aligned}
\end{equation}
where we make use of Doob's maximum inequality in the second last inequality.

To prove the $S^1$ convergence of the sequence $Y^n$, we first deal with the integral term $ \int_0^T\left|\Delta_{n,m} f(s)\right| d A_s$. Consider the sequence of stopping times $\left(T_K\right)_{K \geq 0}$ defined by:
\begin{equation}
\label{def-TK}
T_K=\inf \left\{t \geq 0, \mathbb{E}\left[\exp \left(e^ {3\beta 
A_T}\lambda\left(\left|\xi\right|\vee L_*^+\right)+\int_0^T \lambda e^{3\beta A_s}3\alpha_s d A_s\right) \mid \mathcal{F}_t\right]>e^{\lambda K}\right\}.
\end{equation}
The sequence $\left(T_K\right)_{K \geq 0}$ converges to infinity when $K$ goes to infinity. Moreover,  from (\ref{|y| bound}), we find that the sequence $\{Y_{\cdot \wedge T_K}^{n}\}_{n\ge 1}$ is uniformly bounded, in precise, $\sup_n\|Y^{n}_{\cdot\wedge T_K}\|_{S^{\infty}}\le K$. Then $\|Y^0_{t\wedge T_K}\|_\infty\le K$. Let  $U_t^{n, K}=U^{n} 1_{t<T_K}$.  Then, in view of \cite[Corollary 1]{Morlais_2009}, $\mathbb P-a.s.$, $|U_t^{n,K}|\le 2K$ and $|U_t^{0,K}|\le 2K$. Then, for $t\le T_K$, on $\{|y|\le K,\|u\|_t\le 2K\}$ $|f(t,y,u)|\le 3\alpha_t+3\beta|y|+\frac{2}{\lambda}e^{2\lambda K }.$ Then, inspired by  \cite[Lemma 1]{Lepeltier_1997}, on $\{(\omega,t):\ dA_t(\omega)>0\}\cap\{t< T_K\}$, $\lim_{n\to\infty}f^n(t,Y_t^n,U^n_t)=f(t,Y^0_t,U^0_t)$. In the same manner, it can also be proven that $\lim_{N\to\infty}\sup_{m,n>N}|f^n(t,Y_t^m,U^m_t)-f(t,Y^0_t,U^0_t)|=0$.
Note that
$$
\Delta_{n,m}f(t)=\left|f^n\left(t, \omega, Y_t^m, U_t^m\right)-f^m\left(t, \omega, Y_t^m, U_t^m\right)\right|+ \left|f^m\left(t, \omega, Y_t^n, U_t^n\right)-f^n\left(t, \omega, Y_t^n, U_t^n\right)\right|.
$$
It then obvious that 
$$
\lim _{N \rightarrow \infty}\sup_{m,n>N}\Delta_{n,m}f(t,\omega)=0.
$$
Hence, with the uniform integrablity conditions on $Y^n$, $Y^m$ $U^n$, $U^m$, and (\ref{delta f}), making use of dominated convergence theorem, we deduce that for each $p\ge 1$,
\begin{equation}
\label{cvg_deltafmn}
\lim_{N\to\infty}\sup_{m,n\ge N}\mathbb E\left[\left(\int_0^{T\wedge T_K}\left|\Delta_{n,m}f(t)\right|dA_t\right)^p\right]=0.
\end{equation}
As a consequence, letting $K$ goes to infinity, by monotone convergence theorem, it turns out that
$$
\lim_{N\to\infty}\sup_{m,n\ge N}\mathbb E\left[\left(\int_0^{T}\left|\Delta_{n,m}f(t)\right|dA_t\right)^p\right]=0,\ \forall p\ge 1.
$$
Thus, first letting $N\rightarrow \infty$ in (\ref{y s1 cvg}) and then letting $\theta\rightarrow 1$ yield that $\lim_{N\to\infty}\sup _{m, n \geq N} \mathbb{E}\left[\sup _{t \in[0, T]}\left|Y_t^n-Y_t^m\right|\right]=0$, which implies that the sequence $Y^n$ converges in $S^1$.
 Then, up to a subsequence, there exists a process $\tilde Y^0$ such that  $\lim_{n\to\infty}\sup_{t\in[0,T]}|Y^n_t-\tilde Y^0_t|=0$. Then, $\tilde Y^0$ is c\`adl\`ag. With the help of (\ref{sub mart |yn|}) and Doob's maximum inequality, $\tilde Y^0\in\mathcal E$.

Fix $p \in[1, \infty)$. Since $\mathbb E\left[\exp \left\{2 p \lambda \cdot \sup _{t \in[0, T]}\left|Y_t^n-\tilde Y_t^0\right|\right\}\right] \leq \frac{1}{2} E\left[e^{4 p \lambda Y_*^n}+e^{4 p \lambda  \tilde Y_*^0}\right] \leq {\Xi}(4 p\lambda, 3\alpha, 3 \beta)$ holds for any $n >C_0$, it turns out that $\left\{\exp \left\{p \lambda \cdot \sup _{t \in[0, T]}\left|Y_t^n-\tilde Y_t^0\right|\right\}\right\}_{n >C_0}$ is a uniformly integrable sequence in $\mathbb{L}^1\left(\mathcal{G}_T\right)$. Then it follows that $\lim _{n \rightarrow \infty} \mathbb E\left[\exp \left\{p \lambda \cdot \sup _{t \in[0, T]}\left|Y_t^n-\tilde Y_t^0\right|\right\}\right]=1$, which in particular implies that
\begin{equation}
\label{cvg_Y}
\lim _{n \rightarrow \infty}\mathbb E\left[\sup _{t \in[0, T]}\left|Y_t^n-\tilde Y_t^0\right|^q\right]=0, \quad \forall q \in[1, \infty).
\end{equation}

\textbf{Step 5:  Find candidate solution $K^0$ and  verify the solution $(\tilde Y^0,U^0, K^0)$. } 

We have constructed a candidate solution $(\tilde Y^0,\ U^0)\in \mathcal E\times H_{\nu}^{2,p}$. We are going to find a candidate solution $K^0$ for the third component. 
Consider the sequence of stopping times $\left(T_K\right)_{K \geq 0}$ defined as in (\ref{def-TK}). For $t \leq T_K\wedge T$, we obtain exactly as step 4, on $\{(t,\omega), dA_t(\omega)\not=0\}\cap\{t<T_K\}$,
$$
\lim _{n \rightarrow \infty} f^n\left(t, \omega, Y_t^n, U_t^n\right)=f\left(t, \omega, \tilde Y_t^0, U_t^0\right).
$$
Thus, by dominated convergence theorem,
$$
\begin{aligned}
&\mathbb{E}\left[\int_0^{T_K\wedge T}\left|f^{n}\left(t, Y_t^{n},U_t^{n}\right)-f\left(t, \tilde Y^0_t, U^0_t\right)\right| d A_t\right]\\
&=\mathbb{E}\left[\int_0^{T_K\wedge T}\left|f^{n}\left(t, Y_t^{n},U_t^{n,K}\right)-f\left(t, \tilde Y^{0}_t, U^{0,K}_t\right)\right| d A_t\right],
\end{aligned}
$$
which goes to zero when $n$ goes to infinity. 
As a consequence, first letting $n$ go to infinity, then pushing $K$ to infinity, it turns out that
\begin{equation}
\label{cvg_f}
\lim_{n\to\infty }\mathbb{E}\left[\int_0^{T}\left|f^{n}\left(t, Y_t^{n}, U_t^{n}\right)-f\left(t, \tilde Y^0_t,U^0_t\right)\right| d A_t\right]=0.
\end{equation}

For any $n, m >C_0$, it holds $\mathbb{P}$-a.s. that
$$
K_t^n-K_t^m=Y_0^n-Y_0^m-\left(Y_t^n-Y_t^m\right)-\int_0^t\left(f^n\left(s, Y_s^n, U_s^n\right)-f^m\left(s, Y_s^m, U_s^m\right)\right) d A_s+\int_0^t \int_E\left(U_s^n-U_s^m\right) q(d s d e), \quad t \in[0, T] .
$$
The Burkholder-Davis-Gundy inequality then implies that there exists a universal constant $c>0$ such that
$$
\begin{aligned}
& E\left[\sup _{t \in[0, T]}\left|K_t^n-K_t^m\right|^2\right] \leq c E\left[\sup _{t \in[0, T]}\left|Y_t^n-Y_t^m\right|^2\right]+c E\left[\left(\int_0^T\left|f^n\left(s, Y_s^n, U_s^n\right)-f^m\left(s, Y_s^m, U_s^m\right)\right|^2 d A_s\right)\right] \\
& +c E\left[\left(\int_0^T \int_E\left|U_s^n-U_s^m\right|^2 \phi_s(e) d A_s\right)\right] \\
& \leq c E\left[\sup _{t \in[0, T]}\left|Y_t^n-Y_t^m\right|^2\right]+c E\left[\left(\int_0^T\left|f^n\left(s, Y_s^n, U_s^n\right)-f\left(s, Y_s^0, U_s^0\right)\right|^2 d A_s\right)\right] \\
& +c E\left[\left(\int_0^T\left|f^m\left(s, Y_s^m, U_s^m\right)-f\left(s, Y_s^0, U_s^0\right)\right|^2 d A_s\right)\right]+c E\left[\left(\int_0^T \int_E\left|U_s^n-U_s^m\right|^2 \phi_s(e) d A_s\right)\right].
\end{aligned}
$$
Together with (\ref{cvg_Y}), (\ref{cvg_f}) and (\ref{estimate-U}), we deduce that:
$$
\lim _{N\rightarrow \infty} \sup _{m, n \geq N} E\left[\sup _{t \in[0, T]}\left|K_t^n-K_t^m\right|^2\right]=0 .
$$
Hence there is a $K^0$ in $\mathbb{K}^2$ such that
\begin{equation}
\label{cvg_K}
\lim _{n \rightarrow \infty} E\left[\sup _{t \in[0, T]}\left|K_t^n-K_t^0\right|^2\right]=0 .
\end{equation}
By the a priori estimate (\ref{un kn a priori estimate}) on $\left\{K^n\right\}$, $K_T^0 \in L^p$, for each $p \geq 1$. Thus $K^0 \in \mathbb{K}^p$, for each $p \geq 1$.

Finally, we are left to verify that $(\tilde Y^0,U^0, K^0)$ is a solution.  With the help of (\ref{cvg_Y}), (\ref{cvg_f}), (\ref{estimate-U}), (\ref{cvg_K}) and Burkholder-Davis-Gundy inequality, it turns out that
\begin{equation}
\label{verify_solution}
\begin{aligned}
&\mathbb E\left[\left|\tilde Y^0_t-(\xi+\int_t^T f\left(s,\tilde Y^0_s, U^0_s\right) d A_s -\int_t^T \int_E U^0_s(e) q(d s, d e))\right|\right]\\
&\quad\le\lim_{n\to\infty} \mathbb E\left[\left|\tilde Y^0_t-Y^n_t\right|\right]+\lim_{n\to\infty}\mathbb E\left[\left|\left(\int_t^T f\left(s,\tilde Y^0_s, U^0_s\right) d A_s-\int_t^T \int_E U^0_s(e) q(d s, d e)+K_T^0-K_t^0\right)\right.\right.\\
&\quad\quad\left.\left.-\left(\int_t^T f^n\left(s,Y^n_s, U^n_s\right) d A_s -\int_t^T \int_E U^n_s(e) q(d s, d e)+K_T^n-K_t^n\right)\right|\right]\\
&\quad= 0.
\end{aligned}
\end{equation}

Then, we find a solution $(\tilde Y^0, U^0, K^0)\in \mathcal E\times H_{\nu}^{2,p}\times \mathbb K^p$ to RBSDE($\xi, f, L$) with bounded $(\xi, L)$, for each $p\ge 1$, and with the help of assumption (H3), Corollary \ref{sub mart |y|} and Proposition \ref{priori estimate on U and K}, the  a priori estimates (\ref{|y| bound}) and (\ref{a pri eq on U,K}) hold for $(Y^0, U^0, K^0)$.

The proof is end.

\end{proof}

\section{Existence of convex / concave quadratic exponential RBSDEs with unbounded terminal and obstacle}
\label{section existence for unbdd}

Now we turn to prove the existence for RBSDE(\ref{reflected BSDE}) with unbounded terminal and obstacle.  The main theorem reads as follows.
\begin{thm}
\label{thm Quadratic unbounded}
Assume that assumptions (H1') and (H2)-(H5) are fulfilled.  Then the RBSDE (\ref{reflected BSDE}) admits  a unique  solution $(Y, U, K) \in \mathcal E\times H^{2,p}\times \mathbb K^p$, for all $p\ge 1$.
\end{thm}
The uniqueness also inherited from Corollary \ref{uniqueness}, we are left to prove existence.  We approximate the solution by the solutions of RBSDEs with bounded terminal and obstacle. Although some parts of the proof are inherited directly from the proof of Theorem \ref{thm Quadratic bounded}, we cope with the component $K$ of the solution and the barrier $L$ in a different manner. Therefore, for the sake of self-containedness, there may be slight overlap with the proof of Theorem \ref{thm Quadratic bounded} when it is necessary.  
\begin{proof}[Proof of existence of RBSDE with unbounded terminal and obstacle]
With a little abuse of notation, in view of Theorem \ref{thm Quadratic bounded}, denote the unique solution to RBSDE($\xi^n,\ f,\ L^n$) by ($Y^n,\ U^n,\ K^n$), where $\xi^n=(\xi\wedge n)\vee -n$, $L^n=(L\wedge n)\vee -n$. Then $(Y^n,U^n, K^n)\in \mathcal E\times H_{\nu}^{2,p}\times \mathbb K^p$ for each $p\ge 1$. Moreover, by Corollary \ref{sub mart |y|}, for each $n\in \mathbb N$
\begin{equation}
\label{sub mart |yn|2}
\exp \left\{p \lambda\left|Y^{n}_t\right|\right\} \leq \mathbb{E}_t\left[\exp \left\{p \lambda e^{\beta A_T}(|\xi|\vee L_*^+)+p\lambda  \int_t^T e^{\beta A_s}\alpha_s d A_s\right\}\right].
\end{equation}
We use the sequence $(Y^n,U^n, K^n)$ to approximate the solution for RBSDE($ \xi,\ f,\ L$). The proof is separated into three steps. In the first two steps,we construct a candidate solution $(Y^0,\ U^0)$, then we construct $K^0$ and  verify the solution in the last step.

\textbf{Step 1: Construction of candidate solution $Y^0$.}

As in the bounded scenario, we first find an a priori estimate of $\left|Y^n-Y^m\right|$, for $m, n \in \mathbb{N}$. For a fixed
$\theta \in(0,1)$, we first show that $\mathbb{P}$-a.s.
$$
\left|Y_t^n-Y_t^m\right| \leq(1-\theta)\left(\left|Y_t^m\right|+\left|Y_t^n\right|\right)+\frac{1-\theta}{\lambda} \ln \left(\sum_{i=1}^3 I_t^{m, n, i}\right), \quad t \in[0, T],
$$
where $\zeta_\theta:=\frac{\lambda e^{\tilde{\beta}\left\|A_T\right\|_{\infty}}}{1-\theta}$, and $I_t^{n, m, i}:=E\left[I_T^{n, m, i} \mid \mathscr{G}_t\right]$ for $i=1,2,3$ such that for a given constant $\varepsilon>0$,
$$
\begin{aligned}
& I_T^{n, m, 1}:=\left(D_T^m+D_T^n\right) \eta_{n, m} \text { with } D_t^m:=\exp \left\{\lambda e^{2 \tilde{\beta}\left\|A_T\right\|_{\infty}} \int_0^t\left(\alpha_s+(\beta+\tilde{\beta})\left|Y_s^m\right|\right) d A_s\right\}; \\
& D_t^n:=\exp \left\{\lambda e^{2 \tilde{\beta}\left\|A_T\right\|_{\infty}} \int_0^t\left(\alpha_s+(\beta+\tilde{\beta})\left|Y_s^n\right|\right) d A_s\right\}, \ t \in[0, T]; \\
& \eta_{n, m}:=\exp \left\{\zeta_\theta e^{\tilde{\beta}\left\|A_T\right\|_{\infty}}\left(\left|\xi^n-\theta \xi^m\right| \vee\left|\xi^m-\theta \xi^n\right|\right)\right\} ; \\
& I_T^{n, m, 2}:=\zeta_\theta \exp \left\{\tilde{\beta}\left\|A_T\right\|_{\infty}+\varepsilon \zeta_\theta e^{\tilde{\beta}\left\|A_T\right\|_{\infty}}\right\}\left(D_T^m+D_T^n\right) \exp \left\{\lambda e^{2 \tilde{\beta}\left\|A_T\right\|_{\infty}}\left(Y_*^m+Y_*^n\right)\right\}\left(K_T^m+K_T^n\right) ; \\
& \Upsilon_{n, m}:=\exp \left\{\zeta_\theta e^{\tilde{\beta}\left\|A_T\right\|_{\infty}}\left(Y_*^n+Y_*^m\right)\right\} ; \\
& I_T^{n, m, 3}:={\zeta_\theta} e^{\tilde{\beta}\left\|A_T\right\|_{\infty}}\left(D_T^m+D_T^n\right) \Upsilon_{n, m}\left(1_{\{L^+_*>(m\wedge n)\}}+1_{\{L_*^->(m\wedge n)\}}\right)\left(K_T^m+K_T^n\right) .
\end{aligned}
$$
Set $\tilde{Y}^{(n, m)}:=Y^n-\theta Y^m, \tilde{U}^{(n, m)}:=U^n-\theta U^m$ and define two processes
$$
a_t^{(n, m)}:=\mathbf{1}_{\left\{\tilde{Y}_t^{(n, m)} \neq 0\right\}} \frac{f\left(t, Y_t^n, U_t^n\right)-f\left(t, \theta Y_t^m, U_t^n\right)}{\tilde{Y}_t^{(n, m)}}-\tilde{\beta} \mathbf{1}_{\left\{\tilde{Y}_t^{(n, m)}=0\right\}}, \quad \tilde{A}_t^{(n, m)}:=\int_0^t a_s^{(n, m)} d A_s, \quad t \in[0, T] .
$$
Similar as section \ref{section existence for bdd}, for a given $t\in [0,T]$, consider a sequence of stopping times $\left(T^t_K\right)_{K \geq 0}$ defined by:
\begin{equation}
\label{def-TKt2}
T_K^t=\inf \left\{s>t, \mathbb{E}\left[\exp \left(e^ {\beta 
A_T}\lambda(\left|\xi\right|\vee L_*^+)+\int_0^T \lambda e^{\beta A_s}\alpha_s d A_s\right) \mid \mathcal{G}_s\right]>e^{\lambda K}\right\}.
\end{equation}
The sequence $\left(T^t_K\right)_{K \geq 0}$ converges to infinity when $K$ goes to infinity. Moreover,  from (\ref{|y| bound}), we find that the sequence $\{Y_{\cdot \wedge T_K^t}^{n}\}_{n\ge 1}$ is uniformly bounded on $(t,T]$, in precise, $\sup_n\|Y^{n}_{\cdot\wedge T_K^t}\|_{S^{\infty}(t,T]}\le K$. Let  $U_s^{n, K}=U_s^{n} 1_{s<T_K^t}$.  Then, in view of \cite[Corollary 1]{Morlais_2009}, $\mathbb P-a.s.$, $|U_s^{n,K}|\le 2K$ on $(t,T]$. 
Applying Itô's formula to the process $\Gamma_s^{n, m}:=\exp \left\{\zeta_\theta e^{\tilde{A}_s^{(n, m)}} \tilde{Y}_s^{(n, m)}\right\}$ on $s \in[t, T_K^t\wedge T]$ yields that
\begin{equation}
\label{gamma ito2}
\begin{aligned}
\Gamma_t^{n, m}= & \Gamma_{T_K^t\wedge T}^{n, m}+\int_t^{T_K^t\wedge T} \Gamma_{s^-}^{n,m}\left[\zeta_\theta e^{\tilde{A}_s^{n, m}} F_s^{n, m}-\int_E\left(e^{\zeta_\theta e^{\tilde{A}_s^{n, m}} \tilde{U}_s^{n, m}}-\zeta_\theta e^{\tilde{A}_s^{n, m}} \tilde{U}_s^{n, m}-1\right) \phi_s(d e)\right] d A_s \\
& -\int_t^{T_K^t\wedge T} \int_E\Gamma_{s^-}^{n,m}\left(e^{\zeta_\theta e^{\tilde{A}_s^{n, m}} \tilde{U}_s^{n, m}}-1\right) q(d s d e)+\zeta_\theta \int_t^{T_K^t\wedge T} \Gamma_{s^{-}} e^{\tilde{A}_s^{n, m}}\left(d K_s^{(n)}-\theta d K_s^{(m)}\right) \\
:= & \Gamma_{T_K^t\wedge T}^{n, m}+\int_t^{T_K^t\wedge T} G_s^{n, m} d A_s-\int_t^{T_K^t\wedge T} \int_E Q_s^{n, m} q(d s d e)+\zeta_\theta \int_t^{T_K^t\wedge T} \Gamma_{s^{-}}^{n, m} e^{\tilde{A}_s^{(n, m)}}\left(d K_s^{(n)}-\theta d K_s^{(m)}\right),
\end{aligned}
\end{equation}
where
\begin{equation}
\label{F and G}
\begin{gathered}
F_s^{n, m}=f\left(s, Y_s^n, U_s^n\right)-\theta f\left(s, Y_s^m, U_s^m\right)-a_s^{(n, m)} \tilde{Y}_s^{(n, m)}, \\
G_s^{n, m}=\zeta_\theta \Gamma_s^{n, m} e^{\tilde{A}_s^{(n, m)}}\left(f\left(s, Y_s^n, U_s^n\right)-\theta f\left(s, Y_s^m, U_s^m\right)-a_s^{(n, m)} \tilde{Y}_s^{(n, m)}\right)-\Gamma_s^{n, m} j_1\left(\zeta_\theta e^{\tilde{A}_s^{(n, m)}} \tilde{U}_s^{n, m}\right) .
\end{gathered}
\end{equation}
Similarly to (\ref{cvx estimate}), (H4)(b) and the convexity of $f$ in $u$ show that $d t \otimes d \mathbb{P}$-a.e.
$$
f\left(s, Y_s^m, U_s^n\right) \leq \theta f\left(s, Y_s^m, U_s^m\right)+(1-\theta)\left(\alpha_s+\beta\left|Y_s^m\right|\right)+\frac{1-\theta}{\lambda} j_\lambda\left(\frac{\tilde{U}_s^{n, m}}{1-\theta}\right),
$$
which together with (\ref{F and G}) implies that $d t \otimes d \mathbb{P}$-a.e.,
$$
\begin{aligned}
G_s^{n, m} & =\zeta_\theta \Gamma_s^{n, m} e^{\tilde{A}_s^{(n, m)}}\left(f\left(s, \theta Y_s^m, U_s^n\right)-\theta f\left(s, Y_s^m, U_s^m\right)\right)-\Gamma_s^{n, m} j_1\left(\zeta_\theta e^{\tilde{A}_s^{(n, m)}} \tilde{U}_s^{n, m}\right) \\
& \leq \lambda e^{2 \tilde{\beta}\left\|A_T\right\|_{\infty}} \Gamma_s^{n, m}\left(\alpha_s+(\beta+\tilde{\beta})\left|Y_s^m\right|\right).
\end{aligned}
$$
It follows from integration by parts that
\begin{equation}
\label{ito002}
\begin{aligned}
\Gamma_t^{n, m} \leq & D_t^m \Gamma_t^{n, m} \\
\leq & D_{T_K^t\wedge T}^m \Gamma_{T_K^t\wedge T}^{n, m}+\zeta_\theta \int_t^{T_K^t\wedge T} D_s^m \Gamma_{s^{-}}^{ n, m} e^{\tilde{A}_s^{n, m}} d K_s^n-\zeta_\theta \int_t^{T_K^t\wedge T} \int_E D_s^m \Gamma_{s^{-}}^{n, m}\left(e^{\zeta_\theta e^{\tilde{A}_s^{n, m}} \tilde{U}_s^{n, m}}-1\right) q(d s d e) \\
\leq & D_T^m\eta^{n,m}_{T_K^t\wedge T}+\zeta_\theta e^{\tilde{\beta}\left\|A_T\right\|_{\infty}} D_T^m \int_0^T \Gamma_{s^{-}}^{n, m} d K_s^n-\zeta_\theta \int_t^T \int_E D_s^m \Gamma_{s^{-}}^{n, m}\left(e^{\zeta_\theta e^{\tilde{A}_s^{n, m}} \tilde{U}_s^{n, m}}-1\right) q(d s d e), \quad t \in[0, T] .
\end{aligned}
\end{equation}
where
\begin{equation}
\label{etamnlim}
\eta^{n,m}_{T_K^t\wedge T}=\exp \left\{\zeta_\theta e^{\tilde\beta\|A_T\|}\left(|Y_{T_K^t\wedge T}^n-\theta Y_{T_K^t\wedge T}^m|\vee|Y_{T_K^t\wedge T}^m-\theta Y_{T_K^t\wedge T}^n|\right)\right\}\stackrel{K\to\infty}{\longrightarrow}\eta^{n,m},
\end{equation}
for each $n,m\in\mathbb N$.
We deal with the right-hand side of (\ref{ito002}) term by term. Firstly, the flat-off condition of $\left(Y^n, U^n, K^n\right)$ implies that
\begin{equation}
\label{gamma dk2}
\begin{aligned}
& \int_0^T \Gamma_s^{n, m} d K_s^n=\int_0^T \mathbf{1}_{\left\{Y_{s^{-}}^n=L_s^n\right\}} \Gamma_s^{n, m} d K_s^n=\int_0^T \mathbf{1}_{\left\{Y_{s^{-}}^n=L_s^n \leq L_s^m+\varepsilon\right\}} \Gamma_s^{n, m} d K_s^n+\int_0^T \mathbf{1}_{\left\{Y_{s^{-}}^n=L_s^n>L_s^m+\varepsilon\right\}} \Gamma_s^{n, m} d K_s^n \\
& \leq \int_0^T \mathbf{1}_{\left\{Y_{s^{-}}^n \leq Y_{s^{-}}^m+\varepsilon\right\}} \exp \left\{\lambda e^{2 \tilde{\beta}\left\|A_T\right\|_{\infty}}\left|Y_s^m\right|+\varepsilon \zeta_\theta e^{\tilde{\beta}\left\|A_T\right\|_{\infty}}\right\} d K_s^n+\Upsilon_{n, m} \int_0^T \mathbf{1}_{\left\{\left|L_s^n-L_s^m\right|>\varepsilon\right\}} d K_s^n \\
& \leq \exp \left\{\lambda e^{2 \tilde{\beta}\left\|A_T\right\|_{\infty}} Y_*^m+\varepsilon \zeta_\theta e^{\tilde{\beta}\left\|A_T\right\|_{\infty}}\right\} K_T^n+ \Upsilon_{n, m}\int_0^T(1_{\{\sup_{s\in[0,T]}L_s>(m\wedge n)\}}+1_{\{\inf_{s\in[0,T]}L_s<-(m\wedge n)\}}) dK_s^n\\
&\leq \exp \left\{\lambda e^{2 \tilde{\beta}\left\|A_T\right\|_{\infty}} Y_*^m+\varepsilon \zeta_\theta e^{\tilde{\beta}\left\|A_T\right\|_{\infty}}\right\} K_T^n+ \Upsilon_{n, m}(1_{\{L_*^+>(m\wedge n)\}}+1_{\{L_*^{-}>(m\wedge n)\}})K_T^n, \quad \mathbb{P} \text {-a.s. }
\end{aligned}
\end{equation}
We also note that
For each $p \in(1, \infty)$, Proposition \ref{priori estimate on U and K} and (\ref{|y| bound}) imply that
{\color{red}
\begin{equation}
\label{uniform bounded Y U 2 }
\sup _{n \in \mathbb{N}}\mathbb E\left[e^{p \lambda Y_*^{n}}+\left(\int_0^T\int_E\left|U_s^{n}\right|^2\phi_t(de) d A_s\right)^p+(K^n_T)^p\right]  \leq C_p \Xi\left(72 p \lambda(1+\beta\|A_T\|_\infty),\ \alpha,\ \beta\right),
\end{equation}
where $\Xi$ defined in Corollary \ref{sub mart |y|} denotes the uniform bound.  
Thus, it follows that
\begin{equation}
\label{ess-eta-nm2}
\mathbb E\left[(\eta^{n,m})^p\right] \leq\mathbb E\left[e^{2p \zeta_\theta e^{\tilde \beta \|A_T\|}\left|\xi\right|}\right]\le\mathbb E\left[e^{\frac{2p \lambda}{1-\theta} e^{2\tilde\beta \|A_T\|}\left|\xi\right|}\right]  \leq \Xi\left(\frac{2p\lambda}{1-\theta} e^{2\tilde\beta \|A_T\|},\ \alpha,\ \beta\right).
\end{equation}                                                 }
\begin{equation}
\label{ess002}
\begin{aligned}
\mathbb E\left[\Upsilon_{n,m}^p\right] & \leq \frac{1}{2}\mathbb E\left[e^{2 p \zeta_\theta e^{\tilde\beta \|A_T\|} Y_*^n}+e^{2 p \zeta_\theta e^{\tilde\beta \|A_T\|} Y_*^m}\right] \leq {\Xi}\left(\frac{2 p\lambda}{1-\theta} e^{2\tilde \beta \|A_T\|},\ \alpha,\ \beta\right).\\
\end{aligned}
\end{equation}  

Moreover,
$$
D_T^m \leq \exp \left\{\lambda e^{2 \tilde{\beta}\left\|A_T\right\|_{\infty}} \int_0^T \alpha_s d A_s\right\} \exp \left\{\lambda(\beta+\tilde{\beta})\left\|A_T\right\|_{\infty} e^{2 \tilde{\beta}\left\|A_T\right\|_{\infty}} Y_*^m\right\},
$$
and then $D_T^m \in \mathbb{L}^p\left(\mathscr{G}_T\right)$ with the help of (\ref{uniform bounded Y U 2 }). Thus, random variables $I_T^{n, m, i}, i=1,2,3$ are integrable by Young's inequality and (\ref{uniform bounded Y U 2 })-(\ref{ess-eta-nm2}). 
In addition, $\eta^{n,m}_{T_K^t\wedge T}\le \exp\{\frac{\lambda e^{2\tilde\beta\|A_T\|_\infty}}{1-\theta}(Y^n_*+Y^m_*)\}$,  in which the right hand side is also integrable.

Note that $\int_t^{T_{K}^t\wedge T}\int_E D^m_s \Gamma_{s^-}^{n,m}(e^{\zeta_\theta e^{\tilde A^{n,m}_s}\tilde U^{n,m}_s}-1)q(dsde)$ is a true martingale due to the boundedness of the integrand. Taking $E\left[\cdot \mid \mathscr{G}_t\right]$ in (\ref{ito002}) and in view of (\ref{gamma dk2}), it turns out that $\Gamma_t^{n, m} \leq \mathbb E_t[D_T^m\eta^{n,m}_{T^t_K\wedge T}]+\sum_{i=2}^3 I_t^{n, m, i}, \ \mathbb{P}$-a.s.
Letting $K\to\infty$, by dominated convergence and (\ref{etamnlim}), it holds
$$
\Gamma_t^{n, m} \leq \mathbb E_t[D_T^m\eta^{n,m}]+\sum_{i=2}^3 I_t^{n, m, i}\le \sum_{i=1}^3 I_t^{n, m, i}, \ \mathbb{P}-a.s.
$$
It then follows that
$$
Y_t^n-\theta Y_t^m \leq \frac{1-\theta}{\lambda} e^{-\tilde{\beta}\left\|A_T\right\|_{\infty}-\tilde{A}_t^{n, m}} \ln \left(\sum_{i=1}^3 I_t^{n, m, i}\right) \leq \frac{1-\theta}{\lambda} \ln \left(\sum_{i=1}^3 I_t^{n, m, i}\right), \quad \mathbb{P} \text {-a.s. }
$$
which implies
$$
Y_t^n-Y_t^m \leq(1-\theta)\left|Y_t^m\right|+\frac{1-\theta}{\lambda} \ln \left(\sum_{i=1}^3 I_t^{n, m, i}\right) \leq(1-\theta)\left(\left|Y_t^m\right|+\left|Y_t^n\right|\right)+\frac{1-\theta}{\lambda} \ln \left(\sum_{i=1}^3 I_t^{n, m, i}\right), \quad \mathbb{P} \text {-a.s. }
$$
Exchanging the role of $Y^m$ and $Y^n$, we deduce the other side of the inequality.
Hence,
\begin{equation}
\label{ess0502}
|Y_t^m-Y_t^n| \leq(1-\theta)(\left|Y_t^m\right|+\left|Y_t^n\right|)+\frac{1-\theta}{\lambda} \ln \left(\sum_{i=1}^3 I_t^{m,n, i}\right), \quad \mathbb P \text {-a.s. }
\end{equation}
Based on (\ref{ess0502}), we apply Doob’s martingale inequality and Jensen's inequality to obtain, for $N\in\mathbb N$,
\begin{equation}
\label{y s1 cvg2}
\begin{aligned}
& \sup _{m, n \geq N} \mathbb{E}\left[\sup _{t \in[0, T]}\left|Y_t^n-Y_t^m\right|\right] \leq \sup _{m, n \geq N} \mathbb E\left[(1-\theta)\left(Y_*^m+Y_*^n\right)\right]+\sup _{m, n \geq N} \mathbb{E}\left[\frac{1-\theta}{\lambda} \ln \left(\sum_{i=1}^3 I_*^{n, m, i}\right)\right] \\
& \leq \sup _{m, n \geq N} \mathbb E\left[(1-\theta)\left(Y_*^m+Y_*^n\right)\right]+\frac{1-\theta}{\lambda} \sup _{m, n \geq N} \ln \left(\sum_{i=1}^3 \mathbb{E}\left[I_*^{n, m, i}\right]\right) \\
& \leq \sup _{m, n} \frac{1-\theta}{\lambda} \mathbb{E}\left[e^{\lambda Y_*^m}+e^{\lambda Y_*^n}\right]+\frac{1-\theta}{\lambda} \ln \left(2\sup _{m, n\ge N} \left(\mathbb{E}\left[\left(D_T^m+D_T^n\right)^2\exp\{\frac{2\lambda}{1-\theta} e^{2\tilde\beta \|A_T\|_\infty}(|\xi^m-\theta\xi^n|\vee|\xi^n-\theta\xi^m|)\}\right]\right)^{1/2}\right. \\
&\quad \quad\left.+2\sup _{m, n}\left( \mathbb{E}\left[\zeta_\theta^2 \exp \left\{2\tilde{\beta}\left\|A_T\right\|_{\infty}+2\varepsilon \zeta_\theta e^{\tilde{\beta}\left\|A_T\right\|_{\infty}} \right\}\left(D_T^m+D_T^n\right)^2 \exp \left\{2\lambda e^{2 \tilde{\beta}\left\|A_T\right\|_{\infty}}\left(Y_*^m+Y_*^n\right)\right\}\left(K_T^m+K_T^n\right)^2\right]\right)^{1/2}\right.\\
&\quad\quad\left. +2\sup _{m, n \geq N} \left(\mathbb{E}\left[\zeta_\theta^2 e^{2\tilde{\beta}\left\|A_T\right\|_{\infty}}\left(D_T^m+D_T^n\right)^2 \Upsilon_{n, m}^2\left(1_{\{L_*^+>(m\wedge n)\}}+1_{\{L_*^{-}>(m\wedge n)\}}\right)^2\left(K_T^m+K_T^n\right)^2\right]\right)^{1/2}\right)\\
& \leq \sup _{m, n} \frac{1-\theta}{\lambda} \mathbb{E}\left[e^{\lambda Y_*^m}+e^{\lambda Y_*^n}\right]\\
&\quad\quad+\frac{1-\theta}{\lambda} \ln \left(2\sup _{m, n} \left(\mathbb{E}\left[\left(D_T^m+D_T^n\right)^4\right]\right)^{1/4}\sup _{m, n\ge N}\left(\mathbb E\left[\exp\{\frac{4\lambda}{1-\theta} e^{2\tilde\beta \|A_T\|_\infty}(|\xi^m-\theta\xi^n|\vee|\xi^n-\theta\xi^m|)\}\right]\right)^{1/4}\right. \\
&\quad \quad\left.+2\sup _{m, n}\left( \mathbb{E}\left[\zeta_\theta^2 \exp \left\{2\tilde{\beta}\left\|A_T\right\|_{\infty}+2\varepsilon \zeta_\theta e^{\tilde{\beta}\left\|A_T\right\|_{\infty}} \right\}\left(D_T^m+D_T^n\right)^2 \exp \left\{2\lambda e^{2 \tilde{\beta}\left\|A_T\right\|_{\infty}}\left(Y_*^m+Y_*^n\right)\right\}\left(K_T^m+K_T^n\right)^2\right]\right)^{1/2}\right.\\
&\quad\quad\left. +2\left(\sup _{m, n} \left(\mathbb{E}\left[\zeta_\theta^4 e^{4\tilde{\beta}\left\|A_T\right\|_{\infty}}\left(D_T^m+D_T^n\right)^4 \Upsilon_{n, m}^4\left(K_T^m+K_T^n\right)^4\right]\right)^{1/4}\right)\left(\mathbb E\left[\left(1_{\{L_*^+>N\}}+1_{\{L_*^{-}>N\}}\right)^4\right]\right)^{1/4}\right),
\end{aligned}
\end{equation}
where we make use of H\"older's inequality in the last inequality. Thus, in view of $\lim_{n,m \rightarrow \infty}\eta^{n,m}=\exp\{\lambda e^{2\tilde \beta\|A_T\|_\infty}|\xi|\}$ and dominated convergence theorem, first letting $\varepsilon\to 0$, then $N\to\infty$ and finally $\theta\to 1$, we obtain 
$$
\lim_{N\to\infty}\sup _{m, n \geq N} \mathbb{E}\left[\sup _{t \in[0, T]}\left|Y_t^n-Y_t^m\right|\right]=0,
$$
which implies that the sequence $Y^n$ converges in $S^1$.
 Then, up to a subsequence, there exists a process $Y^0$ such that 
\begin{equation}
\label{ycvgunif}
 \lim_{n\to\infty}\sup_{t\in[0,T]}|Y^n_t-Y^0_t|=0.
 \end{equation}
 Then, $Y^0$ is c\`adl\`ag. With the help of (\ref{sub mart |yn|2}) and Doob's maximum inequality, $ Y^0\in\mathcal E$.

Fix $p \in[1, \infty)$. Inherited from the proof of (\ref{cvg_Y}), since $\mathbb E\left[\exp \left\{2 p \lambda \cdot \sup _{t \in[0, T]}\left|Y_t^n- Y_t^0\right|\right\}\right] \leq \frac{1}{2} E\left[e^{4 p \lambda Y_*^n}+e^{4 p \lambda  Y_*^0}\right] \leq {\Xi}(4 p\lambda, \alpha,  \beta)$ holds for any $n\in\mathbb N$, it turns out that $\left\{\exp \left\{p \lambda \cdot \sup _{t \in[0, T]}\left|Y_t^n-\tilde Y_t^0\right|\right\}\right\}_{n \in\mathbb N}$ is a uniformly integrable sequence in $\mathbb{L}^1\left(\mathcal{G}_T\right)$. Then it follows that $\lim _{n \rightarrow \infty} \mathbb E\left[\exp \left\{p \lambda \cdot \sup _{t \in[0, T]}\left|Y_t^n- Y_t^0\right|\right\}\right]=1$, which in particular implies that
\begin{equation}
\label{cvg_Y2}
\lim _{n \rightarrow \infty}\mathbb E\left[\sup _{t \in[0, T]}\left|Y_t^n- Y_t^0\right|^q\right]=0, \quad \forall q \in[1, \infty).
\end{equation}

\textbf{Step 2: Construction of candidate solution $U^0$.}

For any $m,\ n \in \mathbb{N}$, applying Itô's formula to the process $\left|Y^n-Y^m\right|^2$, we can deduce that
$$
\begin{aligned}
& \int_0^T \int_E\left|U_s^n(e)-U_s^m(e)\right|^2 \phi_s(d e) d A_s \\
& \left|\xi^n-\xi^m\right|^2-\left|Y_0^n-Y_0^m\right|^2+2 \int_0^T\left(Y_{s^{-}}^n-Y_{s^{-}}^m\right)\left(f\left(s, Y_s^n, U_s^n\right)-f\left(s, Y_s^m, U_s^m\right)\right) d A_s \\
& +2 \int_0^T\left(Y_{s^{-}}^n-Y_{s^{-}}^m\right)\left(d K_s^n-d K_s^m\right)-2 \int_0^T \int_E\left(\left(Y_{s^{-}}^n-Y_{s^{-}}^m\right)\left(U_s^n-U_s^m\right)+\left|U_s^n-U_s^m\right|^2\right) q(d s d e) \\
\leq & 2 \sup _{t \in[0, T]}\left|Y_t^n-Y_t^m\right|\left(2 \int_0^T \alpha_t d A_t+\beta\left\|A_T\right\|_{\infty}\left(Y_*^n+Y_*^m\right)+\frac{1}{\lambda} \int_0^T\left(j_\lambda\left(U_s^m\right)+j_\lambda\left(U_s^n\right)+j_\lambda\left(-U_s^m\right)+j_\lambda\left(-U_s^n\right)\right) d A_s+K_T^n+K_T^m\right) \\
& +\sup _{t \in[0, T]}\left|Y_t^n-Y_t^m\right|^2-2 \int_0^T \int_E\left(\left(Y_{s^{-}}^n-Y_{s^{-}}^m\right)\left(U_s^n-U_s^m\right)+\left|U_s^n-U_s^m\right|^2\right) q(d s d e), \quad \mathbb{P} \text {-a.s. }
\end{aligned}
$$
The last term is a martingale due to the integrability condition of $\left\{Y^n\right\}$ and $\left\{U^n\right\}$. Then taking expectation, in
views of Hölder's inequality, (\ref{e^U}), (\ref{e^-U}) and (\ref{sub mart |yn|2}), it deduces that there exists a universal constant $c>0$ such that
$$
\begin{aligned}
&\mathbb E\left[\left(\int_0^T \int_E\left|U_s^n-U_s^m\right|^2 \phi_s(e) d A_s\right)\right] \leq\mathbb E\left[\sup _{t \in[0, T]}\left|Y_t^n-Y_t^m\right|^2\right]+ \\
& +c\left\{\mathbb E\left[\sup _{t \in[0, T]}\left|Y_t^n-Y_t^m\right|^2\right]\right\}^{\frac{1}{2}}\left\{\sup _{m \in \mathbb{N}} \mathbb E\left[e^{2 \lambda Y_*^m}+\int_0^T \int_E\left(e^{\lambda U_s^m}-1\right)^2 \phi_s(e) d A_s+\int_0^T \int_E\left(e^{-\lambda U_s^m}-1\right)^2 \phi_s(e) d A_s+\left(K_T^m\right)^2\right]\right\}^{\frac{1}{2}} \\
& \leq\mathbb E\left[\sup _{t \in[0, T]}\left|Y_t^n-Y_t^m\right|^2\right]+c\left\{\mathbb E\left[\sup _{t \in[0, T]}\left|Y_t^n-Y_t^m\right|^2\right]\right\}^{\frac{1}{2}} .
\end{aligned}
$$
Hence, it follows that
$$
\lim _{N \rightarrow \infty} \sup _{m, n \geq N} \mathbb E\left[\left(\int_0^T \int_E\left|U_s^n(e)-U_s^m(e)\right|^2 \phi_s(d e) d A_s\right)\right]=0 .
$$
Thus, $\left\{U^n\right\}$ is a Cauchy sequence in $H_{\nu}^{2,2}$, which implies that there exists a $U^0 \in H_{\nu}^{2,2}$ such that
$$
\lim _{n \rightarrow \infty} \mathbb E\left[\left(\int_0^T \int_E\left|U_s^n(e)-U_s^0(e)\right|^2 \phi_s(d e) d A_s\right)\right]=0 .
$$
Thus, up to a subsequence,
\begin{equation}
\label{cvg U2}
\lim _{n \rightarrow \infty}\left(\int_0^T \int_E\left|U_s^n(e)-U_s^0(e)\right|^2 \phi_s(d e) d A_s\right)=0,\ \mathbb{P}-\text { a.s. }
\end{equation}
Then, on the set $\left\{(t, \omega) \in[0, T] \times \Omega: d A_t(\omega) \neq 0\right\}$, it holds that
\begin{equation}
\label{ucvgphi}
\lim _{n \rightarrow \infty}\left(\int_E\left|U_t^n(e)-U_t^0(e)\right|^2 \phi_t(d e)\right)=0.
\end{equation}
Moreover, with the help of Proposition \ref{priori estimate on U and K} and Fatou's Lemma, similar as (\ref{u0}), $U^0 \in H_{\nu}^{2, p}$, for each $p\ge 1$.
Moreover, since $f$ is continuous with respect to $y$ and $u$, in view of (\ref{ycvgunif}) and (\ref{ucvgphi}), for $t \in[0,T]$, on $\{(t,\omega), dA_t(\omega)\not=0\}$,
$$
\lim _{n \rightarrow \infty} f\left(t, \omega, Y_t^n, U_t^n\right)=f\left(t, \omega,  Y_t^0, U_t^0\right).
$$
Thus, by dominated convergence theorem, for each $p\ge 1$
\begin{equation}
\label{cvgf2}
\lim_{n\to\infty}\mathbb{E}\left[\left(\int_0^{T}\left|f\left(t, Y_t^{n},U_t^{n}\right)-f\left(t, Y^0_t, U^0_t\right)\right| d A_t\right)^p\right]=0.
\end{equation}

\textbf{Step 3: Construction of candidate solution $K^0$ and verification of the solution $( Y^0,U^0, K^0)$. } 

For any $n, m \in \mathbb{N}$, it holds $\mathbb{P}$-a.s. that
$$
K_t^n-K_t^m=Y_0^n-Y_0^m-\left(Y_t^n-Y_t^m\right)-\int_0^t\left(f\left(s, Y_s^n, U_s^n\right)-f\left(s, Y_s^m, U_s^m\right)\right) d A_s+\int_0^t \int_E\left(U_s^n-U_s^m\right) q(d s d e), \quad t \in[0, T] .
$$
The Burkholder-Davis-Gundy inequality then implies that there exists a universal constant $c>0$ such that
$$
\begin{aligned}
&\mathbb E\left[\sup _{t \in[0, T]}\left|K_t^n-K_t^m\right|^2\right] \leq c\mathbb E\left[\sup _{t \in[0, T]}\left|Y_t^n-Y_t^m\right|^2\right]+c \mathbb E\left[\left(\int_0^T\left|f\left(s, Y_s^n, U_s^n\right)-f\left(s, Y_s^m, U_s^m\right)\right| d A_s\right)^2\right] \\
& +c\mathbb E\left[\left(\int_0^T \int_E\left|U_s^n-U_s^m\right|^2 \phi_s(e) d A_s\right)\right] \\
& \leq c\mathbb E\left[\sup _{t \in[0, T]}\left|Y_t^n-Y_t^m\right|^2\right]+c \mathbb E\left[\left(\int_0^T\left|f\left(s, Y_s^n, U_s^n\right)-f\left(s, Y_s^0, U_s^0\right)\right|d A_s\right)^2\right] \\
& +c\mathbb E\left[\left(\int_0^T\left|f\left(s, Y_s^m, U_s^m\right)-f\left(s, Y_s^0, U_s^0\right)\right| d A_s\right)^2\right]+c\mathbb E\left[\left(\int_0^T \int_E\left|U_s^n-U_s^m\right|^2 \phi_s(e) d A_s\right)\right].
\end{aligned}
$$
Together with (\ref{cvg_Y2}), (\ref{cvg U2}) and (\ref{cvgf2}), we deduce that:
$$
\lim _{N \rightarrow \infty} \sup _{m, n \geq N} \mathbb E\left[\sup _{t \in[0, T]}\left|K_t^n-K_t^m\right|^2\right]=0 .
$$
Hence there is a $K^0$ in $\mathbb{K}^2$ such that
\begin{equation}
\label{cvg K2}
\lim _{n \rightarrow \infty}\mathbb E\left[\sup _{t \in[0, T]}\left|K_t^n-K_t^0\right|^2\right]=0 .
\end{equation}
By the a priori estimate (\ref{a pri eq on U,K}) on $\left\{K^n\right\}$, $K_T^0 \in L^p$, for each $p \geq 1$. Thus $K^0 \in \mathbb{K}^p$, for each $p \geq 1$.

Finally, we are left to verify that $\left(Y^0, U^0, K^0\right)$ is a solution. First, $Y^n \geq L^n$ and the Skorohod condition are obviously hold by putting a pointwise limit as $n \rightarrow \infty$. Moreover, with the help of (\ref{cvg_Y2}), (\ref{cvg U2}), (\ref{cvgf2}), (\ref{cvg K2}) and Burkholder-Davis-Gundy inequality, it turns out that
$$
\begin{aligned}
& \mathbb{E}\left[\left|Y_t^0-\left(\xi+\int_t^T f\left(s, Y_s^0, U_s^0\right) d A_s+\int_t^T d K_s^0-\int_t^T \int_E U_s^0(e) q(d s, d e)\right)\right|^2\right] \\
& \leq \lim _{n \rightarrow \infty} 2 \mathbb{E}\left[\left|Y_t^0-Y_t^n\right|^2\right] \\
& \quad+\lim _{n \rightarrow \infty} 2 \mathbb{E}\left[\mid\left(\xi+\int_t^T f\left(s, Y_s^0, U_s^0\right) d A_s+\int_t^T d K_s^0-\int_t^T \int_E U_s^0(e) q(d s, d e)\right)\right. \\
& \left.\quad-\left.\left(\xi^n+\int_t^T f\left(s, Y_s^n, U_s^n\right) d A_s+\int_t^T d K_s^n-\int_t^T \int_E U_s^n(e) q(d s, d e)\right)\right|^2\right] \\
& =0 .
\end{aligned}
$$
The proof is complete.

\end{proof}

\begin{remark}
Under bounded terminal and bounded obstacle assumptions,  the existence of the quadratic exponential RBSDE holds well without the assumption of convexity, see Matoussi\cite{matoussi2020generalized}. The authors make use of a different way to approximate the solution from the solutions of a collection of  Lipschitz RBSDEs. However, no uniqueness result is provided in \cite{matoussi2020generalized}. With the help of $\theta$-method, we work out the uniqueness under the additional convexity assumption on the driver $f$. 
\end{remark}

\section{Application}
\label{section appl}
In the section, we present an application of quadratic exponential RBSDE's to the pricing of American options via utility maximization.
 As in \cite{Morlais_2009}, assume that the interest rate is zero in  the financial market. The portfolio  consists of one risk-free asset and one single risky asset, whose price process is denoted by $S$, satisfying
$$
\mathrm{d} S_s=S_{s-}\left(b_s \mathrm{~d} s+\sigma_s \mathrm{~d} W_s+\int_{\mathbb{R} \backslash\{0\}} \beta_s(x) \tilde{N}_p(\mathrm{~d} s, \mathrm{~d} x)\right),
$$
where  $W=\left(W_t\right)_{t \in[0, T]}$ is a standard one dimensional Brownian motion and $N_p(\mathrm{~d} s, \mathrm{~d} x)$ is  the associated counting measure to  a  Poisson type point process $p$  defined on $[0, T] \times \mathbb{R} \backslash\{0\}$ and independent with $W$. The compensator of $N_p$ is
$$
\hat{N}_p(\mathrm{~d} s, \mathrm{~d} x)=n(\mathrm{~d} x) \mathrm{d} s
$$
and the  measure $n(\mathrm{~d} x)$ (also denoted by $n$ ) is, without loss of generality assumed to be a probability measure with $n(\{0\})=0$.

All processes $b, \sigma$ and $\beta$ are assumed to be bounded and predictable and $\beta$ satisfies: $\beta>-1$. This condition ensures that the price process $S$ is positive, which is reasonable. The  market price of risk process is defined as $\theta_s=\sigma_s^{-1}b_s$, which is obviously bounded.

Next, we list notions of trading strategies and self financing portfolio. Assume that all trading strategies  taking values in a compact set denoted by $\mathcal{C}$.  

\begin{definition}[\cite{Morlais_2009}]
A predictable $\mathbb{R}$-valued process $\pi$ is a self-financing trading strategy, if it takes its values in a constraint set $\mathcal{C}$ and if the process $X^{\pi, t, x}$ such that
$$
\forall s \in[t, T], \quad X_s^{\pi, t, x}:=x+\int_t^s \pi_s \frac{\mathrm{d} S_s}{S_{s-}}.
$$
is in the space $\mathcal{H}^2$ of semimartingales (see chapter 4, \cite{Protter_2005}). Such a process $X^\pi:=X^{\pi, t, x}$ stands for the wealth of an agent having strategy $\pi$ and wealth $x$ at time $t$.
\end{definition}

In addition, as \cite[Lemma 1]{Morlais_2009}, under the assumption of compactness of the constraint set $\mathcal{C}$, all constrained strategies satisfy: for each $\tilde\alpha>0$, the collection
$\left\{\exp \left(-\alpha X_\tau^\pi\right), \tau: \mathcal{G}\right.$-stopping time $\}$ is uniformly integrable.

We are at the position to characterize dynamically the value process associated to the exponential utility maximization problem.  In the sequel, we denote by $U_\alpha$ the exponential utility function with parameter $\alpha$, which is defined on $\mathbb{R}$ by: $U_\alpha(\cdot)=-\exp (-\alpha \cdot)$.  The price of an European option with terminal payoff $B$ is deduced from the value function as in \cite{Rouge_2000}.

 The utility maximization problem consists in maximizing the expected value of the exponential utility of the portfolio. More precisely, the value process $V^B$ is defined as:
\begin{equation}
\label{value}
V_t^B(x)=\sup _{\pi \in A_t} \mathbb{E}\left(U_{\tilde\alpha}\left(x+\int_t^T \pi_s \frac{\mathrm{d} S_s}{S_{s-}}-B\right) \mid \mathcal{G}_t\right) .
\end{equation}
Here, $A_t$ denotes the set of admissible strategies starting the wealth process from $t$.  $B$ stands for the contingent claim, which is assumed to be a bounded $\mathcal{G}_t$-measurable random variable and $x$ is a constant standing for the wealth at time $t$.

The value function $V^B$ is associated with a quadratic exponential BSDE stated as follows:
\begin{thm}[\cite{Morlais_2009}]
\label{thm7.2}
For any constant $x$, the expression of $V_t^B(x)$, as defined in (\ref{value}), is
$$
V_t^B(x)=-\exp \left(-\tilde\alpha\left(x-Y_t\right)\right)
$$
where $Y_t$ is the first component of the solution $(Y, Z, U)$ of the BSDE:
$$
Y_t=B+\int_t^T f_s\left( Z_s, U_s\right) d s -\int_t^T Z_s d W_s-\int_t^T \int_{\mathbb{R} \backslash\{0\}} U_s(x) \tilde N_p(d s, d x), \quad 0 \leq t \leq T, \quad \mathbb{P} \text {-a.s. },
$$
and whose generator $f$ is
$$
f(s, z, u)=\inf _{\pi \in \mathcal{C}}\left(\frac{\tilde\alpha}{2}\left|\pi \sigma_s-\left(z+\frac{\theta}{\tilde\alpha}\right)\right|^2+\frac{1}{\tilde\alpha}j_{\tilde\alpha}(s,u-\pi \beta_s)\right)-\theta z-\frac{|\theta|^2}{2 {\tilde\alpha}} .
$$
Moreover, there exists an optimal strategy $\pi^*$ such that: $\pi^* \in A_t$, and satisfying
$$
\pi_s^* \in \arg \min _{\pi \in \mathcal C}\left(\frac{{\tilde\alpha}}{2}\left|\pi \sigma_s-\left(Z_s+\frac{\theta_s}{{\tilde\alpha}}\right)\right|^2+\frac{1}{{\tilde\alpha}}j_{\tilde\alpha}(s,U_s-\pi \beta_s)\right)
$$
\end{thm}
\begin{remark}
Theorem \ref{thm7.2} holds well under the assumption that $\mathbb E[\exp\{p|B|\}]<\infty$ for each $p\ge 1$ instead of bounded terminal $B$. Note that since
$$
R_s^\pi:=-\mathrm{e}^{-{\tilde\alpha} X_s^\pi} \mathrm{e}^{{\tilde\alpha} Y_s},
$$
with the help of both the uniform integrability of $\left(\mathrm{e}^{-{\tilde\alpha} X_\tau^\pi}\right)$ for each ${\tilde\alpha}>0$, where $\tau$ runs over the set of all stopping times and $Y\in \mathcal E$, the supermartingale property of $R^\pi$ holds well.
\end{remark}

As in \cite{Rouge_2000}, the price of contingent claim $B$ defined via utility function reads:
$$
pr_t(B)=\inf\{y\in\mathbb R,V_t^B(y)\ge V_t^0(0)=-1\}=Y_t:=Y_t(T,B).
$$

Next, consider the valuation of an American option with early exercise payoff $\left\{\xi_t, 0 \leqslant t \leqslant T\right\}$ which satisfies the assumptions of the obstacle and $\xi_T=B$. Precisely speaking, the holder has the right to exercise the option at any stopping time $\tau$ between 0 and $T$. If the holder exercises at time $\tau$, then the holder receives the payoff $\xi_\tau$. As in \cite{kobylanski2002reflected}, the forward price process $\left\{Y^A_t\right\}$ of such an option is naturally defined by the right-continuous process
$$
Y_t^A=\operatorname{ess} \sup _{\tau \in \mathcal S'_{t,T}} Y_t\left(\tau, \xi_\tau\right),
$$
where $\mathcal S'_{0,T}$ denotes the collection of $\mathbb G$-stopping times $\tau$ such that $0\le\tau\le T,\ \mathbb P-a.s.$. For any $\tau\in\mathcal S'_{0,T}$, $\mathcal S'_{\tau,T}$ denotes the collection of $\mathbb G$-stopping times $\tilde\tau$ such that $\tau\le\tilde\tau\le T,\ \mathbb P-a.s.$

Applying \cite[Proposition 3.1]{foresta2021optimal}, $Y_t^A=ess sup_{\tau\in \mathcal S'_{t,T}}\mathbb E_t\left[B+\int_t^\tau f(s,Z_s,U_s)ds+\xi_{\tau}1_{\tau<T}+B 1_{\tau=T}\right]$.  Hence, $Y^A$ is the unique solution of the RBSDE:
\begin{equation}
\left\{\begin{array}{l}
Y_t^A=B+\int_t^T f\left(s, Z_s, U_s\right) \mathrm{d} s-\int_t^T Z_s \mathrm{~d} W_s-\int_t^T \int_{\mathbb{R} \backslash\{0\}} U_s(x) \tilde{N}_p(\mathrm{~d} s, \mathrm{~d} x)+\int_t^T d K_s,\\
Y_t^A\ge \xi_t,\\
\int_0^T\left(Y^A_{s^{-}}-\xi_{s}\right) d K_s=0, \quad \mathbb{P} \text {-a.s. };
\end{array}\right.
\end{equation}
where 
$$
f(s, z, u)=\inf _{\pi \in \mathcal{C}}\left(\frac{{\tilde\alpha}}{2}\left|\pi \sigma_s-\left(z+\frac{\theta}{{\tilde\alpha}}\right)\right|^2+\frac{1}{{\tilde\alpha}}j_{\tilde\alpha}(s,u-\pi \beta_s)\right)-\theta z-\frac{|\theta|^2}{2 {\tilde\alpha}},
$$
which is obviously jointly convex in $u$ and $z$,  satisfying the exponential quadratic growth condition in Remark \ref{remark_RBSDEJ} (with $dC_t=dA_t=dt$) and linear bound condition inherited from assumption (H5).

\hspace*{\fill}\\
\hspace*{\fill}\\
\textbf{Acknowledgements}

The authors wish to acknowledge Prof. Ying Hu for his valuable suggestions.

\begin{appendices}
\section{Proofs in section \ref{section 3}}
\label{appexdix A}
\begin{proof}[Proof of Lemma \ref{G-estimation}]
We prove the lemma in the following three cases. 

 1) For $d t \otimes d \mathbb P$-a.e. $(t, \omega) \in\left\{Y_t(\omega) \geq 0\right\}$, applying (\ref{cvx estimate}) with $y=Y_t$, we can deduce from (\ref{G estimate1}) and $|a_t|\le\tilde \beta$ that
$$
\begin{aligned}
\hat{G}_t & \leq \zeta_\theta  e^{\tilde A_t}\left(\theta \mathfrak{F}\left(t, Y_t, \widehat{U}_t\right)-\theta \mathfrak{F}\left(t, \widehat{Y}_t, \widehat{U}_t\right)-a_t \tilde Y_t+(1-\theta)\left(\alpha_t+\beta\left|Y_t\right|\right)\right) \\
& =\zeta_\theta e^{\tilde A_t}\left((\theta-1) a_t  Y_t+(1-\theta)\left(\alpha_t+\beta\left|Y_t\right|\right)\right) \leq \lambda  e^{2 \tilde{\beta} \|A_T\|_\infty} \left(\alpha_t+(\beta+\tilde{\beta})\left|Y_t\right|\right)=\lambda  e^{2 \tilde{\beta} \|A_T\|_\infty} \left(\alpha_t+(\beta+\tilde{\beta}) Y_t^{+}\right) .
\end{aligned}
$$
In the first inequality, we use the fact that:
\begin{equation}
\label{j inequality}
\frac{\zeta_\theta e^{\tilde A_t}(1-\theta)}{\lambda}j_{\lambda}(\frac{\tilde U_t}{1-\theta})-j_1(\zeta_\theta e^{\tilde A}\tilde U_t)\le 0.
\end{equation}
which is equivalent to
\begin{equation}
\label{eq ju}
\frac{\zeta_\theta e^{\tilde A_t}(1-\theta)}{\lambda}\left(e^{\frac{\lambda\tilde U_t}{1-\theta}}-\frac{\lambda\tilde U_t}{1-\theta}-1\right)-\left(e^{\zeta_\theta e^{\tilde A_t}\tilde U_t}-ce^{\tilde A_t}\tilde U_t-1\right)=\frac{ \zeta_\theta e^{\tilde A_t}(1-\theta)}{\lambda}\left(e^{\frac{\lambda\tilde U_t}{1-\theta}}-1\right)-\left(e^{ \zeta_\theta e^{\tilde A_t}\tilde U_t}-1\right)\le 0.
\end{equation}
Indeed, plugging in $\zeta_\theta= \frac{\lambda  e^{\tilde{\beta} \|A_T\|_\infty}}{1-\theta}$, denote, 
$$
g_t(v)=\frac{ \zeta_\theta e^{\tilde A_t}(1-\theta)}{\lambda}\left(e^{\frac{\lambda v}{1-\theta}}-1\right)-\left(e^{ \zeta_\theta e^{\tilde A_t}v}-1\right)=e^{\tilde A_t+\beta\|A_T\|_\infty}(e^{\frac{\lambda v}{1-\theta}}-1)-e^{\frac{\lambda v}{1-\theta}e^{(\tilde A_t+\beta \|A_T\|_\infty)}}+1.
$$
We claim that $g(v)\le 0$, for any $v\in\mathbb R$, which implies (\ref{eq ju}). Notice that $g(0)=0$, 
$$
g_t'(v)=\frac{\lambda}{1-\theta}e^{\tilde A_t+\beta\|A_T\|_\infty}\left(e^{\frac{\lambda v}{1-\theta}}-e^{\frac{\lambda v}{1-\theta}e^{(\tilde A_t+\beta \|A_T\|_\infty)}}\right).
$$
Therefore, $g'(v)>0$ if $v<0$, and $g'(v)\le0$ if $v\ge 0$, which implies $g(v)\le g(0)=0$.

2) For $d t \otimes d \mathbb P$-a.e. $(t, \omega) \in\left\{Y_t(\omega)<0 \leq \widehat{Y}_t(\omega)\right\}$, applying (\ref{cvx estimate}) with $y=0$, we see from (\ref{G estimate1}) and \textcolor{red}{(H4)(b)} that
$$
\begin{aligned}
\hat{G}_t & \leq \zeta_\theta  e^{\tilde A_t}\left(\left|\mathfrak{F}\left(t, Y_t, U_t\right)-\mathfrak{F}\left(t, 0, U_t\right)\right|+\mathfrak{F}\left(t, 0, U_t\right)-\theta \mathfrak{F}\left(t, \widehat{Y}_t, \widehat{U}_t\right)+\tilde{\beta}\left(Y_t-\theta \widehat{Y}_t\right)\right)-j_1(\zeta_\theta e^{\tilde A}\tilde U_t) \\
& \leq \zeta_\theta e^{\tilde A_t}\left(\theta\left|\mathfrak{F}\left(t, 0, \widehat{U}_t\right)-\mathfrak{F}\left(t, \widehat{Y}_t, \widehat{U}_t\right)\right|+(1-\theta) \alpha_t-\tilde{\beta} \theta \widehat{Y}_t\right) \leq \alpha_t \lambda  e^{2 \tilde{\beta} \|A_T\|_\infty}.
\end{aligned}
$$
We use (\ref{j inequality}) again in the second inquality.

3) For $d t \otimes d P$-a.e. $(t, \omega) \in\left\{Y_t(\omega) \vee \widehat{Y}_t(\omega)<0\right\}$, applying (\ref{cvx estimate}) with $y=\widehat{Y}_t$, we see from (\ref{G estimate1}) and \textcolor{red}{(H4)(b)} that
$$
\begin{aligned}
\hat{G}_t & \leq \zeta_\theta e^{\tilde A_t}\left(\mathfrak{F}\left(t, \theta \widehat{Y}_t, U_t\right)-\theta \mathfrak{\mathfrak { F }}\left(t, \widehat{Y}_t, \widehat{U}_t\right)\right)-j_1(\zeta_\theta e^{\tilde A}\tilde U_t) \\
& \leq \zeta_\theta e^{\tilde A_t}\left(\left|\mathfrak{\mathfrak { F }}\left(t, \theta \widehat{Y}_t, U_t\right)-\mathfrak{F}\left(t, \widehat{Y}_t, U_t\right)\right|+\mathfrak{F}\left(t, \widehat{Y}_t, U_t\right)-\theta \mathfrak{F}\left(t, \widehat{Y}_t, \widehat{U}_t\right)\right)-j_1(\zeta_\theta e^{\tilde A}\tilde U_t) \\
& \leq \lambda  e^{2 \tilde{\beta} \|A_T\|_\infty} \left(\alpha_t+(\beta+\tilde{\beta})\left|\widehat{Y}_t\right|\right)=\lambda  e^{2 \tilde{\beta} \|A_T\|_\infty} \left(\alpha_t+(\beta+\tilde{\beta}) \widehat{Y}_t^{-}\right).
\end{aligned}
$$
The last inequality is also deduced by (\ref{j inequality}).
Combining the three cases above, we finish the proof.
\end{proof}

\end{appendices}
\bibliographystyle{abbrv}
\bibliography{RBSDEJ}

\end{document}